%% file: main.tex
\documentclass[11pt,a4paper,reqno]{amsart}
\usepackage[utf8]{inputenc}

\input{header.tex}

\allowdisplaybreaks

\title{Graph-to-local limit for the nonlocal interaction equation}
\author{Antonio Esposito \and Georg Heinze \and André Schlichting} 
\address{A. Esposito -- Mathematical Institute, University of Oxford, Woodstock Road, Oxford, OX2 6GG, United Kingdom.}
\address{G. Heinze -- Research Group "Partial Differential Equations", Weierstrass Institute, Mohrenstrasse 39, 10117 Berlin, Germany}
\address{A. Schlichting -- Institute for Analysis and Numerics, University of Münster, Orléans-Ring 10, 48149 Münster, Germany}
\email{antonio.esposito@maths.ox.ac.uk}
\email{georg.heinze@wias-berlin.de}
\email{a.schlichting@uni-muenster.de}
\date{}

\begin{document}

\begin{abstract}
    We study a class of nonlocal partial differential equations presenting a tensor-mobility, in space, obtained asymptotically from nonlocal dynamics on localising infinite graphs. Our strategy relies on the variational structure of both equations, being a Riemannian and Finslerian gradient flow, respectively. More precisely, we prove that weak solutions of the nonlocal interaction equation on graphs converge to weak solutions of the aforementioned class of nonlocal interaction equation with a tensor-mobility in the Euclidean space. This highlights an interesting property of the graph, being a potential space-discretisation for the equation under study.
\end{abstract}

\keywords{optimal transport on graphs, upwind, tensor-mobility, nonlocal interaction, graph-discretisation, local limit}
\subjclass[2020]{35R02, 35A01, 35A15, 35R06, 05C21}

\maketitle

\tableofcontents

\section{Introduction}

In this manuscript we study the connection between nonlocal dynamics on infinite graphs and the corresponding local counterparts in the underlying Euclidean space. This problem is a natural consequence of the recent interest in evolution equations on graphs, motivated by applications to data science, among others. Graphs are indeed a suitable mathematical structure to classify and represent data, as done in \cite{Belkin02laplacianeigenmaps, GTS16, GTSspectral, JordanNg, KanVemVet04,RoithBungert22} and the references therein. Furthermore, it is worth to mention recent advances in the use of graphs in the context of social dynamics or opinion formation, \cite{McQuade_Piccoli_Pour_M3AS_19,Ayi_Pouradier_JDE21,PouradierDuteil_NHM22}, kinetic networks, \cite{Burger_kinetic_net22}, and synchronization, \cite{Gkogkas_Kuehn_Xu_CMS}. 

In this work we consider the nonlocal interaction equation on graphs, introduced in \cite{EPSS2021}, which is relevant to detecting local concentrations in networks. We will refer to it as nonlocal-nonlocal in view of the nonlocal nature of the graph.  

More precisely, we think of equations describing mass displacement among vertices --- points in $\Rd$ --- connected according to a given weight function, $\eta$ in the following. One of the essential differences with the Euclidean space $\Rd$ is that on the graph the mass is moving rather than the point. More precisely, in $\Rd$ we usually describe movements of particles with an associated mass, whilst on the graph the particle, or vertex, is fixed and the mass is transported. In order to cope with this structural property, one needs to consider suitable interpolating functions so that to be able to describe the flux, hence the dynamics. We refer to \cite{EPSS2021,EspositoPatacchiniSchlichting2023,HPS21} for more details, as well as related works \cite{MaasGradFlowEntropyFiniteMarkov2011, Mielke2011gradflowreaction, ChowFokkerPlanckMarkovGraph2012}. Another important aspect is to deal with a large number of entities, for instance individuals or data; hence it is crucial to consider discrete and continuum models. The setup introduced in~\cite{EPSS2021} allows to consider both descriptions in a unified framework, as follows.

The class of partial differential equations (PDEs) we consider can be specified through three elements: a \emph{nonlocal continuity equation}, an \emph{upwind flux interpolation}, and a \emph{constitutive relation} for a nonlocal velocity. 
The nonlocal continuity equation is concerned with the time-evolution of a probability measure $\rho_t \in \calP(\R^d)$, for~$t\in [0,T]$, where mass located at a vertex $x\in \R^d$ can be nonlocally transported to~$y\in \R^d$ along a channel with capacity, referred to as weight, given by an edge weight function $\eta: \R^d \times \R^d\setminus\{x\neq y\} \to [0,\infty)$. The nonlocal continuity equation on a time interval $[0,T]$ is of the form 
\begin{subequations}\label{eq:intro:NLNL}
\begin{equation}\label{eq:intro:NL-CE}
  \partial_t \rho_t + \bdiv j_t = 0, \qquad\text{with}\quad \bdiv j_t(\dx x) = \int_{\Rd\setminus\!\set*{x}} \eta(x,y) \dx j_t(x,y) ,
\end{equation}
where the flux is a time-dependent antisymmetric measure, $j_t \in \calM(G)$, on the set $G=\set*{ (x,y)\in \Rddiag: \eta(x,y)>0}$, being $\Rddiag\coloneqq\set{(x,y)\in\Rd\times\Rd:x\ne y}$. We will use the shorthand notation $(\bs \rho,\bs j) =((\rho_t)_t,(j_t)_t)\in \NCE_T$ for any solution of~\eqref{eq:intro:NL-CE}, cf.~Definition~\ref{def:nce-flux-form}.

The relation constituting the flux depends on a $\sigma$-finite absolutely continuous measure $\mu\in \calM^+(\R^d)$, as in~\cite{EPSS2021,HPS21, EspositoPatacchiniSchlichting2023}, wherein $\mu$ acts as an abstract notion of vertices of a graph. More precisely,
we associate to a nonlocal time-dependent velocity field $v_t : G\to \R$ the induced flux by using an upwind interpolation as follows
\begin{equation}\label{eq:intro:NL-flux}
	\dx j_t(x,y) = v_t(x,y)_+\dx(\rho\otimes\mu)(x,y)-v_t(x,y)_-\dx(\mu\otimes\rho)(x,y).
\end{equation}
Here, for $a\in \R$, we denote with $a_+ = \max\set*{a,0}$ and $a_- = \max\set*{-a,0}$ the positive and negative part, respectively. Intuitively, the support of $\mu$ defines the underlying set of vertices, i.e. $V = \supp\mu$. In particular, any finite graph can be represented by choosing $\mu=\mu^N=\delta_{x_i}/N$, for $x_1, x_2,\dots,x_N\in\Rd$. 

The last element of the model is the identification of the velocity field in terms of a symmetric interaction potential $K:\R^d\times\R^d \to \R$ and a potential $P:\R^d\to \R$ by
\begin{equation}\label{eq:intro:NL-velocity}
  v_t(x,y) = - \babla K\ast \rho_t(x,y) - \babla P(x,y) ,
\end{equation}
\end{subequations}
where the nonlocal gradient is defined by $\babla f(x,y) \coloneqq f(y)-f(x)$.

System~\eqref{eq:intro:NLNL} was introduced in~\cite{EPSS2021} as a \textit{Finslerian} gradient flow of the interaction energy 
\begin{equation}\label{eq:interaction_energy}
	\calE(\rho)=\frac{1}{2}\iint_{\R^{2d}} K(x,y)\dx\rho(y)\dx\rho(x)+\int_\Rd P(x)\dx\rho(x). 
\end{equation}
Note that the velocity field is given as the nonlocal gradient of the first variation of the energy, that is $v_t = -\babla \calE'(\rho_t)$, where $\calE'(\rho)= K \ast \rho + P$ denotes the variational derivative of $\calE$ and $(K\ast \rho)(x) = \int_\Rd K(x,y)\dx \rho(y)$, for any $x\in\Rd$.

\medskip

An intriguing problem is to understand the limiting behaviour of weak solutions to~\eqref{eq:intro:NLNL} as the graph structure localises, i.e. the range of connection between vertices decreases, while the weight of each connecting edge increases. 
Following a formal argument presented in~\cite[Section 3.5]{EPSS2021}, one expects to approximate weak solutions of the more standard nonlocal interaction equation on~$\Rd$. However, as we shall see, the intrinsic geometry of the graph impacts the limiting gradient structure of the equation.
Accordingly, the main goal of this work is to provide a rigorous proof of the local limit of the system~\eqref{eq:intro:NLNL} along a sequence of edge weight functions $\eta^\eps:\Rddiag\to[0,\infty)$ defined by
\begin{align}\label{eq:intro:etaeps}
	\eta^\eps(x,y) \coloneqq \frac{1}{\eps^{d+2}}\vartheta\bra*{\frac{x+y}{2},\frac{x-y}{\eps}},
\end{align} 
in terms of a reference connectivity $\vartheta:\Rd\times\Rd\!\setminus\!\set{0}\to [0,\infty)$ satisfying the Assumptions~\eqref{theta1}~--~\eqref{theta4} below.

The scaling in~\eqref{eq:intro:etaeps} leads to the local evolution
\begin{equation}\label{eq:intro:NLIE:one}
    \partial_t\rho_t=\operatorname{div}(\rho_t  \bbT(\nabla K*\rho_t+\nabla P)), \tag{$\mathsf{NLIE}_\bbT$} 
\end{equation}
where the \emph{tensor} $\bbT:\R^d \to \R^{d\times d}$ depends on the nonlocal structure encoded through the reference measure~$\mu$ and the connectivity~$\vartheta$.

\eqref{eq:intro:NLIE:one} can be similarly decomposed into three components. First, the  \emph{local continuity equation} on $\R^d$ given as
\begin{subequations}\label{eq:intro:NLIE:sub}
\begin{equation}\label{eq:intro:CE}
	\partial_t \rho_t + \operatorname{div} \hat\jmath_t = 0 ,
\end{equation}
where now $\hat\jmath_t\in \calM(\R^d;\R^d)$ is a vector-valued flux.
Second, a \emph{kinetic relation}, for the flux in terms of a vector field $\hat v_t: \R^d \to \R^d$ encoding the tensor structure of~\eqref{eq:intro:NLIE:one} as
\begin{equation}\label{eq:intro:fluxtensor}
	\hat\jmath_t(\dx x) = \rho_t(\dx x) \bbT(x) \hat v_t(x) = \rho_t(\dx x) \sum_{i,k=1}^d \bbT_{ik}(x) \hat v_{t,k}(x)e_i . 
\end{equation}
Third, a  \emph{constitutive relation} for the velocity linking to the interaction energy~\eqref{eq:interaction_energy} given by
\begin{equation}\label{eq:intro:velocity}
	\hat v_t = - \nabla K \ast \rho_t - \nabla P = -\nabla \calE'(\rho_t) .
\end{equation}
\end{subequations}
We provide a rigorous proof for the convergence of system~\eqref{eq:intro:NLNL} with $\eta^\eps$ given through~\eqref{eq:intro:etaeps} to the system~\eqref{eq:intro:NLIE:sub}, in case of $C^1$ interaction kernels. This result is somewhat sharp since for attractive pointy potentials one does not expect convergence of weak solutions, as pointed out in~\cite[Remark 3.18]{EPSS2021}.

First, we give an heuristic argument and match each of the three elements of the systems~\eqref{eq:intro:NLNL} and~\eqref{eq:intro:NLIE:sub}, separately.
The nonlocal continuity equation~\eqref{eq:intro:NL-CE} can be represented by its local counterpart~\eqref{eq:intro:CE} thanks to a continuous reconstruction for the nonlocal flux. 
More precisely, we denote by $\bdiveps j^\eps$ the nonlocal divergence as defined in~\eqref{eq:intro:NL-CE} with~$\eta$ replaced by $\eta^\eps$ given in~\eqref{eq:intro:etaeps}; see Definition~\ref{def:nl_grad_div}. Next, in Proposition~\ref{prop:jhat}, for any $j^\eps\in \calM(\Rddiag)$ and any $\eps>0$ we construct a local vector-valued flux $\jh^\eps \in \calM(\R^d;\R^d)$ such that $\bdiveps j^\eps = \div \jh^\eps$. Inspired by~\cite{Smirnov1993,Klartag2017}, we use a superposition with \emph{needle} measures defined as one-dimensional Hausdorff measure $\calH^1_{\llbracket x,y\rrbracket}$ restricted to the line-segment ${\llbracket x,y\rrbracket}\subset\Rd$ from $x$ to $y$. The \emph{tentative} definition of~$\jh^\eps$ is, for any Borel set $A$, is then
\begin{equation}\label{eq:intro:reconstruction}
  \jh^\eps[A] :\approx  \frac{1}{2} \iint_{G^{\!\!\:\eps}} \frac{y-x}{\abs{y-x}} \calH^1\bra*{A\cap \llbracket x,y\rrbracket} \eta^\eps(x,y) \dx j^\eps(x,y) .
\end{equation}
However, checking that the measure-valued local divergence of the above measure agrees with the nonlocal divergence of $j^\eps$ asks for special uniform integrability properties of $j^\eps$, which are not necessarily verified in our setting. For this reason, we proceed by exploiting a finite approximation of~\eqref{eq:intro:reconstruction} by replacing $j^\eps$ with an empirical approximation $j^\eps_N$. The actual measure $\jh$ is obtained as a suitable limit of the sequence $\jh^\eps_N$ defined as in~\eqref{eq:intro:reconstruction} with $j^\eps$ replaced by $j^\eps_N$. Since the argument is based on compactness, the uniqueness of such a limit is not achieved nor can we ensure that the limit has indeed the representation~\eqref{eq:intro:reconstruction}. 
For the present application, the mere existence is sufficient.
We point out that the ``reverse'' question on the decomposition of measure-valued divergence vector fields into one-dimensional \emph{needles} is studied in~\cite[Theorem C]{Smirnov1993} and commented on in \cite[Remark 4.4]{Ambrosio2003c}. 

Having established the continuous reconstruction, the constitutive relation for the velocity~\eqref{eq:intro:NL-velocity} can be localized under suitable regularity assumptions on the kernel $K$ and potential $P$. Indeed, for $\abs{x-y}=\calO(\eps)$, neglecting higher order terms in $\eps$, we obtain the identity
\begin{equation}\label{eq:intro:velocityapprox}
	v^\eps_t(x,y) \approx \hat v^\eps_t(x) \cdot (y-x) \quad\text{with}\quad \hat v^\eps_t(x) =- (\nabla K \ast \rho^\eps_t)(x) - \nabla P(x) . 
\end{equation}
The crucial mathematical step consists in the combination of the above two equations through the upwind flux interpolation~\eqref{eq:intro:NL-flux}. Formally, if the measures $\mu$ and~$\rho_t$ have a smooth density, we can do another expansion to arrive for $x\in \R^d$ at the approximate identity
\begin{equation}\label{eq:intro:fluxapprox}
	\jh^\eps_t(x) \approx \rho^\eps_t(x) \bbT^\eps(x) \hat v^\eps_t(x) \text{ with }	\bbT^\eps(x) = \frac12\int_{\Rdx} \!\!(x-y)\otimes(x-y)\,\eta^\eps(x,y)\dx\mu(y). 
\end{equation}
In particular, in this step it is shown that the non-symmetric upwind-based structure can be replaced by symmetric approximate tensors $\bbT^\eps$ with vanishing error as $\eps\to 0$. This symmetrization is generalized to arbitrary $\rho\in\calP(\Rd)$ by introducing a mollifier on an intermediate scale $\eps^\alpha$ for $\alpha>0$ sufficiently small.
The final step consists in identifying the limit of the tensor $\bbT^\eps$, where the specific scaling hypothesis~\eqref{eq:intro:etaeps} provides, after a change of variable, the approximate identity $\bbT^\eps(x) \approx \bbT(x)$ with
\begin{equation}\label{eq:intro:limittensor}
	\bbT(x)= \frac12 \frac{\dx\mu}{\dx\scrL^d}(x)\int_{\Rdzero} w\!\otimes\! w \ \vartheta(x,w)\dx w \tag{$\bbT$}.
\end{equation}
The limit tensor $\bbT$ gives, together with the reconstruction~\eqref{eq:intro:reconstruction}, flux identification~\eqref{eq:intro:fluxapprox} and velocity approximation~\eqref{eq:intro:velocityapprox}, the limit system~\eqref{eq:intro:NLIE:one}.

The heuristic argument above is based on several approximations and smoothness assumptions which are not a priori satisfied by solutions to the systems~\eqref{eq:intro:NLNL} and~\eqref{eq:intro:NLIE:one}. We make it rigorous by using a variational framework, allowing to handle measure-valued solutions. 

\smallskip
An interesting byproduct of our result is the link between Finslerian and Riemannian gradient flows, which is the first result in this direction to the best of our knowledge. More precisely, \eqref{eq:intro:NLNL} is shown to be a gradient flow of the nonlocal interaction energy in the infinite-dimensional Finsler manifold of probability measures endowed with a nonlocal upwind transportation quasi-metric, $\mathcal{T}$, peculiar of the upwind interpolation~\eqref{eq:intro:NL-flux}. Due to the loss of symmetry the underlying structure of $\calP(\Rd)$ does not have the formal Riemannian structure, but Finslerian instead. We refer the reader to~\cite{EPSS2021}~and~\cite{HPS21,HeinzePiSch-2} for further details.

On the other hand, following \cite{Lisini_ESAIM2009}, we establish a chain-rule inequality for the nonlocal interaction energy in a $2$-Wasserstein space defined over $\R^d_{\bbT}$, which is~$\Rd$ endowed with a metric induced by $\bbT^{-1}$. Upon considering the corresponding \emph{Wasserstein scalar product} on the tangent space of $\calP_2(\R^d_{\bbT})$, at some probability measure with bounded second moment, one can notice the underlying Riemannian structure, following~\cite{Otto2001GeometryPME, AmbrosioGigliSavare2008}, thereby making the connection between the weak and variational formulations of \eqref{eq:intro:NLIE:one}.

We stress that not only do we connect the graph and tensorized local gradient structures using the notion of curves of maximal slope for gradient flows after De~Giorgi~\cite{DeGiorgiMarinoTosques80,AmbrosioGigliSavare2008}, but, upon
identifying weak solutions of~\eqref{eq:intro:NLIE:one} with curves of maximal slopes, we also obtain an existence result for~\eqref{eq:intro:NLIE:one} via stability of gradient flows. This is indeed another interesting property of the graph, as it represents a valuable space-discretisation for the PDE under study, working in any dimension, in addition to other methods, e.g. particle approximations and tessellations. Indeed, our result can be also seen as a deterministic approximation of~\eqref{eq:intro:NLIE:one}.

As an outlook, we point out that a quantitative convergence statement might be obtained upon establishing a suitable nonlocal replacement for \emph{weak BV estimates}, as in the numerical literature~\cite[Chapters 5–7]{eymard2000finite}. This would allow to adapt techniques previously developed for determining the rate of convergence for upwind schemes with rough coefficients, for which a measure-valued solution framework is used~\cite{DelarueLagoutiereVauchelet16,SchlichtingSeis2017,SchlichtingSeis2018}. Alternatively, one could directly use the limiting metric structure along the lines of~\cite{LagoutiereSantambrogioTranTien2023}.

Focusing on the quasi-metric space introduced in~\cite{EPSS2021}, natural problems to address are the asymptotic behaviour of the quasi-metric and the Gromov-Hausdorff convergence of the according quasi-metric spaces. First steps in the metric setting are done in~\cite{SlepcevWarren2022} and~\cite{GigliMaas2013,GLADBACH2020204,Trillos_grom_haus,GladbachKopferMaasPortinale2023}, respectively.

Finally, generalizing the present approach to finite-graph approximation for diffusion equations seems possible by using a gradient flow formulation similar to the Scharfetter-Gummel finite volume scheme recently introduced in~\cite{SchlichtingSeis21,HraivoronskaSchlichtingTse2023}.

\subsection{Relation to the literature}

This manuscript is a natural question raised in~\cite{EPSS2021}, where nonlocal dynamics on graphs are considered for the aforementioned upwind interpolation. In~\cite{EspositoPatacchiniSchlichting2023}, the authors deal with a class of continuity equations on graphs with general interpolations, and provide a well-posedness theory exploiting a fixed-point argument. Depending on the interpolation chosen, one can notice structural differences with the more standard Euclidean space. This theory has been extended to the case of co-evolving graphs in~\cite{EspMik23}, namely a scenario including weight functions evolving in time. In \cite{HPS21}, the analysis of \cite{EPSS2021} is extended to two species with cross-interactions as well as nonlinear mobilities and $\alpha$-homogeneous flux-velocity relations for $\alpha > 0$, by generalizing the underlying Finslerian structure and notion of gradient. In \cite{HeinzePiSch-2} the generalized system is explored both in analytical case studies and numerical simulations on finite graphs of varying shapes and connectivity, leading to the observation of various patterns, including mixing of the two species, partial engulfment, or phase separation.

Regarding the limiting equation we mention \cite{Carrillo2011}, where well-posedness for \eqref{eq:intro:NLIE:one} with $\bbT =\Id$ is shown for pointy locally $\lambda$-convex potentials. To this end the authors generalize the theory of \cite{AmbrosioGigliSavare2008} and employ a minimizing movement scheme to show the existence of solutions. Furthermore, in~\cite{Lisini_ESAIM2009}, the author studies a class of equation similar to \eqref{eq:intro:NLIE:one} in view of the presence of a tensor $\bbT$ (not necessarily of the form~\eqref{eq:intro:limittensor}). However, velocity fields considered are of the form $v=-\nabla (F'(u)+V)$, under suitable assumptions on $F$ and $V$. Hence, differently from our case, in \cite{Lisini_ESAIM2009} the author mainly focuses on linear and nonlinear diffusion.

Our variational approach is similar to the one used in \cite{hraivoronska2022diffusive,HraivoronskaSchlichtingTse2023} to study the limiting behaviour of random walks on tessellations in the diffusive limit. 
Starting from the forward Kolmogorov equation on a general family of finite tessellations, the authors show that solutions of the forward Kolmogorov equation
converge to a non-degenerate diffusion process solving an equation of the form \eqref{eq:intro:NLIE:one}, with diffusion instead of the nonlocal interaction velocity field, similar to~\cite{Lisini_ESAIM2009}.
Another difference is that the method used relies on a generalized gradient structure of the forward Kolmogorov equation with respect to the relative entropy --- the so called cosh gradient structure to be precise. This result on tessellations is related to \cite{DisserLiero2015,forkertEvolutionaryGammaConvergence2020}, where a similar problem is considered, though the gradient structure is the so called quadratic gradient structure. The latter manuscripts concern the convergence of discrete optimal transport distances to their continuous counterparts, see also~\cite{GLADBACH2020204,GladbachKopferMaasPortinale2023}. 

In this context, the work~\cite{Trillos_grom_haus} studies Gromov--Hausdorff limit of Wasserstein spaces on point clouds, obtained by constructing a geometric graph similar to our setting (restricted on the torus). The previous manuscript is, indeed, an invitation to the study of stability of evolution PDEs on graphs, though the analysis focuses on the general problem of approximating the underlying metric space rather than the gradient dynamics --- differently from our result. Moreover, we also point out that the upwind interpolation does not fall in the class of interpolating function considered in the aforementioned works. Related to this point, we also mention~\cite{PeletierRossiSavareTse22}, where a direct gradient flow formulation of jump processes is established --- the authors consider driving energy functional containing entropies. The kinetic relations used there are symmetric, hence excluding for instance the upwind interpolation, which is the one we use on graphs.

It is worth to mention the work~\cite{CraigTrillosSlepcev} dealing with dynamics on graphs for data clustering by connecting the mean shift algorithm with spectral clustering at discrete and continuum levels via suitable Fokker--Planck equations on graphs. See also~\cite{cucuringu2019mbo,LauxLelmi2023} for clustering algorithms based on thresholding schemes on graphs, giving also rise to a nonlocal dynamic.

Concerning graph structure, we point out the recent work~\cite{EspositoGvalaniSchlichtingSchmidtchen2021} on a novel interpretation of the aggregation equation (arising in applications to granular media) as the gradient flow of the kinetic energy, rather than the interaction energy. This is possible upon introducing a suitable \textit{nonlocal collision metric}. The underlying state space resembles a graph, indeed, and the PDE under study takes the form of a local-nonlocal continuity equation.

The scaling of the function $\eta$ in~\eqref{eq:intro:etaeps} is chosen such that two derivatives emerge in the local limit. This can be seen as a second order counterpart to the nonlocal first-order calculus developed in~\cite{Du2013}. With a first order scaling many classical results from calculus can be recovered, like for instance the recent divergence theorem in~\cite{HeppKassmann2023}. 
Second order nonlocal calculus with emphasis on questions of regularity and connections to jump processes got also attention with many contributions by Kassmann~\cite{Kassmann2018,FoghemKassmann2022arXiv} and references therein.

We also refer to the book \cite{Du2019} and the references therein for further applications of nonlocal vector calculus with connections to numerical approximations, and various applications, such as nonlocal dynamics of anomalous diffusion and nonlocal peridynamic models of elasticity and fracture mechanics (perydynamic). In the context of numerical schemes for local conservation laws, we mention~\cite{Du2017}, where the authors proposed a class of monotonicity-preserving nonlocal nonlinear conservation laws, in one space dimension. Under suitable assumptions on the kernel, one could interpret the latter class of PDEs as equations on graphs. It is, in our opinion, interesting to explore possible applications of the current manuscript to other nonlocal conservation laws. 

We observe that the graph structure can be also seen as suitable space-discretization, resembling those obtained from tessellations for finite volume schemes. In this regard, there is a natural connection to numerical schemes for gradient flows in the Wasserstein space as also studied in, e.g.,~\cite{CCH15,DisserLiero2015,BailoCarrilloHu2018,Cances2020,Bailo_etal2020}.

\subsection{Structure of the manuscript} First we present the notation we use. In Section~\ref{sec:preliminaries} we explain the setup and recall known results on both the nonlocal-nonlocal interaction equation on graphs and the limiting PDE. We state our main results and give some meaningful examples. The connection between nonlocal and local structure is provided in Section~\ref{sec:graph_and_local_structure}, where we construct the local flux from the nonlocal ones and identify the tensor-mobility. In Section~\ref{sec:graph-to-local-limit} we prove the main results of our manuscript: the graph-to-local limit in terms of curves of maximal slopes and existence of weak solutions for \eqref{eq:intro:NLIE:one}. We integrate the manuscript with an appendix including some additional results, among them the extension of~\cite{EPSS2021} to $\sigma$-finite base measures and a formulation of the main result in terms of EDP-convergence.

\subsection{Notation}

We list here notation used throughout the manuscript.
\begin{itemize}
	\item $a_+\coloneqq\max\{0,a\}$ and $a_-\coloneqq(-a)_+$ denote the positive and negative parts of $a \in \R$, respectively.
    \item Given a set $A$ we set $\bbone_{A}(x)=1$ for $x\in A$ and $\bbone_{A}(x)=0$ for $x\notin A$.
    \item Moduli of continuity are always denoted by $\omega$.
    \item In a metric space $(X,d)$ for $R>0$ we set $B_R(x)\coloneqq \set{y\in X:d(x,y)\le R}$.
    \item In a normed space $(X,\norm{\cdot})$ for $R>0$ we set $B_R\coloneqq B_R(0)$.
	\item $\calB(\R^d)$ is the $\sigma$-algebra of Borel sets of $\Rd$.
    \item $C_\mathrm{b}(A)$ is  the set of bounded continuous functions from $A$ to $\R$. 
	\item $\calM(A)$ is the set of Radon measures on $A \subseteq \R^d$. 
	\item $\Mplus(A)$  is the set of non-negative Radon measures on $A$.
	\item $\mathcal{P}(A)\subset \Mplus(A)$ is the set of Borel probability measures on $A$.
	\item $\calP_{2}(A)\subseteq \calP(A)$ is the set of $\rho\in\calP(A)$ with finite second moment, that is,
	\begin{equation*}%
		m_2(\rho) \coloneqq \int_{A} |x|^2 \dx\rho(x) < \infty. 
	\end{equation*}
	\item $\Rddiag \coloneqq \set{(x,y)\in\R^d\times \R^d : x\ne y}$ is the off-diagonal of $\Rd\times\Rd$.
	\item $\mu\in\Mplus(\R^d)$ denotes the base measure setting the underlying geometry.
	\item $\vartheta:\Rd\!\setminus\!\set*{0}\to[0,\infty)$ is the edge connectivity map, cf. Section \ref{ssec:graph}.
	\item $\eta^\eps \coloneqq \frac{1}{\eps^{d+2}}\vartheta\bra*{\frac{x+y}{2},\frac{x-y}{\eps}} $ is the $\eps$-dependent edge weight function. 
	\item $G^{\!\!\:\eps}\coloneqq \{ (x,y) \in \Rddiag: \eta^\eps(x,y)>0\}$ is the set of edges.
    \item $G_{\!\!\:x}^{\!\!\:\eps}\coloneqq\set{y\in\Rd\!\setminus\!\set*{x}:\eta^\eps(x,y)>0}$ is the set of points connected to $x$.
    \item $\rho\in \calP(\R^d)$ denotes a mass configuration.
    \item $\bs\rho \coloneqq (\rho_t)_{t\in[0,T]}\subset\calP(\Rd)$ denotes a family of mass configurations.
	\item $j\in\calM(\Rddiag)$ denotes a measure-valued nonlocal flux.
    \item $\bs j \coloneqq (j_t)_{t\in[0,T]}\subset\calM(\Rddiag)$ denotes a family of nonlocal fluxes.
    \item $\jh\in\calM(\Rd;\Rd)$ denotes a measure-valued local flux, cf. Proposition \ref{prop:jhat}.
    \item $\bs \jh \coloneqq (\jh_t)_{t\in[0,T]}\subset\calM(\Rd;\Rd)$ denotes a family of local fluxes
    \item $v:\Rddiag\to\R$ denotes a nonlocal velocity field.
    \item $\bs v \coloneqq (v_t)_{t\in[0,T]}$ denotes a family of velocity fields.
	\item $\calA(\mu,\eta;\rho,j)$ stands for the $\mu$ and $\eta$ dependent action density of  $\rho$ and $j$.
	\item $\bs\calA(\mu,\eta;\bs\rho,\bs j)\coloneqq\int_0^T\calA(\mu,\eta;\rho_t,j_t)\dx t$ denotes the action of  $\bs\rho$ and $\bs j$.
	\item $\tA(\mu,\eta;\rho,v)$ stands for the $\mu$ and $\eta$ dependent action density of  $\rho$ and $v$.
	\item $\bs\tA(\mu,\eta;\bs\rho,\bs v)\coloneqq\int_0^T\tA(\mu,\eta;\rho_t,v_t)\dx t$ denotes the action of  $\bs\rho$ and $\bs v$.
    \item $T^\eps_\rho\calP_2(\Rd)\subset\calM(G^{\!\!\:\eps})$ denotes the space of nonlocal tangent fluxes at the configuration $\rho\in\calP_2(\Rd)$, cf.~\eqref{eq:def:T_rho}.
	\item $\widetilde T^\eps_\rho\calP_2(\Rd)$ denotes the space of nonlocal tangent velocities $v:G^{\!\!\:\eps}\to\R$ at the configuration $\rho\in\calP_2(\Rd)$, cf.~\eqref{eq:def:tildeT_rho}.
	\item $\widetilde l^\eps_{\rho}(v)[w]$ stands for the first variation of $\tA(\mu,\eta^\eps;\rho,\cdot)$ at $v\in T^\eps_\rho\calP_2(\Rd)$ in the direction of $w\in T^\eps_\rho\calP_2(\Rd)$ for $\rho\in\calP_2(\Rd)$, cf.~\eqref{eq:def:l}.
    \item $\bbT^\eps(x) \coloneqq \frac12\int_{\Rdx} \!(x-y)\!\otimes\!(x-y)\,\eta^\eps(x,y)\dx\mu(y)$ is the approximate tensor.
    \item $\bbT(x)\coloneqq \frac12 \frac{\dx\mu}{\dx\scrL^d}(x)\int_{\Rdzero} w\!\otimes\! w \ \vartheta(x,w)\dx w$ denotes the limit tensor.
    \item $\calT_{\mu,\eta}$ denotes the nonlocal extended quasi-metric, cf. Definition~\ref{def:T-metric}.
    \item $\AC^2([0,T];(X,d)$ denotes the set of $2$-absolutely continuous curves, which map from the time interval $[0,T]$ into the (quasi-)metric space $(X,d)$.
    \item $\abs{\rho_t'}_{\mu,\eta}$ denotes for the curve $\bs\rho\in\AC^2([0,T];(\mathcal{P}_2(\Rd),\mathcal{T}_{\mu,\eta}))$ the (forward) metric derivative at time $t\in[0,T]$.
	\item $\babla f(x,y)=f(y)-f(x)$ is the nonlocal gradient of a function $f : \R^d \to \R$ 
	\item $\babla \cdot j$ is the nonlocal divergence of a flux $j\in \calM(\Rddiag)$, cf. Definition \ref{def:nl_grad_div}.
	\item $\NCE_T$ denotes the set of solutions to the nonlocal continuity equation on the time intervall $[0,T]$; $\NCE\coloneqq\NCE_1$, cf. Definition~\ref{def:nce-flux-form}.
	\item $\NCE(\varrho_0,\varrho_1)$ is the subset of $(\bs\rho,\bs j)\in\NCE$ which satisfy $\rho_0=\varrho_0$, $\rho_1=\varrho_1$.
	\item $\CE_T$ denotes the set of solutions to the local continuity equation on the time intervall $[0,T]$; $\CE\coloneqq\CE_1$, cf. Definition~\ref{def:ce-flux}. 
	\item $\CE(\varrho_0,\varrho_1)$ is the subset of $(\bs\rho,\bs \jh)\in\CE$ which satisfy $\rho_0=\varrho_0$, $\rho_1=\varrho_1$.
	\item $\calE(\rho)\coloneqq \frac12\iint_{\R^{2d}} K\dx(\rho\otimes\rho)$ denotes the interaction energy of $\rho\in\calP(\Rd)$.
    \item $\frac{\delta \calE}{\delta\rho}$ denotes the variational derivative of $\calE$.
    \item $\calD_{\!\eps}$ is the metric slope of $\calE$ with respect to $\calT_{\mu,\eta^\eps}$, cf. Definition~\ref{def:ls-Giorgi}.
    \item $\bs\calG_{\!\!\eps}$ denotes the graph De Giorgi functional, cf. Definition~\ref{def:ls-Giorgi}.
    \item $W_\bbT$ denotes the $2$-Wasserstein metric on $\calP(\R^d_\bbT)$.
    \item $\calD_{\;\!\bbT}$ is the metric slope of $\calE$ with respect to $W_\bbT$, cf. Definition~\ref{def:deGiorgi_local}.
    \item $\bs\calG_\bbT$ denotes the local De Giorgi functional, cf. Definition~\ref{def:deGiorgi_local}.
\end{itemize}
Let us also specify the notions of \emph{narrow convergence} and \emph{convolution}. 
A sequence $(\rho^n)_n\subset \calM(A)$ is said to converge narrowly to $\rho\in \calM(A)$, in which case we write $\rho^n \rightharpoonup \rho$, provided that
\[
\forall f\in C_\mathrm{b}(A): \qquad \int_A f \dx{\rho^n} \to \int_A f \dx{\rho} \qquad  \textrm{as } n \to \infty. 
\]
Given a function $f \colon A \times A \to \R$ and $\rho \in \calM(A)$, we write $f*\rho$ the convolution of $f$ and $\rho$, that is,
\begin{equation*}
	f*\rho(x)\coloneqq\int_A f(x,y)\dx\rho(y),
\end{equation*}
for any $x \in A$ such that the right-hand side exists.

\section{Preliminaries on graphs, gradient structures, and main results}\label{sec:preliminaries}

For the sake of clarity we divide this section in subsections. We specify the graph structure, recall results on the nonlocal-nonlocal interaction equation as well as the limiting PDE, including possible examples covered by our theory. At the end of the section we present our main results.

\subsection{Graph}\label{ssec:graph}

The graph is identified through a pair $(\mu,\eta)$, being $\mu$ a base measure standing for a (subset) of vertices and $\eta$ an edge weight function. We consider a non-negative $\sigma$-finite base measure $\mu\in\Mplus(\Rd)$ such that $\dx\mu = \widetilde\mu\dx\scrL^d$. We assume that the density $\widetilde\mu$ is bounded and uniformly continuous, denoting by $\omega_\mu\in C([0,\infty);[0,\infty))$ its modulus of continuity, which satisfies $\omega_\mu(\delta)\to 0$ as $\delta\to 0$. More precisely, this means
\begin{align}
	\label{mu1}\tag{$\muup_1$} 
    &\forall\, x,y\in\Rd \text{ it holds } \abs{\widetilde\mu(x)-\widetilde\mu(y)}\le\omega_\mu(\abs{x-y}),\\
	\label{mu2}\tag{$\muup_2$} 
    &\exists\, c_\mu,C_\mu>0 \text{ such that } \forall x\in\Rd \text{ it holds } c_\mu \le\widetilde\mu(x) \le C_\mu.
\end{align}
We fix an edge connectivity map $\vartheta:\Rd\times(\Rd\!\setminus\!\set*{0})\to[0,\infty)$ and another modulus of continuity $\omega_\vartheta\in C([0,\infty);[0,\infty))$ satisfying $\omega_\vartheta(\delta)\to0$ as $\delta\to 0$, and we make the following assumptions:
\begin{align}
	&\label{theta1}\tag{$\varthetaup_1$} \forall z\in\Rd:
 w\mapsto\vartheta(z,w) \text{ is symmetric and continuous on } \set{\vartheta(z,\cdot)>0};\\
    &\label{theta_new}\tag{$\varthetaup_2$} 
    \forall z,\bar z \in\Rd, w\in\Rd\!\setminus\!\set*{0} \text{ it holds } \abs{\vartheta(z,w)-\vartheta(\bar z,w)}\le \omega_\vartheta(\abs{z-\bar z}); \\
	&\label{theta2}\tag{$\varthetaup_3$}
	\exists\!\: \Csupp>0 \text{ such that }\forall z\in\Rd \text{ it holds } \supp\vartheta(z,\cdot) \subset B_{\Csupp};\\
	&\label{theta3}\tag{$\varthetaup_4$}
	\exists\!\: \Ctheta>0 \text{ such that }\sup_{(z,w)\in\Rd\times(\Rdzero)} \abs{w}^2\vartheta(z,w) \le \Ctheta;\\
    &\label{thetaBCprime}\tag{$\varthetaup_5$}
	\lim_{\delta\to 0}\sup_{z\in\Rd}\sup_{w\in B_\delta\setminus\{0\}} \abs{w}^2 \vartheta(z,w)=0;\\
	&\label{theta4}\tag{$\varthetaup_6$}
	\exists\!\: \ctheta>0 \text{ such that } \forall z,\xi\in\Rd \text{ it holds} \int_{\Rdzero}\! \abs{w\cdot\xi}^2\vartheta(z,w)\dx w \ge \ctheta \abs{\xi}^2.
\end{align}
We define the set $\Rddiag\coloneqq\set{(x,y)\in\Rd\times\Rd:x\ne y}$ and the family of edge weight function $(\eta^\eps)_{\eps>0}$, $\eta^\eps:\Rddiag\to[0,\infty)$ by
\begin{align}\label{eq:def:eta^eps}\tag{$\etaup$}
	\eta^\eps(x,y) \coloneqq \frac{1}{\eps^{d+2}}\vartheta\bra*{\frac{x+y}{2},\frac{x-y}{\eps}}.
\end{align} 
Furthermore, for every $\eps>0$, we introduce the sets
\begin{align*}%
	G^{\!\!\:\eps}\coloneqq\set{(x,y)\in\Rddiag:\eta^\eps(x,y)>0},
\end{align*}
and, for a given $x\in\Rd$
\begin{align}\label{eq:def:G^eps}
	G_{\!\!\:x}^{\!\!\:\eps}\coloneqq\set{y\in\Rd\!\setminus\!\set*{x}:\eta^\eps(x,y)>0}.
\end{align}
In order to lighten notations, we will drop the index $\eps>0$ for $\eta^\eps$ and $G^{\!\!\:\eps}$ when this is not relevant for the analysis. The next lemma collects some basic estimates following from the above assumptions, which are used throughout the article.
\begin{lemma}\label{lem:properties_eta-mu}
	Let $\mu,\vartheta$ satsify \eqref{mu1}, \eqref{mu2}, and \eqref{theta1}~--~\eqref{theta4}, respectively. For any $\eps>0$ let $\eta^\eps$ as in \eqref{eq:def:eta^eps}. Then, for any $\eps>0$ it holds
	\begin{align}
		\label{supp}\tag{\textsf{supp}}
		&\;\forall (x,y)\in G^{\!\!\:\eps} \text{ it holds }\abs{x-y}\le \Csupp\eps,\\
		\label{meas}\tag{\textsf{meas}}
		&\sup_{x\in\Rd}\abs{G_{\!\!\:x}^{\!\!\:\eps}} \le \Cmeas \eps^d,\\
		\label{mom}\tag{\textsf{mom}}
		&\sup_{(x,y)\in \Rddiag} \abs{x-y}^2\eta^\eps(x,y) \le\frac{\Cmom}{\eps^d},\\
		\label{int}\tag{\textsf{int}}
		&\sup_{\eps>0}\sup_{x\in\Rd} \int_{\Rdx}\abs{x-y}^2\eta^\eps(x,y)\dx{\mu}(y) \le \Cint,     
	\end{align}
	where $\Cmeas>0$ depends only on $\Csupp$ and the dimension $d$ and one can set $\Cint=\Cmu\Cmom\Cmeas$.
\end{lemma}       
\begin{proof}
	The assumption on the support of $\vartheta$,~\eqref{theta2}, and the definition of $\eta^\eps$ immediately yield \eqref{supp}. This also implies that $G_{\!\!\:x}^{\!\!\:\eps}$ is contained in a ball of radius $\Csupp\eps$ around $x$, whence \eqref{meas}. Furthermore, the latter observation and \eqref{theta3} give the bound \eqref{mom}. Finally, \eqref{int} is obtained by combining \eqref{mom}, \eqref{meas} and the upper bound from \eqref{mu2}.
\end{proof}

\begin{remark}[On the assumptions of the base measure $\mu$]\label{rem:link_to_EPSS}
    As $\supp\mu$ generalises the set of vertices in the nonlocal model, at first glance $\mu\ll\scrL^d$ seems to be rather restrictive, since it excludes finite graphs. However, models with $\mu\ll\scrL^d$ are known to be approximated by finite graphs \cite{EPSS2021}, the only restriction being that $\mu$ is a finite measure. In Theorem \ref{thm:extension_sigma_finite} we replace this assumption by those satisfied in the present paper; thereby the results in \cite{EPSS2021}, in conjunction with the present work, can be applied to obtain a finite graph approximation for the nonlocal interaction equation, cf. Theorem~\ref{thm:finite_graph}. Indeed, we note that the assumptions \eqref{theta1}~--~\eqref{theta4} and \eqref{mu1}, \eqref{mu2} imply that for any $\eps>0$ the pair $(\mu,\eta^\eps)$ satisfies all the assumptions from \cite{EPSS2021}. More precisely, for $x,y\in G^{\!\!\:\eps}$ we have the moment bounds
	\begin{align*}
		\sup_{(x,y)\in \Rddiag}|x-y|^2\lor|x-y|^4\eta^\eps(x,y)&\le\Cmom\frac{1\vee(\Csupp\eps)^2}{\eps^d},\\
		\sup_{x\in \Rd}\int_{\Rdx}|x-y|^2\lor|x-y|^4\eta^\eps(x,y)\dx\mu(y)&\le\Cint\bra*{1\vee(\Csupp\eps)^2},
	\end{align*}
	while the local blow-up control condition 
    \begin{align*}
		\lim_{\delta\rightarrow 0}\sup_{x\in\Rd} \sup_{y\in B_\delta(x)\setminus\{x\}}\abs{x-y}^2\eta^\eps(x,y) = 0,
	\end{align*}
	is implied by \eqref{thetaBCprime}. This condition plays a role in the proof of \cite[Theorem~3.15]{EPSS2021} for the construction of a family of finite base measures that uniformly satisfies the integrated blow-up control condition
    \begin{align*}
		\lim_{\delta\rightarrow 0}\sup_{x\in\Rd} \int_{B_\delta(x)\setminus\{x\}}\abs{x-y}^2\eta^\eps(x,y)\dx{\mu}(y) = 0.
	\end{align*}
    In the present work, this condition holds true due to \eqref{mu2} and \eqref{theta3}.
    
    For the sake of completeness, we mention that one could fix $\mu=\scrL^d$ by redefining $\vartheta^\mu(z,w) = \widetilde\mu\bra*{z} \vartheta(z,w)$ without changing the limiting equation, since $\mu$ is assumed to be absolutely continuous, with uniformly bounded and uniformly continuous density. However, the base measure $\mu$ makes a difference for the dynamic on the graph, thus the graph-to-local limit is more universal keeping $\mu$.
\end{remark}
\begin{remark}[More general edge weights $\eta^\eps$]
    The definition of $\eta^\eps$ can be relaxed to 
    \begin{align*}
        \eta^\eps(x,y) \coloneqq \frac{1}{\eps^{d+2}}\vartheta\bra*{\frac{x+y}{2},\frac{x-y}{\eps}}+f(\eps)g\bra*{\frac{x+y}{2},\frac{x-y}{\eps}}, 
    \end{align*}
    for some $g:\Rd\times(\Rd\!\setminus\!\set*{0})\to\R$, which is uniformly continuous in its first argument as well as symmetric and continuous in its second argument, and some $f:[0,\infty)\to\R$, which satisfies $f\in o(\eps^{-d-2})$ as $\eps\to 0$. Indeed, with this scaling the perturbation vanishes in the limit $\eps \to 0$, leaving the limiting tensor $\bbT$ unchanged.
\end{remark}

\subsection{Examples}%
We provide relevant examples of edge connectivity maps satisfying the conditions \eqref{theta1}~--~\eqref{theta4}. The requirements \eqref{theta1}~--~\eqref{theta3} are not difficult to check. 

Furthermore, observe that condition \eqref{thetaBCprime} holds for any $\vartheta$ for which there is $q<2$ and $C>0$ such that $\vartheta(z,w) \le C\abs{w}^{-q}$ for all $z\in\Rd$, $w\in\Rdzero$. 

Condition \eqref{theta4} is satisfied by any edge connectivity $\vartheta$ for which there exist $0\le r<R<\infty$ and $C>0$ such that $\vartheta(z,\cdot)|_{\overline{B}_R\setminus\overline{B}_r} \ge C$ for any $z\in\Rd$. Indeed, first note that for $\xi = 0$ there is nothing to show. Taking $w,\xi\in\Rd\!\setminus\!\set*{0}$, let us denote by $\varphi$ the angle between $w$ and $\xi$ and recall that $w\cdot\xi = \abs{w}\abs{\xi}\cos\varphi$. Choosing any $0 <\varphi_0 <s/2$, for any $0\le\varphi\le\varphi_0$ we have $\cos(\varphi)\ge \cos(\varphi_0)>0$. Therefore, denoting by $V_{\varphi_0}$ the volume of $B_1\cap\set{\varphi\le\varphi_0}$, which is the ($d$-dimensional) spherical sector of $B_1$ corresponding to $\varphi_0$, we have
\begin{align*}
	\int_{\Rdzero} \abs{w\cdot\xi}^2\vartheta(z,w)\dx w &\ge C V_{\varphi_0} \bra*{R^{d+2}-r^{d+2}} \abs{\cos(\varphi_0)}^2\abs{\xi}^2,
\end{align*}
and hence \eqref{theta4}. 
A simple but important example for a pair $(\mu,\vartheta)$ satisfying the above assumptions is $\vartheta(z,w) = C_d\bbone_{B_1}(w)$ with dimension dependent constant $C_d>0$ and $\mu=\scrL^d$. Indeed, \eqref{mu1}, \eqref{mu2} and \eqref{theta1}~--~\eqref{theta3} are easily checked, while \eqref{thetaBCprime} and \eqref{theta4} follow from the above considerations. This choice of (approximating) graph is peculiar as ~\eqref{eq:nlnl-interaction-eq} converges to the standard nonlocal interaction equation on $\Rd$, being $\bbT=\Id$, cf.~\eqref{eq:intro:limittensor}.

As a generalization, we consider a symmetric tensor field $\mathbb{D}\in C(\Rd ; \R^{d \times d})$ uniformly elliptic and bounded, i.e. $0<D_* \Id\leq \mathbb{D} \leq D^* \Id <\infty$ in the sense of quadratic forms, a variable radius $R\in C(\Rd;(R_*,R^*))$ for some $0<R_*<R^* <\infty$, and a normalization function $d\in C(\Rd; (d_*, d^*))$ for $0<d_* < d^* < \infty$. We define the connectivity function
\[ 
    \vartheta(z,w)=\begin{cases} 
    d(z), & \skp{w,\mathbb{D}(z)^{-1} w}\leq R(z) ;\\ 
    0, & \skp{w,\mathbb{D}(z)^{-1} w}> R(z) .
    \end{cases} 
\]
By construction, $\vartheta$ satisfies~\eqref{theta1}~--~\eqref{theta4}. 
The formula~\eqref{eq:intro:limittensor} gives, after a change of variable $w= \mathbb{D}(z)^{\frac{1}{2}} y$, where $\mathbb{D}(z)^{\frac{1}{2}}$ denotes the unique symmetric square root of $\mathbb{D}$:
\begin{align*}
	\bbT(z) &= \frac{1}{2} \int_\Rdzero w\otimes w \; \vartheta(z,x)\dx w  \\
    &= \frac{1}{2}d(z) \int_\Rd w\otimes w \; \bbone_{\skp{w,\mathbb{D}(z)^{-1} w}\leq R(z)}\dx w  \\
	&= \frac{1}{2} d(z)\bra*{\det \mathbb{D}(z)}^{\frac{1}{2}} \int_{B_{R(z)}(0)} \bra*{\mathbb{D}(z)^{\frac{1}{2}} y} \otimes \bra*{\mathbb{D}(z)^{\frac{1}{2}} y} \dx y \\
	&= \frac{1}{2} d(z)\bra*{\det \mathbb{D}(z)}^{\frac{1}{2}} C_d R(z)^{d+2} \mathbb{D}(z) ,
\end{align*}
where we used the identity
\begin{align*}
	\MoveEqLeft \int_{B_{R(z)}(0)} \bra*{\mathbb{D}(z)^{\frac{1}{2}} y} \otimes \bra*{\mathbb{D}(z)^{\frac{1}{2}} y} \dx y =  \mathbb{D}(z)^{\frac{1}{2}}\bra[\bigg]{\int_{B_{R(z)}(0)} y \otimes y \dx y}\mathbb{D}(z)^{\frac{1}{2}}\\
	 &= \mathbb{D}(z)^{\frac{1}{2}} \bra*{C_d R(z)^{d+2}\Id} \mathbb{D}(z)^{\frac{1}{2}} = C_d R(z)^{d+2} \mathbb{D}(z)
\end{align*}	
with $C_d = \int_{B_1(0)} y_1^2 \dx y = \pi^{\frac{d}{2}} /(2\Gamma(\frac{d}{2}+2)).$ 
In particular, by choosing the normalization
\[
  d(z) = \frac{2}{C_d R(z)^{d+2} \bra*{\det \mathbb{D}(z)}^{\frac{1}{2}}}, 
\]
and $\mu=\scrL^d$, we obtain the identity $\bbT= \mathbb{D}$.

\subsection{Nonlocal interaction energy}
	The nonlocal interaction energy considered in what follows is defined by
	\begin{align*}
		\calE(\rho) \coloneqq \frac12\iint_{\R^{2d}} K(x,y)\dx\rho(x)\dx\rho(y).
	\end{align*}
	The interaction kernel $K\colon\Rd\times\Rd\to\R$ is assumed to satisfy the following assumptions:
	\begin{align}
		\label{K1}\tag{\textsf{K1}}
		&K\in C^1(\Rd\times\Rd);\\
		\label{K2}\tag{\textsf{K2}}
		&K(x,y)=K(y,x)\quad \text{ for  } (x,y)\in\Rd\times\Rd;\\
		\label{K3}\tag{\textsf{K3}}
		\begin{split}
			&\exists\, L_K>0 \text{ such that for all }(x,y),(x',y')\in\Rd\times\Rd \text{ it holds}\\
			&|K(x,y)-K(x',y')|\le L_K\left(|(x,y)-(x',y')|\vee|(x,y)-(x',y')|^2\right);
		\end{split}\\
		\label{K4}\tag{\textsf{K4}}
		\begin{split}
			&\exists\, C_K>0 \text{ such that for all }(x,y)\in\Rd\times\Rd \text{ it holds}\\
			&|\nabla K(x,y)|\le C_K(1+|x|+|y|).
		\end{split}
	\end{align}
We observe the assumptions on the interaction kernel are somehow sharp as one does not expect the result holds true for pointy potentials, cf.~\cite[Remark 3.18]{EPSS2021}.
	\begin{remark}\label{rem:second_moment}
		As already noticed in \cite{EPSS2021}, assumption~\eqref{K3} implies that, for some $C >0$ and all $x,y\in \Rd$,
		\begin{equation}\label{as:K:Quad}
			\abs{K(x,y)} \leq C \bra*{1+ \abs{x}^2 + \abs{y}^2};
		\end{equation}
		indeed, for fixed $(x',y')\in \Rd\times \Rd$, \eqref{K3} yields
		\[
		\abs{\abs{K(x,y)} - \abs{K(x',y')}} \leq L_K \bra*{ 1 \vee 2\bra*{ |(x,y)|^2 + |(x',y')|^2}},
		\]
		and bounding the maximum on the right-hand side ($\vee$) by the sum, we arrive at $\abs{K(x,y)} \leq L_K +2 L_K \bra*{|(x',y')|^2 + |(x,y)|^2} + \abs{K(x',y')}$, which gives~\eqref{as:K:Quad} with $C=2L_K\bigl(1+|(x',y')|^2\bigr) + \abs{K(x',y')}$. We also notice, that the bound~\eqref{as:K:Quad} implies that $\calE\colon \calP_2(\Rd)\to \R$ is proper, since its domain contains ~$\calP_2(\Rd)$.
	\end{remark} 
The analysis in this manuscript easily extends to free energies of the form~\eqref{eq:interaction_energy} including potential energies $\calE_P(\rho)\coloneqq\int_\Rd P \dx{\rho}$, for some external potential $P\colon\R^d \to \R$ satisfying a local Lipschitz condition with at-most-quadratic growth at infinity and linear growth for $\nabla P$: similarly to~\eqref{K3}~--~\eqref{K4}, there exist $L, C\in (0,\infty)$ so that for all $x,y\in \R^d$ we have
\begin{align*}
	\abs{P(x)-P(y)} &\le L \bra*{ |x-y|\vee |x-y|^2},\\
	|\nabla P(x)| &\le C(1+|x|).
\end{align*}
For ease of presentation we shall not include the potential energy in our proofs, as no additional technical difficulties arise.

\subsection{Nonlocal-nonlocal interaction equation}\label{subsec:nlnl_interaction_eq}

In this subsection we recall results on the nonlocal interaction equation on graphs from \cite{EPSS2021} we shall use in the following. For simplicity, let $\rho\ll\mu$, where we use the notation~$\rho$ to denote both the measure and the density with respect to $\mu$. The equation reads, for $\mu-$a.e. $x$,
\begin{equation}\label{eq:nlnl-interaction-eq}
    \begin{split}
     \partial_t\rho_t(x)+\int_\Rd &\dgrad(K*\rho)(x,y)_- \eta(x,y) \rho_t(x) \dx\mu(y) \\
     &- \int_\Rd \dgrad(K*\rho)(x,y)_+ \eta(x,y) \dx\rho_t(y)=0.
     \end{split}\tag{$\mathsf{NL^2IE}$}
 \end{equation}
The theory in \cite{EPSS2021} also applies to the case when $\rho$ is not absolutely continuous with respect to $\mu$. The general weak form of \eqref{eq:nlnl-interaction-eq} is obtained in terms of the nonlocal continuity equation we specify later.

In order to have a graph-analogue of Wasserstein gradient flows for interaction energies we defined a suitable quasi-metric space, where the quasi-distance is obtained in a dynamical formulation à la Benamou--Brenier, \cite{BenamouBrenier2000,DNS09_CVPDE}. For this reason, it is crucial to identify paths connecting probability measures, a \emph{nonlocal} continuity equation, and an action functional to be minimized, resembling the total kinetic energy.
\begin{definition}[Action]\label{def:action}
	Let $\mu\in\calM^+(\Rd)$ and $\eta:\Rddiag\to[0,\infty)$ as before. For $\rho\in\calP(\Rd)$ and $j\in\calM(G)$, consider $\lambda\in\calM(\Rd\times\Rd)$ such that $\rho\otimes\mu,\mu\otimes\rho,|j|\ll|\lambda|$. We define
	\begin{equation*}%
		\A(\mu,\eta;\rho,j)\!\coloneqq\!\frac12\iint_G \left(\alpha\left(\frac{\dx{j}}{\dx{|\lambda|}},\frac{\dx{(\rho\otimes\mu)}}{\dx{|\lambda|}}\right)+\alpha\left(-\frac{\dx{j}}{\dx{|\lambda|}},\frac{\dx{(\mu\otimes\rho)}}{\dx{|\lambda|}}\right)\right)\eta \dx{|\lambda|} .
	\end{equation*}
	Hereby, the lower semicontinuous, convex, and positively one-homogeneous function $\alpha\colon \R\times\R_+\to\R_+\cup\{\infty\}$ is defined, for all $j\in\R$ and $r\geq 0$, by
	\begin{equation}\label{eq:def:alpha}
		\alpha(j,r)\coloneqq\begin{cases}
			\frac{(j_+)^2}{r} \qquad &\text{if}\ r>0,\\
			0 \qquad &\text{if}\ j\leq 0\ \text{and}\ r=0,\\
			\infty \qquad &\text{if}\ j> 0\ \text{and}\ r=0,
		\end{cases}
	\end{equation}
	with $j_+=\max\{0,j\}$. If $\mu$ and $\eta$ are clear from the context, we  write $\calA(\rho,j)$ for $\calA(\mu,\eta;\rho,j)$. Given a pair of curves $(\bs\rho,\bs j)\coloneqq ((\rho_t)_{t\in[0,T]},(j_t)_{t\in[0,T]})$ with $\rho_t\in\calP(\Rd)$ and $j_t\in\calM(G)$, we define
	\begin{align*}
		\bs\A(\mu,\eta;\bs\rho,\bs j)\coloneqq \int_0^T\A(\mu,\eta;\rho_t,j_t)\dx t.
	\end{align*}
\end{definition} 
The concept of graph gradient and graph divergence is as follows.
\begin{definition}[Nonlocal gradient and divergence] \label{def:nl_grad_div}
	For any function $\phi \colon \Rd \to \R$ we define its \emph{nonlocal gradient} $\babla \phi \colon G \to \R$ by
	\begin{equation*}%
		\babla\phi(x,y)=\phi(y)-\phi(x) \quad \mbox{for all $(x,y)\in G$}.
	\end{equation*}
	For any $j\in \calM(G)$, its \emph{nonlocal divergence} $\babla\cdot j \in \calM(\R^d)$ is defined as the negative $\eta$-weighted adjoint of $\babla$, i.e., for any $\varphi\in C_0(\Rd)$,
	\begin{align*}%
		\int \varphi \dx\babla \cdot j &= - \frac12\iint_G\babla\varphi(x,y)
		\eta(x,y)\dx{j}(x,y)\\
		&= \frac{1}{2}\int \varphi(x) \int \eta(x,y) \bra*{ \dx{j}(x,y) - \dx{j}(y,x)}.\notag
	\end{align*}
	In particular, for $j\in \calM^{\mathrm{as}}(G) \coloneqq \set{j\in \calM(G)\colon \dx j(x,y)=-\dx j(y,x)}$,
	\begin{equation*}%
		\int \varphi \dx{\babla\cdot j}= \iint_G \varphi(x) \eta(x,y) \dx{j}(x,y).
	\end{equation*}
\end{definition}
The following two lemmas will be employed for $\eps$ fixed.
\begin{lemma}%
	Let $\mu\in\calM^+(\Rd)$, $\eta^\eps\colon\Rddiag\to\R$ be as in \eqref{eq:def:eta^eps} such that \eqref{theta1}~--~\eqref{theta4} are satisfied. Let $(\rho^\eps)_{\eps>0}\subset\calP(\Rd)$ and $(j^\eps)_{\eps>0}\subset\calM(G^{\!\!\:\eps})$ be such that $\calA(\mu,\eta^\eps;\rho^\eps,j^\eps)< \infty$. Then, for any measurable $\Phi\colon G^{\!\!\:\eps}\rightarrow\R_+$, it holds
	\begin{equation*}%
		\frac12\iint_{G^{\!\!\:\eps}}\Phi\,\eta^\eps\dx{\abs{j^\eps}} \le \sqrt{\calA(\mu,\eta^\eps;\rho^\eps,j^\eps)\iint_{G^{\!\!\:\eps}}\Phi^2\eta^\eps \dx{(\rho^\eps\otimes\mu+\mu\otimes\rho^\eps)}}.
	\end{equation*}
\end{lemma}
\begin{proof}
	By Remark \ref{rem:link_to_EPSS}, this follows from \cite[Lemma 2.10]{EPSS2021}.
\end{proof}
\begin{corollary}\label{cor:A-bound}
	Let $(\mu^\eps)_{\eps>0} \subset \calM^+(\Rd)$ and $(\eta^\eps)_{\eps>0}$ be families  satisfying \eqref{mu1}, \eqref{mu2}, and \eqref{theta1}~--~\eqref{theta4}, respectively, uniformly in $\eps$. Let $(\rho^\eps)_{\eps>0} \subset \calP(\Rd)$ and $(j^\eps)_{\eps>0} \subset\calM(G^{\!\!\:\eps})$. Then, for any measurable $\Phi:G^{\!\!\:\eps}\to\R_+$ satisfying $\Phi(x,y)\le \abs{x-y}\lor\abs{x-y}^2$, we have
	\begin{equation*}
		\frac12\iint_{G^{\!\!\:\eps}} \Phi\, \eta\dx{\abs{j^\eps}} \le \sup_{\eps>0} \sqrt{2\Cint\calA(\mu^\eps,\eta^\eps;\rho^\eps,j^\eps)}.
	\end{equation*}
\end{corollary}
\begin{proof}
	Keeping in mind Remark \ref{rem:link_to_EPSS} and the upper bound for $\Phi$, this immediately follows from \cite[Corollary 2.11]{EPSS2021}. 
\end{proof}
We consider the following nonlocal continuity equation in flux form
\begin{equation}\label{eq:nlce_measures}
	\partial_t\rho_t+\babla\cdot j_t=0 \qquad \text{on}\ (0,T)\times\Rd,
\end{equation}
where $\bs\rho=(\rho_t)_{t\in[0,T]}$ and $\bs j=( j_t)_{t\in[0,T]}$ are unknown Borel families of measures in $\calP(\Rd)$ and $\calM(G)$, respectively. Equation \eqref{eq:nlce_measures} is understood in the weak form, i.e. $\forall\bs\varphi\in C_c^\infty((0,T)\times\Rd)$,
\begin{equation}\label{eq:nce-weak}
	\int_0^T\int_\Rd\partial_t\varphi_t(x)\dx\rho_t(x)\dx t +\frac12\int_0^T\iint_G\babla\varphi_t(x,y)\eta(x,y)\dx j_t(x,y)\dx t=0.
\end{equation}
Since $|\babla\bs\varphi(x,y)|\le||\bs\varphi||_{C^1}(2\wedge|x-y|)$, the weak formulation is well-defined under the integrability condition
\begin{equation}\label{eq:integrability-cond}
	\int_0^T\iint_G(2\wedge|x-y|)\eta(x,y)\dx \abs{j_t}(x,y)\dx t<\infty .
\end{equation}
\begin{remark}
	The integrability condition~\eqref{eq:integrability-cond} is automatically satisfied by any pair $(\bs \rho,\bs j)$, which satisfies $\bs\A(\mu,\eta;\bs\rho,\bs j)< \infty$, due to Corollary \ref{cor:A-bound}.
\end{remark}
Notice that any curve satisfying~\eqref{eq:nce-weak} and~\eqref{eq:integrability-cond} has a weakly continuous representative.
\begin{lemma}\label{lem:nce:weak_cont}
	Let $T>0$ and $(\bs\rho,\bs j)$ satisfy~\eqref{eq:nce-weak}~and~\eqref{eq:integrability-cond}. There exists a weakly continuous curve $(\bar{\rho}_t)_{t\in[0,T]}\subset\calP(\Rd)$ such that $\bar{\rho}_t=\rho_t$ for a.e.\ $t\in[0,T]$. Moreover, for any $\bs\varphi\in C_c^\infty([0,T]\times\Rd)$ and all $0\le t_0\le t_1\le T$ it holds
	\begin{equation*}%
		\begin{aligned}
			\int_\Rd\varphi_{t_1}(x)\dx\bar{\rho}_{t_1}(x)&-\int_\Rd\varphi_{t_0}(x)\dx\bar{\rho}_{t_0}(x)=\int_{t_0}^{t_1}\int_\Rd\partial_t\varphi_t(x)\dx\rho_t(x)\dx t\\
			&+\,\frac12\int_{t_0}^{t_1}\iint_G\babla\varphi_t(x,y)\eta(x,y)\dx j_t(x,y)\dx t.
		\end{aligned}
	\end{equation*}
\end{lemma}
\begin{proof}
    See \cite[Lemma 8.1.2]{AmbrosioGigliSavare2008} and \cite[Lemma 3.1]{Erb14}.
\end{proof}

Hence we arrive at the following definition of weak solution of the nonlocal continuity equation:
\begin{definition}[Nonlocal continuity equation in flux form]\label{def:nce-flux-form}
	A pair $(\bs\rho,\bs j)= ((\rho_t)_{t\in[0,T]},(j_t)_{t\in[0,T]})$ with $\rho_t\in\calP(\Rd)$ and $j_t\in\calM(G)$
	is called a \emph{weak solution} to the nonlocal continuity equation \eqref{eq:nlce_measures} provided that\begin{enumerate}
		\item %
		$\bs\rho$ is a weakly continuous curve in $\calP(\Rd)$;
		\item %
		$\bs j$ is a Borel-measurable curve in $\calM(G)$;
		\item %
		the pair $(\bs\rho,\bs j)$ satisfies \eqref{eq:nce-weak}.
	\end{enumerate}
	We denote the set of all weak solutions on the time interval $[0,T]$ by $ \NCE_T$.
		For $\varrho_0,\varrho_1\in\calP(\Rd)$, a pair $(\bs\rho,\bs j)\in \NCE(\varrho_0,\varrho_1)$ if $(\bs\rho,\bs j)\in\NCE\coloneqq\NCE_1$ and in addition $\rho_0=\varrho_0$ and $\rho_1=\varrho_1$. When the dependence of $\eta$ on $\eps>0$ needs to be emphasized, we will write $\NCE^\eps_T$, $\NCE^\eps$ and $\NCE^\eps(\varrho_0,\varrho_1)$ for the respective objects.
\end{definition}

An important property is the preservation of second moments, uniformly in $\eps$.
\begin{lemma}[Uniformly bounded second moments]\label{lem:2nd-mom-propagation}
	Let $(\mu^\eps)_{\eps>0} \subset \calM^+(\Rd)$ and $(\eta^\eps)_{\eps>0}$ be families  satisfying \eqref{mu1}, \eqref{mu2}, and \eqref{theta1}~--~\eqref{theta4}, respectively, uniformly in $\eps$. Let $(\rho_0^\eps)_\eps \subset \calP_{2}(\Rd)$ be such that $\sup_{\eps>0} M_2(\rho_0^\eps) < \infty$ and $(\bs\rho^\eps,\bs j^\eps)_n \subset \NCE_T^\eps$ so that $\sup_{\eps>0} \bs\A(\mu^\eps,\eta^\eps;\bs\rho^\eps,\bs j^\eps)<\infty$. Then, $\sup_{\eps>0} \sup_{t\in [0,T]}M_2(\rho_t^\eps) < \infty$.
\end{lemma}
\begin{proof}
	By Remark \ref{rem:link_to_EPSS}, this follows from \cite[Lemma 2.16]{EPSS2021}.
\end{proof}
The dyamical nonlocal quasi-metric is defined as follows.
\begin{definition}[Nonlocal extended quasi-metric]\label{def:T-metric}
	For $\mu \in \calM^+(\Rd)$, $\eta^\eps$ as before and $\varrho_0, \varrho_1 \in \calP_2(\Rd)$, we define a nonlocal extended quasi-metric by
	\begin{equation*}%
		\calT_{\mu,\eta^\eps}(\varrho_0,\varrho_1) =\inf\!\left\{\bra*{\int_0^1 \!\calA(\mu,\eta^\eps;\rho_t,j_t)\dx t: (\bs\rho,\bs j)\in\NCE^\eps(\varrho_0, \varrho_1)}^{1/2}\right\}.
	\end{equation*}
	When $\mu$ and $\eta^\eps$ are clear from the context or the dependence on $\eps>0$ is not important, we will shorten notation by writing $\calT$ or $\calT_{\eps}$ instead of $\calT_{\mu,\eta^\eps}$. 
\end{definition}
Properties of $\calT_{\mu,\eta^\eps}$ can be found in \cite[Section 2.4]{EPSS2021}, including that it is indeed an extended quasi-metric on $\calP_2(\Rd)$.

We denote by $\AC^2([0,T];(\mathcal{P}_2(\Rd),\mathcal{T}_{\mu,\eta^\eps}))$ the set of \emph{$2$-absolutely continuous curves} with respect to $\calT_{\mu,\eta^\eps}$, that is for such a \emph{$2$-absolutely continuous curve} $\bs\rho$ there exists $m\in L^2([0,T])$ such that
\begin{equation*}
	\mathcal{T}_{\mu,\eta^\eps}(\rho_s,\rho_t) \leq \int_s^t m(\tau)\dx \tau \qquad \text{ for all } 0\leq s \leq t \leq T . 
\end{equation*}
\begin{definition}[Absolutely continuous curve and metric derivative]\label{def:AC-curves_metric_derivative}
	The \emph{(forward) metric derivative} of a curve $\bs\rho\in\AC^2([0,T];(\mathcal{P}_2(\Rd),\mathcal{T}_{\mu,\eta^\eps}))$ is defined for a.e. $t\in[0,T]$ by
	\begin{align*}
		\abs{\rho_t'}_{\mu,\eta^\eps}\coloneqq \lim_{\tau\searrow 0} \frac{\mathcal{T}_{\mu,\eta^\eps}(\rho_t,\rho_{t+\tau})}{\tau} = \lim_{\tau\searrow 0} \frac{\mathcal{T}_{\mu,\eta^\eps}(\rho_{t-\tau},\rho_t)}{\tau}.
	\end{align*}
	We often shorten the notation by writing  $\abs{\rho_t'}_\eps$ instead of $\abs{\rho_t'}_{\mu,\eta^\eps}$.
\end{definition}
\begin{remark}
    We emphasize that the metric derivative defined above is only forward-in-time, or one-sided, due to the lack of symmetry for the nonlocal transportation cost~$\calT$.    
    The metric derivative 
    is well-defined as a consequence of the works~\cite[Proposition 2.2]{RossiMielkeSavare2008} and~\cite[Theorem 3.5]{ChenchiahRiegerZimmer2009} (see also~\cite[Lemma 7.1]{OhtaSturm2009}), which generalize~\cite[Theorem 1.1.2]{AmbrosioGigliSavare2008} to the asymmetric setting.
    
    Moreover, the metric derivative can be identified with the action of a suitable minimal flux as a consequence of \cite[Proposition 2.25]{EPSS2021}, the proof of which is based on \cite[Theorem 5.17]{DNS09_CVPDE} and \cite[Lemma 1.1.4]{AmbrosioGigliSavare2008}. Noting that the time reparametrizations in these proofs preserve orientation, it is clear that they generalize to the quasi-metric case based on the one-sided derivative from Definition~\ref{def:AC-curves_metric_derivative}.
\end{remark}
The measure-flux form of \eqref{eq:nlnl-interaction-eq} as nonlocal continuity equation is specified below.
\begin{definition}\label{def:nl2ie}
A curve $\bs\rho \colon [0,T]\to \calP_2(\R^d)$ is called a \emph{weak solution} to~\eqref{eq:nlnl-interaction-eq} 
if, for the flux $\bs j\colon [0,T]\to\calM(G)$ defined by
\begin{equation*}
  \dx j_t(x,y)=\dgrad\frac{\delta \calE}{\delta\rho}(x,y)_- \dx\rho_t(x)\dx\mu(y)-\dgrad\frac{\delta \calE}{\delta\rho}(x,y)_+ \dx\rho_t(y)\dx\mu(x),
\end{equation*}
the pair $(\bs \rho,\bs j)$ is a weak solution to the continuity equation
\begin{equation*}
  \partial_t\rho_t+\dgrad\cdot j_t=0 \qquad \text{on}\ [0,T]\times\Rd,
\end{equation*}
according to Definition \ref{def:nce-flux-form}.
\end{definition}

Weak solutions of \eqref{eq:intro:NLNL} are curves of maximal slope with respect to a (one-sided) strong upper gradient, which is the square root of the metric slope defined below. This allows to identify the set of weak solutions as the zero-level set of the so called \emph{De Giorgi functional}, \cite[Theorem 3.9]{EPSS2021}, in the quasi-metric space $(\calP_2(\Rd),\calT))$, being $\calT$ the quasi-distance recalled in Definition~\ref{def:T-metric}, cf.~\cite[Section 2.4]{EPSS2021}.
\begin{definition}[Metric slope and De Giorgi functional]\label{def:ls-Giorgi}
For any $\rho\in\calP_2(\Rd)$, let the \emph{metric slope} at $\rho$ be given by
\begin{equation*}%
    \calD_{\!\eps}(\rho) \coloneqq  \tA\bra[\Big]{\mu,\eta^\eps;\rho, -\dgrad \frac{\delta\calE}{\delta \rho}(\rho)}. 
\end{equation*}
For any $\bs\rho \in \AC^2([0,T];(\calP_2(\Rd),\calT))$, the \emph{graph De Giorgi functional} at $\bs\rho$ is defined as
\begin{equation*}%
	\bs\calG_{\!\!\eps}(\bs\rho)\coloneqq\calE(\rho_T)-\calE(\rho_0)+\frac{1}{2}\int_0^T\bra[\big]{\calD_{\!\eps}(\rho_\tau) + |\rho_\tau'|_\eps^2}\dx \tau.
\end{equation*}
\end{definition}

\subsection{The limiting local equation}\label{subsec:local_equation}

In the graph-to-local limit the equation obtained is a continuity equation, which can be interpreted in flux form.
\begin{definition}[Local continuity equation]\label{def:ce-flux}$ $\newline
	A pair $(\bs\rho,\bs j)= \bra[\big]{(\rho_t)_{t\in[0,T]},(j_t)_{t\in[0,T]}}$ with $\rho_t\in\calP(\Rd)$ and $j_t\in\calM(\Rd;\Rd)$ is called a \emph{weak solution} to the continuity equation
\begin{equation}\label{eq:continuity_equation_flux}
    \partial_t\rho_t+\nabla \cdot j_t=0
\end{equation}
provided that\begin{enumerate}
\item %
$\bs\rho$ is a weakly continuous curve in $\calP(\Rd)$;
\item %
$\bs j$ is a Borel-measurable curve in $\calM(\Rd;\Rd)$ such that
\[
\int_0^T\int_{\Rd}\dx|j_t|(\Rd)\dx t<\infty;
\]
\item %
the pair $(\bs \rho, \bs j)$ satisfies \eqref{eq:continuity_equation_flux} in the sense that for any $0\le t_0\le t_1\le T$,
\begin{align}\label{eq:ce_weak_continuous_rep}
\MoveEqLeft\int_{t_0}^{t_1}\int_\Rd\partial_t\varphi_t(x)\dx\rho_t(x)\dx t + \int_{t_0}^{t_1}\int_\Rd\nabla\varphi_t(x)\cdot\dx j_t(x)\dx t \\
&=\int_\Rd\varphi_{t_1}(x)\dx\rho_{t_1}(x)-\int_\Rd\varphi_{t_0}(x)\dx\rho_{t_0}(x) \qquad\forall \bs\varphi\in C_c^1([0,T]\times\Rd).\nonumber
\end{align}
\end{enumerate}
The set of all weak solutions on the time interval $[0,T]$ is denoted by $ \CE_T$.
For $\varrho_0,\varrho_1\in\calP(\Rd)$, a pair $(\bs \rho, \bs j)\in \CE(\varrho_0,\varrho_1)$ if $(\bs \rho, \bs j)\in\CE\coloneqq\CE_1$ and in addition $\rho_0=\varrho_0$ and $\rho_1=\varrho_1$.	
\end{definition}
\begin{remark}\label{rem:ce_Lip_test-function}
In the definition above, we use \cite[Lemma 8.1.2]{AmbrosioGigliSavare2008} which ensures formulation \eqref{eq:ce_weak_continuous_rep} is legitimate, up to considering a continuous representative in the left-hand side. We also observe that one can consider test functions only space dependent, bounded, and Lipschitz, $\operatorname{Lip}_b(\Rd)$; then \eqref{eq:ce_weak_continuous_rep} becomes
\begin{equation*}
	\pderiv{}{t} \int_\Rd \varphi(x) \dx\rho_t(x) = \int_\Rd \nabla\varphi(x) \cdot\dx j_t(x) \qquad\forall \varphi \in \operatorname{Lip}_b(\Rd). 
\end{equation*}
\end{remark}

The localisation outlined in the introduction provides in the limit a kinetic relation depending on a tensor, $\bbT$. More precisely, the general form is
\begin{equation}\label{eq:ce_tensor_general}
    \partial_t\rho_t+\mbox{div}(\rho_t \bbT v_t)=0,
\end{equation}
for a tensor $\bbT:\Rd\to\Rd\times\Rd$ Borel measurable, continuous, symmetric, and uniformly elliptic, cf.~Proposition~\ref{prop:tensor_uniqueness}.

As aforementioned, in \cite{Lisini_ESAIM2009}, the author focuses on the case $v=-\nabla (F'(u)+V)$ and provides a well-posedness theory in an \textit{equivalent} Wasserstein space. Indeed, he considers the Riemannian metric on $\Rd$ induced by $\bbT^{-1}$ (uniformly elliptic and bounded)
\[
d_{\bbT}(x,y)\!=\!\inf\!\left\{\int_0^1\!\!\!\!\!\sqrt{\langle\bbT^{-1}(\gamma(t))\dot\gamma(t),\dot\gamma(t)\rangle}:\!\gamma\in\AC([0,1];\Rd), \gamma(0)\!=\!x,\gamma(1)\!=\!y\right\}\!,
\]
and the corresponding Wasserstein distance is
\[
W_\bbT^2(\mu,\nu)\coloneqq \inf\left\{\iint d_\bbT^2(x,y)\dx\gamma(x,y):\gamma\in\Gamma(\mu,\nu)\right\} .
\]
This is equivalent to the dynamical version, for $\varrho_0,\varrho_1\in \calP_2(\R_\bbT^d)$,
\[
W_\bbT^2(\varrho_0,\varrho_1)=\inf\left\{\int_0^1 \norm*{\frac{\dx j_t}{\dx\rho_t}}_{L^2(\rho_t;\R_\bbT^d)}^2\dx t:(\bs\rho,\bs j)\in\CE(\varrho_0,\varrho_1)\right\},
\]
being
\[
\norm*{\frac{\dx j}{\dx\rho}}_{L^2(\rho;\R_\bbT^d)}^2=\int_\Rd\skp*{\bbT^{-1}(x)\frac{\dx j}{\dx\rho}(x),\frac{\dx j}{\dx\rho}(x)}\dx\rho(x).
\]
We observe that the uniform ellipticity of the tensor $\bbT^{-1}$ implies that the distance $d_\bbT$ is equivalent to the Euclidean one and we denote by $\R_{\bbT}^d$ the corresponding metric space $(\Rd,d_\bbT)$.
In the space $\calP_2(\R^d_\bbT)$ equipped with the $2$-Wasserstein distance, $W_\bbT$, Eq. \eqref{eq:ce_tensor_general}, with the vector field $v=-\nabla (F'(u)+V)$, is  a gradient flow of the free energy associated, meaning it is a curve of maximal slope with respect to a specific strong upper gradient, cf.~\cite[Section 3]{Lisini_ESAIM2009}.

Let us recall the definitions of strong upper gradient and curve of maximal slope. 
\begin{definition}[Strong upper gradient]\label{def:strong_ug}
A function $g:\calP_2(\R^d_\bbT) \to [0,+\infty]$ is called a \emph{strong upper gradient} for the functional $\calE$ if for every absolutely continuous curve $\bs\rho\in \AC^2([0,T];(\calP_2(\R^d_\bbT),W_\bbT))$ the function $g\circ \rho$ is Borel and
\begin{equation*}%
|\calE(\rho_t)-\calE(\rho_s))|\leq \int_s^t g(\rho_\tau)\abs{\rho_\tau'} \dx \tau, \quad \forall\, 0<s\leq t<T. 
\end{equation*}
In particular, if $(g\circ \rho_\cdot)\abs{\rho_\cdot'}\in L^1(0,T)$, then $\calE\circ \rho$ is absolutely continuous and 
\begin{equation*}%
    |(\calE\circ \rho)'|\leq g(\rho_t)\abs{\rho_t'}\,\mbox{ for a.e. }t\in[0,T].
\end{equation*}
\end{definition}

\begin{definition}[Curve of maximal slope]\label{def:maximal_slope_local} A curve $\bs\rho\in\AC^2([0,T];(\calP_2(\R^d_\bbT),W_\bbT)$ is called a \emph{curve of maximal slope} for $\calE$ with respect to the strong upper gradient $g$ if and only if $t\mapsto\calE(\rho_t)$ is  a non-increasing map satisfying
\begin{equation*}%
    \calE(\rho_t)-\calE(\rho_s))+\frac{1}{2}\int_s^t\left(g(\rho_\tau)^2+\abs{\rho_\tau'}^2\right)\dx \tau=0, \quad \forall\, 0<s\leq t<T.
\end{equation*}
\end{definition}
The equation we study, 
\begin{equation*}
    \partial_t\rho_t=\operatorname{div}(\rho_t  \bbT\nabla K*\rho_t), \tag{\ref{eq:intro:NLIE:one}} 
\end{equation*} 
differs from that in \cite{Lisini_ESAIM2009} since we do not consider diffusion, but nonlocal interaction instead. Following \cite{AmbrosioGigliSavare2008} and \cite{Lisini_ESAIM2009}, we can formulate \eqref{eq:intro:NLIE:one} as gradient flows of the energy \eqref{eq:interaction_energy} using the corresponding strong upper gradient, that is the square root of the metric slope defined below, together with the corresponding De Giorgi functional in the continuum setting.

\begin{definition}\label{def:deGiorgi_local}
Let $\rho\in\calP_2(\Rd)$. The metric slope of the nonlocal interaction energy is given by 
\begin{align*}%
        \calD_{\;\!\bbT}(\rho) \coloneqq \int_\Rd \skp[\Big]{\nabla\frac{\delta\calE}{\delta\rho},\bbT\nabla\frac{\delta\calE}{\delta\rho}}\dx\rho.
    \end{align*}
For any $\bs\rho \in \AC^2([0,T];(\calP_2(\R^d_{\bbT}),W_{\bbT}))$, the \emph{local De Giorgi functional} at $\bs\rho$ is defined as
\begin{equation*}%
	\bs\calG_\bbT(\bs\rho)\coloneqq\calE(\rho_T)-\calE(\rho_0)+\frac{1}{2}\int_0^T\bra[\big]{\calD_{\;\!\bbT}(\rho_\tau) + |\rho_\tau'|^2_{\bbT}}\dx \tau.
\end{equation*}
\end{definition}

In Section~\ref{sec:gf_structure} we prove that curves of maximal slope are weak solutions of \eqref{eq:intro:NLIE:one}, cf.~Theorem~\ref{thm:weak_curves_max_slope}. For consistency, we state the definition of weak solutions for \eqref{eq:intro:NLIE:one}.
\begin{definition}\label{def:weak-nlie_tensor}
A curve $\bs\rho \colon [0,T]\to \calP_2(\R^d)$ is called a \emph{weak solution} to~\eqref{eq:intro:NLIE:one} if, for the flux $\bs j\colon [0,T]\to\calM(\Rd;\Rd)$ defined by
\begin{equation*}
  \dx j_t(x)=-\bbT\nabla\frac{\delta \calE}{\delta\rho}(x)\dx\rho_t(x),
\end{equation*}
the pair $(\bs \rho,\bs j)$ is a weak solution to the continuity equation
\begin{equation*}
  \partial_t\rho_t+\nabla\cdot j_t=0 \qquad \text{on}\ [0,T]\times\Rd,
\end{equation*}
according to Definition \ref{def:ce-flux}.    
\end{definition}

\subsection{Main results}

The main result of the present work is the graph-to-local limit for the nonlocal interaction equation.

\begin{theorem}[Graph-to-local limit]\label{thm:main_result}
    Let $(\mu,\vartheta)$ satisfy \eqref{mu1}, \eqref{mu2} and \eqref{theta1}~--~\eqref{theta4}. Let $\eta^\eps$ be given by \eqref{eq:def:eta^eps} and assume $K$ satisfies \eqref{K1}~--~\eqref{K4}. For any $\eps>0$ suppose that $\bs\rho^\eps$ is a gradient flow of $\calE$ in $(\calP_2(\Rd),\calT_\eps))$, that is,
  \begin{equation*}
      \bs\calG_{\!\!\eps}(\bs\rho^\eps) = 0 \quad \text{for any } \eps>0,
  \end{equation*}
  with $(\rho_0^\eps)_\eps \subset \calP_{2}(\Rd)$ be such that $\sup_{\eps>0} M_2(\rho_0^\eps) < \infty$. Then there exists $\bs\rho \in \AC^2([0,T];(\calP_2(\R^d_\bbT),W_{\bbT}))$ such that  $\rho_t^\eps \rightharpoonup \rho_t$ as $\eps\to0$ for all $t\in[0,T]$ and $\bs \rho$ is a gradient flow of $\calE$ in $(\calP_2(\R^d_\bbT),W_{\bbT})$, that is,
    \begin{equation*}
      \bs\calG_\bbT(\bs\rho) = 0.
  \end{equation*}
The tensor $\bbT$ is as in~\eqref{eq:intro:limittensor}.  
\end{theorem}
The previous theorem is also a rigorous link between Finslerian and Riemannian gradient flows of the nonlocal interaction energy. 
Since curves of maximal slope in the local setting are weak solutions of \eqref{eq:intro:NLIE:one}, we actually provide an existence result via stability.

\begin{theorem}\label{thm:weak_sol_nlie_tensor}
Let $\mu,\vartheta$ satisfy \eqref{mu1}, \eqref{mu2}, and \eqref{theta1}~--~\eqref{theta4}, respectively. Consider the tensor $\bbT$ as in~\eqref{eq:intro:limittensor}. Assume $K$ satisfies \eqref{K1}~--~\eqref{K4}. Let $\varrho_0\in\calP_2(\Rd)$ such that $\varrho_0\ll\mu$. There exists a weakly continuous curve $\bs\rho:[0,T]\to\calP_2(\Rd)$ such that $\rho_t\ll\mu$ for all $t\in[0,T]$ which is a weak solution to~\eqref{eq:intro:NLIE:one} with initial datum $\rho_0 = \varrho_0$.   
\end{theorem}

In particular, using a diagonal argument this provides an interesting property of finite graphs as they can be used to discretise the equation under study. 

\begin{theorem}\label{thm:finite_graph}
Let $\mu,\vartheta$ satisfy \eqref{mu1}, \eqref{mu2}, and \eqref{theta1}~--~\eqref{theta4}, respectively. Consider the tensor $\bbT$ as in~\eqref{eq:intro:limittensor}. Assume $K$ satisfies \eqref{K1}~--~\eqref{K4}. Let $\varrho_0\in\calP_2(\Rd)$ be such that $\varrho_0\ll\mu$. Then, for any $T>0$ there exists a sequence of finite graphs $((\mu^N,\eta^N))_{N\in\bbN}$ with $\mu^N\rightharpoonup^\ast\mu$, a sequence of initial data $(\varrho_0^N)_{N\in\bbN}\subset\calP_2(\Rd)$ with $\varrho_0^N\ll\mu^N$ and $\varrho_0^N\rightharpoonup\varrho_0$, and a sequence of curves $(\bs\rho^N)_{N\in\bbN}$, where each $\bs\rho^N$ is a gradient flow in the sense of Definition~\ref{def:nl2ie} with respect to $(\mu^N,\eta^N)$ starting from $\varrho_0^N$, such that
\begin{align*}
    \rho^N_t \rightharpoonup \rho_t \text{ narrowly as }N\to \infty,
\end{align*}
uniformly in $t$ on $[0,T]$, where $\bs\rho$ is the weak solution of \eqref{eq:intro:NLIE:one} from Theorem~\ref{thm:weak_sol_nlie_tensor}.
\end{theorem}

\section{Linking Graph and local structure}\label{sec:graph_and_local_structure}
In this section we connect the nonlocal structure given by the (sequence of) graphs~$(\mu,\eta^\eps)$ with the local Euclidean one. The first step is to construct weak solutions of $\CE$ from those of $\NCE$. Second, we show that in the limit $\eps\to 0$ the antisymmetric part of our quasi-metric structures vanishes. Then, we derive the local structure from the remaining symmetric part, characterizing the tensor~$\bbT$ in terms of the base measure $\mu$, and the edge connectivity function $\vartheta$, thereby identifying the local geometry. Hereafter, we assume the pair $(\mu,\vartheta)$ satisfies \eqref{mu1}, \eqref{mu2} and \eqref{theta1}~--~\eqref{theta4}, and we let $\eta^\eps$ be given by \eqref{eq:def:eta^eps}. 

\subsection{Continuous reconstruction and compactness} 

First, we show that solutions of the nonlocal continuity equation $\NCE$ can be represented as those of the local continuity equation $\CE$ by means of a specific choice of the flux.
\begin{proposition}[Local flux]\label{prop:jhat}
    Let $j\in\calM(\Rddiag)$ satisfy the integrability condition $\iint_{\Rddiag} \abs{x-y}\eta(x,y)\abs*{j}(x,y)<\infty$. Then there exists $\jh\in\calM(\Rd;\Rd)$ such that
    \begin{equation}\label{eq:lem_jh}
        \frac{1}{2}\iint_\Rddiag\babla \varphi\,  \eta \dx{j} =\int_{\R^d}\nabla\varphi\cdot\dx{\jh}, \qquad\text{for all $\varphi\in C_c^1(\Rd)$.}
    \end{equation}
    In particular, if $(\bs\rho,\bs j)\in\NCE_T$ such that $\bs\calA(\mu,\eta;\bs\rho,\bs j)<\infty$, then there exists  $(\jh_t)_{t\in[0,T]}\subset \calM(\R^d;\R^d)$ such that $(\bs{\rho}, \bs\jh)\in\CE_T$.
\end{proposition}
\begin{proof}
	The construction of the local flux hinges on the representation of nonlocal gradients as integrals along elementary \emph{needles}, following the terminology introduced by~\cite{Klartag2017}. Let $(x,y)\in \Rddiag$ fixed and define, for any measurable $A\in \calB(\R^d)$, the measure $\sigma_{x,y}\in \calM^+(\R^d)$ by
	\begin{equation*}
		\sigma_{x,y}[A]\! = \!\calH^1(A\cap \llbracket x,y\rrbracket) \!\quad\text{with}\quad\!  \llbracket x,y\rrbracket \!\coloneqq\! \set*{(1\!-\!s)x+s y \in \R^d:\!  s\in [0,1]}.
	\end{equation*}
	Introducing the unit vector $\nu_{x,y} \coloneqq \frac{y-x}{\abs{y-x}}$, we get the identities
    \begin{align}
		\varphi(y)-\varphi(x) &= \int_0^{\abs{y-x}} \nabla \varphi(x+s\nu_{x,y}) \cdot\nu_{x,y} \dx s = \int_{\llbracket x,y\rrbracket}\nabla \varphi(\xi) \cdot\nu_{x,y} \dx \calH^1(\xi) \notag\\
    &=\int_\Rd \nabla\varphi(\xi) \cdot \nu_{x,y} \dx\sigma_{x,y}(\xi). \label{eq:nonlocal_gradient:needle}
  \end{align}
Next, we define $\lambda\in \calM(\Rddiag)$ by $\lambda(A)\coloneqq \iint_A \abs{x-y}\eta(x,y)\dx j(x,y)$ and observe that it is finite due to the integrability assumption on $j$. 
Splitting $\lambda$ into its positive and negative parts and renormalizing these parts, we can employ \cite[Remark 5.1.2]{AmbrosioGigliSavare2008} to obtain a family of counting measures $(\lambda^N)_{N\in\bbN}$ on $\Rddiag$ such that $\lambda^N\rightharpoonup \lambda$ narrowly as $N\to\infty$. These measures can be written as
		\begin{equation*}
			\lambda^N = \sum_{k=1}^N \lambda_k^N \delta_{(x_k^N,y_k^N)}\qquad\text{with}\qquad 		\sup_{N\in\bbN} \sum_{k=1}^N \abs{\lambda_k^N} <\infty . 
		\end{equation*}
by the finiteness of $\lambda$.
Defining $\tilde\sigma_{x,y}\in\calM(\Rd;\Rd)$ by $\tilde\sigma_{x,y}(A)\coloneqq\nu_{x,y}\frac{\sigma_{x,y}(A)}{\abs{x-y}}$, the representation of nonlocal gradients in \eqref{eq:nonlocal_gradient:needle} yields
\begin{align*}
    \frac12\iint_\Rddiag\frac{\babla\varphi(x,y)}{\abs{x-y}}\dx\lambda^N(x,y)
    &=\frac12\iint_{\Rddiag}\int_{\Rd} \nabla\varphi(\xi) \cdot \dx\tilde\sigma_{x,y}(\xi)\dx\lambda^N(x,y)\\
    &=\frac12\sum_{k=1}^N\int_{\Rd} \nabla\varphi(\xi) \cdot \dx\tilde\sigma_{x_k^N,y_k^N}(\xi)\lambda_k^N\\
    &=\int_{\Rd} \nabla\varphi(\xi) \cdot \frac12\sum_{k=1}^N \lambda_k^N \dx\tilde\sigma_{x_k^N,y_k^N}(\xi).
\end{align*}
	Motivated by this identity, we define $\hat \jmath^N \in \calM(\Rd;\Rd)$ by
	\begin{equation*}
		\hat \jmath^N(A) \coloneqq \frac12\sum_{k=1}^N \lambda_k^N \tilde\sigma_{x_k^N,y_k^N}(A). 
	\end{equation*}
Observing that $\abs{\tilde\sigma_{x,y}}(\Rd) = \abs{\nu_{x,y}}\frac{\abs{\calH^1(\llbracket x,y \rrbracket\cap \Rd)}}{\abs{x-y}} = 1$, we obtain the uniform bound
\begin{align*}
    \sup_{N\in\bbN}\abs*{\hat \jmath^N}(\R^d)\le \sup_{N\in\bbN}\sum_{k=1}^N \abs{\lambda_k^N} <\infty,
\end{align*}
which implies that there exists $\jh\in\calM(\Rd;\Rd)$ and a (not relabelled) subsequence such that $\jh^N\oset{\ast}{\rightharpoonup}\jh$ as $N\to\infty$. 

We note that for any $\varphi\in C^1_c(\Rd)$ the map $\Rddiag\ni (x,y)\mapsto \babla\varphi(x,y)/\abs{x-y}$ is continuous and bounded. Indeed, given $(x,y)\in\Rddiag$ let $0<\delta< \abs{x-y}/2$. 
Then, due to the mean-value theorem, for any $(x',y')\in B_\delta((x,y))\subset \R^{2d}$ with $x'\ne y'$, we obtain the upper bound
\begin{align*}
    &\abs[\bigg]{\frac{\babla\varphi(x,y)}{\abs{x-y}} - \frac{\babla\varphi(x',y')}{\abs{x'-y'}}}\\
    &\le \frac{\abs[\big]{\babla\varphi(x,y)-\babla\varphi(x',y')}}{\abs{x-y}}+\frac{\abs[\big]{\babla\varphi(x',y')}\,\abs[\big]{\abs{x-y}-\abs{x'-y'}}}{\abs{x-y}\abs{x'-y'}}\\
    &\le \frac{2\norm{\nabla\varphi}_{L^\infty}\delta}{\abs{x-y}}+\frac{4\norm{\varphi}_{L^\infty}\delta}{\abs{x-y}(\abs{x-y}-2\delta)},
\end{align*}
where we used that $\abs{\abs{x-y}-\abs{x'-y'}} \le \abs{x-x'-y+y'}\le 2\delta$. This upper bound vanishes as $\delta\to 0$ and thereby ensures continuity. Boundedness follows again from the mean value theorem and the boundedness of $\nabla\varphi$. 

Using this for the narrow convergence of $\lambda^N \rightharpoonup \lambda$ and the weak-$^\ast$ convergence along a subsequence of $\jh^N$ to $\jh$, along this subsequence we obtain the identity
	\begin{align*}
		\frac{1}{2} \iint_{\Rddiag} \frac{\babla \varphi(x,y)}{\abs{x-y}} \dx\lambda^N(x,y) &= \int_\Rd \nabla \varphi(\xi) \cdot \dx\jh^N(\xi) \\
		 \qquad\qquad\downarrow N\to \infty \qquad\quad &\qquad\quad \downarrow N\to \infty \qquad\qquad\\
		 \frac{1}{2} \iint_{\Rddiag} \frac{\babla \varphi(x,y)}{\abs{x-y}} \dx\lambda(x,y) &= \int_\Rd \nabla \varphi(\xi) \cdot \dx \jh(\xi). 
	\end{align*}
Since the left limit is unique, the right limit is independent of the particular subsequence and thus $\div\jh$ is unique as well.
Finally, the second claim follows from the first part of Corollary \ref{cor:A-bound}, and the definitions of solutions to $\NCE$ and $\CE$, Lemma~\ref{lem:nce:weak_cont} and Definition~\ref{def:ce-flux}, respectively.
\end{proof}
The next two results concern compactness for solutions of the continuity equation constructed in Proposition~\ref{prop:jhat}. 
\begin{lemma}\label{lem:comp_jh}
    Let $(\mu^\eps)_{\eps>0} \subset \calM^+(\Rd)$ and $(\eta^\eps)_{\eps>0}$ be families satisfying \eqref{mu1}, \eqref{mu2}, and \eqref{theta1}~--~\eqref{theta4}, respectively, uniformly in $\eps$. For any $\eps>0$, let $\bs\rho^\eps \coloneqq (\rho_t^\eps)_{t\in[0,T]}\subset\calP(\Rd)$ and $\bs j^\eps\coloneqq (j_t^\eps)_{t\in[0,T]} \subset\calM(G)$ such that $\sup_{\eps>0} \bs\calA(\mu^\eps,\eta^\eps;\bs\rho^\eps,\bs j^\eps)<\infty$. Moreover, consider $(\bs {\hat\jmath}^\eps)_{\eps>0}$ associated to $(\bs j^\eps)_{\eps>0}$ as in Proposition \ref{prop:jhat}. Then
    \begin{enumerate}
        \item $(\int_\cdot \jh^\eps_t\dx{t})_{\eps>0}$ is weakly-$^\ast$ compact in $\calM((0,T)\times\Rd;\Rd)$;
        \item $(t\mapsto\jh^\eps_t)_\eps$ is equi-integrable w.r.t. $\scrL^1$.
    \end{enumerate}
    In particular, there exists $\bs\jh\coloneqq( \jh_t)_{t\in[0,T]} \subset \calM(\Rd;\Rd)$ such that (along a subsequence) we have $\int_\cdot\jh^\eps_t\dx{t} \oset{\ast}{\rightharpoonup} \int_\cdot\jh\dx{t}$ weakly-$^\ast$ in $\calM((0,T)\times\Rd;\Rd)$.
\end{lemma}
\begin{proof}
    Combining \eqref{eq:lem_jh}, Corollary \ref{cor:A-bound} and H{\"o}lder's inequality, we obtain for any measurable $I\subset[0,T]$ 
    \begin{align}
        |\bs {\hat\jmath}^\eps|(I\times\Rd)&\le\int_I\abs{\jh^\eps_t}(\Rd)\dx t\notag \\
        &= \sup_{\varphi\in C^1_c(\Rd)\setminus\{0\}}\frac{1}{\norm{\nabla\varphi}_{L^\infty}}\int_I\int_\Rd \nabla\varphi(x)\cdot\dx\jh^\eps_t(x)\dx t\notag\\
        \overset{\mathclap{\eqref{eq:lem_jh}}}&{=} \sup_{\varphi\in C^1_c(\Rd)\setminus\{0\}}\frac{1}{2\norm{\nabla\varphi}_{L^\infty}}\int_I\iint_{G^{\!\!\:\eps}}\babla\varphi(x,y)\eta^\eps(x,y)\dx j^\eps_t(x,y) \dx t\notag\\
        &\le \frac12\int_I\iint_{G^{\!\!\:\eps}} (2\land\abs{x-y})\eta^\eps(x,y)\dx \abs{j_t^\eps}(x,y)\dx t\notag\\
        &\le \sup_{\eps>0}\sqrt{2\abs{I}\Cint\bs\calA(\mu^\eps,\eta^\eps;\bs\rho^\eps,\bs j^\eps)}\label{eq:lem_comp_jh}, 
    \end{align}
    which implies compactness as $|I|\le T$. The equi-integrability follows from \eqref{eq:lem_comp_jh}. Finally, the representation of the limit is a consequence of disintegration and the equi-integrability.
\end{proof}
\begin{proposition}\label{prop:convergence_CE-solutions}
    Let $(\mu^\eps)_{\eps>0} \subset \calM^+(\Rd)$ and $(\eta^\eps)_{\eps>0}$ be families satisfying \eqref{mu1}, \eqref{mu2}, and \eqref{theta1}~--~\eqref{theta4}, respectively, uniformly in $\eps$. Let $(\bs\rho^\eps,\bs j^\eps)_{\eps>0}$ be such that $(\bs\rho^\eps,\bs j^\eps)\in\NCE^\eps_T$ for all $\eps>0$ and $\sup_{\eps>0}\bs\calA(\mu^\eps,\eta^\eps;\bs\rho^\eps,\bs j^\eps)<\infty$. Let $\bs\jh^\eps$ be associated to $\bs j^\eps$ as in Proposition~\ref{prop:jhat}. Then there exists a (not relabeled) subsequence of pairs $(\bs\rho^\eps,\bs \jh^\eps)\in\CE_T$ and a pair $(\bs\rho,\bs \jh)\in\CE_T$ such that $\rho^\eps_t\rightharpoonup\rho_t$ narrowly in $\calP(\Rd)$ for a.e. $t\in[0,T]$ and such that $\int_\cdot\jh^\eps_t\dx{t} \oset{\ast}{\rightharpoonup} \int_\cdot\jh_t\dx{t}$ weakly-$^\ast$ in $\calM((0,T)\times\Rd;\Rd)$.
\end{proposition}
\begin{proof}
    The weak-$^\ast$ convergence of  $\int_\cdot\jh^\eps_t\dx{t}$ to $\int_\cdot\jh_t\dx{t}$ follows from Lemma \ref{lem:comp_jh}. Narrow convergence of $\rho^\eps_t$ to $\rho_t$ can be obtained by using \cite[Proposition 3.3.1]{AmbrosioGigliSavare2008}. More precisely, compactness is ensured by Lemma \ref{lem:2nd-mom-propagation} and Prokhorov's theorem, whilst equicontinuity is a consequence of the equi-integrability of the flux, similarly, e.g., to \cite[Lemma 4.5]{hraivoronska2022diffusive}.
\end{proof}

\subsection{Limiting tensor structure}%

This section is devoted to identifying the tensor $\bbT$ for the limiting interaction equation~\eqref{eq:intro:NLIE:one}. The limiting structure depends on the underlying Finslerian nature of \eqref{eq:nlnl-interaction-eq} as gradient flow of $\calE$ in $(\calP(\Rd),\calT_{\mu,\eta^\eps})$, being $\calT_{\mu,\eta^\eps}$ the upwind mass transportation cost, cf.~Definition~\ref{def:T-metric}.

In order to obtain the tensor we need to recall the definition of graph tangent fluxes introduced in \cite[Proposition 2.25]{EPSS2021}. Given a pair $(\mu,\eta^\eps)$, the space of graph tangent fluxes at $\rho\in\calP_2(\Rd)$ is defined by
\begin{equation}\label{eq:def:T_rho}
\begin{aligned}
    T^\eps_\rho\calP_2(\Rd) \coloneqq \bigl\{&j\in\calM(G^{\!\!\:\eps}):\calA(\mu,\eta^\eps;\rho,j)<\infty \ \text{ and } \\ &\calA(\mu,\eta^\eps;\rho,j)\le \calA(\mu,\eta^\eps;\rho,j+j_{\mathsf{div}}) \;\forall j_{\mathsf{div}}\in\Mdiv(G^{\!\!\:\eps})\bigr\},
\end{aligned}
\end{equation}
where
\begin{align*}%
    \Mdiv(G^{\!\!\:\eps})\coloneqq \set[\bigg]{j_{\mathsf{div}}\in\calM(G^{\!\!\:\eps}): \iint_{G^{\!\!\:\eps}} \babla\varphi\dx j_{\mathsf{div}} = 0\;\forall\varphi\in C_c^\infty(\Rd)}.
\end{align*}
We define the space of tangent velocities by
\begin{equation}\label{eq:def:tildeT_rho}
\begin{aligned}
    \widetilde T^\eps_\rho\calP_2(\Rd) \coloneqq \set*{v:G^{\!\!\:\eps}\to\R: v_+\dx(\rho\otimes\mu)-v_-\dx(\mu\otimes\rho)\in T^\eps_\rho\calP_2(\Rd)}.
\end{aligned}
\end{equation}
In particular, the set $\set{\babla\varphi:\varphi\in C_c^\infty(\Rd)}$ is dense in $\widetilde T^\eps_\rho\calP_2(\Rd)$ with respect to a suitable $L^2$-norm (cf.~\cite[Proposition 2.26]{EPSS2021}).
The crucial operator for the Finslerian structure is the tangent-to-cotangent mapping $\widetilde l^\eps_{\rho}:\widetilde T^\eps_\rho\calP_2(\Rd)\to\bra[\big]{\widetilde T^\eps_\rho\calP_2(\Rd)}^\ast$, which is given, for a fixed $v\in \widetilde T^\eps_\rho\calP_2(\Rd)$, by
\begin{align}\label{eq:def:l}
    \widetilde l^\eps_{\rho}(v)[w]\coloneqq\frac12\iint_{G^{\!\!\:\eps}} w\eta^\eps\pra*{v_+\dx(\rho\otimes\mu)-v_-\dx(\mu\otimes\rho)},
\end{align}
for any $w\in \widetilde T^\eps_\rho\calP_2(\Rd)$. Furthermore, $\widetilde{l}_\rho^\eps$ is also the first variation of $\frac{1}{2}\tA(\mu,\eta^\eps,\rho,\cdot)$ at $v$ in the direction $w$, and takes a similar role of a Riemannian metric in the Finslerian framework from \cite{EPSS2021,HPS21}.
In this subsection we will rigorously identify the limiting tensor arising from the Finslerian structure by showing in Propositions~\ref{prop:symmetry}~and~\ref{prop:tensor} that 
\begin{align*}
     \widetilde l^\eps_{\rho}(\babla\varphi)[\babla\psi] &= \frac12\iint_{G^{\!\!\:\eps}}\babla \varphi(x,y)\babla\psi(x,y)\eta^\eps(x,y)\dx\rho(x)\dx\mu(y) + o(1)\\
     &=\int_\Rd\nabla \varphi(x)\cdot\bbT^\eps(x)\nabla\psi(x)\dx\rho(x)+o(1)
    \end{align*}
where
    \begin{equation*}
      \bbT^\eps(x) \coloneqq \frac12\int_{\Rdx} (x-y)\!\otimes\!(x-y)\,\eta^\eps(x,y)\dx\mu(y). 
    \end{equation*}
Then, we will see that $\bbT^\eps$ converges to a unique limit $\bbT$ on compact subsets of $\R^d$ as $\eps\to 0$  (cf.~Proposition \ref{prop:tensor_uniqueness}), before concluding that $\widetilde l^\eps_{\rho}$ itself converges to a limiting inner product containing $\bbT$ (cf.~Theorem \ref{thm:limiting_inner_product}). Our first main step towards these goals is a proof that $\widetilde{l}^\eps_\rho$ is almost symmetric, with the antisymmetric part vanishing as $\eps \to 0$. To see this, we require the following two technical lemmas.

We remind the reader of the assumption that $(\mu,\vartheta)$ satisfy \eqref{theta1}~--~\eqref{theta4} and \eqref{mu1}, \eqref{mu2}, and that $\eta^\eps$ and $G^{\!\!\:\eps}$ are given by \eqref{eq:def:eta^eps} and \eqref{eq:def:G^eps}, respectively.
\begin{lemma}\label{lem:support_intersection}
    For every $\varphi\in C_c(\Rd)$ it holds
    \begin{align*}
        \abs{G^{\!\!\:\eps}\cap \supp\babla\varphi} \le 2 \Cmeas \abs{\supp\varphi}  \eps^d, \qquad\text{for all } \eps>0.
    \end{align*}
    Furthermore, for every $\eps>0$ the intersection $\abs{G^{\!\!\:\eps}\cap \supp\babla\varphi}$ is compact in $\Rddiag$.
\end{lemma}
\begin{proof}
    First, observe that $\supp\babla\varphi$ is contained in $(\supp\varphi\times\Rd)\cup(\Rd\times \supp\varphi)$. Indeed, if $\babla\varphi(x,y) \ne 0$, then either $\varphi(x)\ne 0$ or $\varphi(y)\ne 0$. Thus, due to the symmetry of $G^{\!\!\:\eps}$ and the assumption \eqref{meas}, we obtain
    \begin{align*}
        \MoveEqLeft\abs*{G^{\!\!\:\eps}\cap \supp\babla\varphi}\le \abs*{G^{\!\!\:\eps}\cap(\supp\varphi\times\Rd)\cup(\Rd\times \supp\varphi)} \\
        &\le 2 \abs*{G^{\!\!\:\eps}\cap (\supp\varphi\times\Rd)}
        = 2 \int_{\supp\varphi}\abs*{G_{\!\!\:x}^{\!\!\:\eps}}\dx x \le 2\Cmeas\abs{\supp\varphi}\eps^d.
    \end{align*}
    The compactness of $G^{\!\!\:\eps}\cap\supp\babla\varphi$ is an easy consequence of the Assumption~\eqref{supp}. Indeed, for $(x,y)\in G^{\!\!\:\eps}\cap(\Rd\times \supp\varphi)$ we have $y\in \supp\varphi$ and $\abs{x-y}\le \Csupp\eps$ and hence $x\in \supp\varphi+B_{\Csupp\eps}$, which for fixed $\eps>0$ is also compact. Due to the symmetry of~$G^{\!\!\:\eps}$, this concludes the proof.
\end{proof}
\begin{lemma}%
\label{lem:Equicontinuity}
    Let $f_k \in C(\R)$, $k=1,2$, be $1$-Lipschitz functions.
    Then, for any $\eps_0>0$ and any $\varphi,\psi\in C_c^2(\R^d)$, the family of maps $(\xi^\eps)_{0<\eps\le\eps_0}$, $\xi^\eps:\Rd\to\R$ defined by
    \begin{align*}%
      \xi^\eps(x)\coloneqq\int_{\R^d\setminus\set*{x}}f_1\bra[\big]{\babla \varphi(x,y)}f_2\bra[\big]{\babla\psi(x,y)}\eta^\eps(x,y)\dx\mu(y)
    \end{align*}
    is equicontinuous and $\bigcup_{0<\eps\le\eps_0} \supp\xi^\eps$ is contained in a compact set. Moreover, the family satisfies the uniform bound $\sup_{\eps>0} \norm{\xi^\eps}_{L^\infty}\le \Cint \norm{\nabla\varphi}_{L^\infty} \norm{\nabla\psi}_{L^\infty}$ and is therefore compact by Ascoli-Arzelà.
\end{lemma}
\begin{proof}%
    Due to the absolute continuity of $\mu$, we can apply the change of variables $w=y-x$ and rewrite $\xi^\eps$ as
    \begin{align*}
        \xi^\eps(x)&=\int_{\Rd\setminus\set0}f_1\bra[\big]{\babla \varphi(x,x+w)}f_2\bra[\big]{\babla\psi(x,x+w)}\eta^\eps(x,x+w)\widetilde\mu(x+w)\dx w\\
        &=\int_{\Rd\setminus\set0} \frac{f_1\bra[\big]{\babla\varphi(x,x+w)}}{\abs{w}}\frac{f_2\bra[\big]{\babla\psi(x,x+w)}}{\abs{w}}\abs{w}^2\eta^\eps(x,x+w)\widetilde\mu(x+w)\dx w\\
        &= \int_{\Rd\setminus\set0} \tilde f_1(x,w)\tilde f_2(x,w)\widetilde\eta^\eps(x,w)\widetilde\mu(x+w)\dx w,
    \end{align*}
for $\tilde f_1(x,w)\coloneqq f_1(\babla\varphi(x,x+w))/|w|$, $\tilde f_2(x,w)\coloneqq f_2(\babla\psi(x,x+w))/|w|$, and $\widetilde{\eta}^\eps(x,w)\coloneqq |w|^2\eta^\eps(x,x+w) = \frac{1}{\eps^{d+2}}|w|^2\vartheta(x+w/2,w/\eps)$.
    In order to prove the uniform equicontinuity, we note that $1$-Lipschitz assumption on $f_k$ implies the following elementary inequality for $a,b,c,d\in \R$ and $k=1,2$:
    \begin{equation}\label{eq:fk-bound}
    \begin{aligned}
        \abs{f_k(a\pm b)-f_k(c\pm d)} &\le \abs{(a\pm b) - (c\pm d)}
        \le \abs*{a-c}+\abs*{b-d}.
    \end{aligned}
    \end{equation}
    Considering the difference $\tilde f_1(x,w) - \tilde f_1(y,w)$ and noting that all considerations apply analogously to $f_2$ and $\psi$, we first assume that $\abs{w}>\abs{x-y}^\alpha$ for $0<\alpha<1$. In this case, \eqref{eq:fk-bound} and the mean value theorem yield
    \begin{align*}
         \abs[\big]{\tilde f_1(x,w) - \tilde f_1(y,w)}\abs{w}&=\abs[\big]{f_1\bra[\big]{\babla \varphi(x,x+w)} - f_1\bra[\big]{\babla \varphi(y,y+w)}}\\
         &\le  \abs[\big]{\babla\varphi(y+w,x+w)}+\abs[\big]{\babla\varphi(x,y)}\\
         &\le 2 \norm{\nabla\varphi}_{L^\infty}\abs*{x-y}\\
         &< 2 \norm{\nabla\varphi}_{L^\infty}\abs*{x-y}^{1-\alpha}\abs*{w}.
    \end{align*}
    Now considering the case $\abs{w}\le\abs{x-y}^\alpha$, using Taylor with Lagrange remainder and again~\eqref{eq:fk-bound}, we obtain
    \begin{align*}
         \MoveEqLeft\abs[\big]{\tilde f_1(x,w) - \tilde f_1(y,w)}\abs{w} = \abs[\big]{f_1\bra[\big]{\babla\varphi(x,x+w)} - f_1\bra[\big]{\babla \varphi(y,y+w)}}\\
         &= \abs[\big]{f_1\bra[\big]{w\cdot\bra[\big]{\nabla\varphi(x)+\tfrac12H_\varphi(\zeta_{x,w}) w}}-f_1\bra[\big]{w\cdot\bra[\big]{\nabla\varphi(y)+\tfrac12H_\varphi(\zeta_{y,w}) w}}}\\
         &\le \abs[\big]{w\cdot\bra*{\nabla\varphi(x)-\nabla\varphi(y)}}+\abs[\big]{w\cdot\bra[\big]{\tfrac12\bra*{H_\varphi(\zeta_{x,w})-H_\varphi(\zeta_{y,w})} w}}\\
         &\le \bra*{\omega_{\nabla\varphi}(\abs{x-y})+\norm{H_\varphi}_{L^\infty}\abs{x-y}^\alpha}\abs{w}.
    \end{align*}
    Arguing analogously for $f_2$, we have $\abs{\tilde f_k(x,w) - \tilde f_k(y,w)} \le \omega(\abs{x-y})$ for a suitable $\omega\in C([0,\infty);[0,\infty))$ with $\omega(\eps)\to 0$ as $\eps\to 0$, which only depends on $\alpha$, and first and second derivatives of $\varphi$ and $\psi$.
    
    Regarding the difference $\abs{\widetilde\eta^\eps(x,w)-\widetilde\eta(y,w)}$, due to \eqref{theta_new} and \eqref{theta2} we have
    \begin{equation}\label{eq:eta-continuity}
    \begin{aligned}
        \MoveEqLeft\int_{\Rd\setminus\set0}\abs{\widetilde\eta^\eps(x,w)-\widetilde\eta^\eps(y,w)}\dx w\\
        &= \int_{\Rd\setminus\set0}\frac{\abs{w}^2}{\eps^{d+2}}\abs*{\vartheta\bra*{x+\frac{w}{2},\frac{w}{\eps}}-\vartheta\bra*{y+\frac{w}{2},\frac{w}{\eps}}}\dx w\\
        &= \int_{B_{\Csupp}}\abs{w}^2\abs*{\vartheta\bra*{x+\frac{\eps w}{2},w}-\vartheta\bra*{y+\frac{\eps w}{2},w}}\dx w\\
        &\le \omega_\vartheta(\abs{x-y})\int_{B_{\Csupp}}\abs{w}^2\dx w \\
        &\le \abs{B_1}C_{\mathsf{supp}}^{d+2}\omega_\vartheta(\abs{x-y}).
    \end{aligned}
    \end{equation}
    Furthermore, since both $f_1$ and $f_2$ are 1-Lipschitz, the mean value theorem yields $\abs*{\tilde f_1(x,w)}\le \norm{\nabla\varphi}_{L^\infty}$ and $\abs*{\tilde f_2(x,w)}\le \norm{\nabla\psi}_{L^\infty}$. Combining this with the previous estimates, \eqref{mu1}, \eqref{mu2} and Lemma~\ref{lem:properties_eta-mu}, we obtain
    \begin{equation}\label{eq:proof_Equicontinuity}
    \begin{aligned}
        &\int_{\Rd\setminus\set0}\abs[\big]{\tilde f_1(x,w)\tilde f_2(x,w) \widetilde\mu(x+w)\widetilde\eta^\eps(x,w)\\
        &\qquad\qquad- \tilde f_1(y,w)\tilde f_2(y,w) \widetilde\mu(y+w)\widetilde\eta^\eps(y,w)}\dx w\\
        &\le\int_{\Rd\setminus\set0} \abs[\big]{\tilde f_1(x,w)-\tilde f_1(y,w)}\tilde f_2(x,w)\widetilde\eta^\eps(x,w)\widetilde\mu(x+w)\dx w\\
        &\quad +\int_{\Rd\setminus\set0} \tilde f_1(y,w)\abs[\big]{\tilde f_2(x,w)-\tilde f_2(y,w)}\widetilde\eta^\eps(x,w)\widetilde\mu(x+w)\dx w\\
        &\quad +\int_{\Rd\setminus\set0} \tilde f_1(y,w)\tilde f_2(y,w)\abs[\big]{\widetilde\eta^\eps(x,w)-\widetilde\eta^\eps(y,w)}\widetilde\mu(x+w)\dx w\\
        &\quad +\int_{\Rd\setminus\set0} \tilde f_1(y,w)\tilde f_2(y,w)\widetilde\eta^\eps(y,w)\abs[\big]{\widetilde\mu(x+w)-\widetilde\mu(y+w)}\dx w\\
        &\le \Cint\bra*{ \norm{\nabla\varphi}_{L^\infty}+\norm{\nabla\psi}_{L^\infty}}\omega(\abs{x-y})\\
        &\quad + C_\mu \norm{\nabla\varphi}_{L^\infty} \norm{\nabla\psi}_{L^\infty}\abs{B_1}C_{\mathsf{supp}}^{d+2}\omega_\vartheta(\abs{x-y})\\
        &\quad +\Cmom \Cmeas\norm{\nabla\varphi}_{L^\infty} \norm{\nabla\psi}_{L^\infty}\omega_\mu(\abs{x-y}).
    \end{aligned}
    \end{equation}
    Together with the former considerations, this proves the equicontinuity of $(\xi^\eps)_{0<\eps\le\eps_0}$.
    
    Next, we observe that due to Lemma \ref{lem:support_intersection} for any fixed $\eps>0$ the support of $f_1\bra[\big]{\babla \varphi(x,y)}f_2\bra[\big]{\babla\psi(x,y)}\eta^\eps(x,y)$ is compact in $\Rd\times\Rd$. Consequently, $\supp\xi^\eps$ is compact in $\Rd$ for any fixed $\eps>0$ and so is the union $\bigcup_{0<\eps\le\eps_0} \supp\xi^\eps$.
    
    Finally, the bound $\norm{\xi^\eps}_{L^\infty}\le \Cint \norm{\nabla\varphi}_{L^\infty} \norm{\nabla\psi}_{L^\infty}$ again follows from the mean value theorem, the fact that both $f_1$ and $f_2$ are 1-Lipschitz, and the bound~\eqref{int}.    
\end{proof}
\begin{proposition}[Symmetry up to small perturbations]\label{prop:symmetry}
    Let $(\rho^\eps)_{\eps>0}$ be such that $\supp\rho^\eps\subset\supp\mu$ and $\rho^\eps$ converges narrowly to some $\rho\in\calP(\Rd)$ as $\eps\to 0$.

    Then, for all $\varphi,\psi\in C^2_c(\Rd)$, it holds
    \begin{align*}
        \widetilde l^\eps_{\rho^\eps}(\babla\varphi)[\babla\psi] = \frac12\iint_{G^{\!\!\:\eps}}\babla \varphi(x,y)\babla\psi(x,y)\eta^\eps(x,y)\dx\rho^\eps(x)\dx\mu(y) + o(1) . 
    \end{align*}
\end{proposition}
\begin{proof}
Using that $a_+=(1/2)(a+|a|)$, for $a\in\R$, and antisymmetry in \eqref{eq:def:l} we rewrite
\begin{align*}
    \widetilde l^\eps_{\rho^\eps}(\babla\varphi)[\babla\psi]&=
    \iint_{G^{\!\!\:\eps}}(\babla \varphi)_+(x,y)\babla\psi(x,y)\eta^\eps(x,y)\dx\rho^\eps(x)\dx\mu(y)\\
    &=\frac12\iint_{G^{\!\!\:\eps}}\babla \varphi(x,y)\babla\psi(x,y)\eta^\eps(x,y)\dx\rho^\eps(x)\dx\mu(y)\\
    &\qquad+\int_\Rd
    \bra*{\frac12\int_{\Rdx}|\babla \varphi|(x,y)\babla\psi(x,y)\eta^\eps(x,y)\dx\mu(y)}
    \dx\rho^\eps(x)\\
    &\eqqcolon \I + \II.
\end{align*}
We want to show that $\II$ vanishes as $\eps\to 0$. To this end, we introduce a mollification at scale $\alpha\in(0,1)$ and define $\dx\widetilde\rho^\eps\coloneqq \widetilde\rho^\eps\dx\scrL^d\coloneqq \nu_{\eps^\alpha}\ast\rho^\eps\dx\scrL^d$, where $\nu$ is a mollifier and $\nu_{\eps^\alpha}(x) \coloneqq \frac{1}{\eps^{\alpha d} }\nu\bra*{\frac{x}{\eps^\alpha}}$. Introducing
\begin{align*}
    \bar\xi^\eps(x)\coloneqq\frac12\int_{\Rdx}|\babla \varphi|(x,y)\babla\psi(x,y)\eta^\eps(x,y)\dx\mu(y),
\end{align*}
we can write
\begin{align*}
    \II &= \int_\Rd \bar\xi^\eps \dx\bra*{\rho^\eps-\widetilde\rho^\eps} + \int_\Rd \bar\xi^\eps\dx\widetilde\rho^\eps \eqqcolon \II_1 + \II_2.
\end{align*}
Due to Lemma \ref{lem:Equicontinuity} applied to the $1$-Lipschitz functions $f_1 = \abs{\cdot}$ and $f_2 = \Id$, we have $(\bar\xi^\eps)_{\eps>0}\subset C_0(\Rd)$ and $\bar\xi^\eps$ converges strongly in $C(\Rd)$ to some $\bar\xi^0\in C_0(\Rd)$ as $\eps\to 0$. Hence, we infer $\II_1$ converges to $0$ as $\eps\to 0$ by weak-strong convergence (see e.g. \cite[Proposition 3.5]{Brezis2010}), since $\rho^\eps$ and $\widetilde\rho^\eps$ weakly-$^\ast$ converge to the same limit~$\rho$. For $\II_2$ we symmetrize and estimate
\begin{align*}
    \II_2 &= \frac{1}{4}\iint_{G^{\!\!\:\eps}} \abs[\big]{\babla\varphi(x,y)} \; \babla\psi(x,y) \; \eta^\eps(x,y)\bra[\big]{\dx\widetilde\rho^\eps(x)\dx\mu(y)-\dx\mu(x)\dx\widetilde\rho^\eps(y)}\\
    &\le\frac{1}{4} \iint_{G^{\!\!\:\eps}}\abs[\big]{\babla\varphi(x,y)}\; \abs[\big]{\babla\psi(x,y)}\; \eta^\eps(x,y)\; \abs[\big]{\widetilde\rho^\eps(x)\widetilde\mu(y)-\widetilde\mu(x)\widetilde\rho^\eps(y)}\dx x\dx y\\
    &\le\frac{1}{4} \norm{\nabla\varphi}_{L^\infty}\norm{\nabla\psi}_{L^\infty}\iint_{G^{\!\!\:\eps}}|x-y|^2\; \eta^\eps(x,y)\;\abs[\big]{\widetilde\rho^\eps(x)\widetilde\mu(y)-\widetilde\mu(x)\widetilde\rho^\eps(y)}\dx x\dx y\\
    &\le \frac{1}{4}\norm{\nabla\varphi}_{L^\infty}\norm{\nabla\psi}_{L^\infty} \frac{\Cmom}{\eps^d} \iint_{G^{\!\!\:\eps}} \widetilde\rho^\eps(x)\;\abs[\big]{\widetilde\mu(y)-\widetilde\mu(x)}\dx x\dx y\\
    &\qquad +\frac{1}{4}\norm{\nabla\varphi}_{L^\infty}\norm{\nabla\psi}_{L^\infty}\iint_{G^{\!\!\:\eps}}|x-y|^2 \eta^\eps(x,y)\;\widetilde\mu(x)\;\abs[\big]{\widetilde\rho^\eps(x)-\widetilde\rho^\eps(y)}\dx x\dx y\\
    &\eqcolon \frac{1}{4} \norm{\nabla\varphi}_{L^\infty}\norm{\nabla\psi}_{L^\infty} \bra*{ \II_{2,1} + \II_{2,2}}  . 
\end{align*}
For the first term \eqref{meas}, \eqref{mom} and \eqref{mu1} yield the bound
\begin{equation*}
  \II_{2,1} \le \Cmom\Cmeas \omega_\mu(\Csupp\eps) .
\end{equation*}
Since  $\supp\nu_{\eps^\alpha}(\cdot-z)\subset B_{\eps^\alpha}(z)$ and the Lebesgue measure is translation invariant, Lemma \ref{lem:support_intersection} provides us the estimate $\abs{G^{\!\!\:\eps}\cap \supp\nu_{\eps^\alpha}(\cdot-z)}\le 2\Cmeas\abs{B_1}\eps^{d(1+\alpha)}$. 
With these preliminary considerations we apply the mean value theorem to the mollifier~$\nu$ together with the bounds \eqref{supp} and \eqref{meas} to obtain
\begin{align*}
 \II_{2,2} &\le \frac{1}{\eps^{\alpha d} }\int_\Rd\iint_{G^{\!\!\:\eps}} \abs{x-y}^2\eta^\eps(x,y)\widetilde\mu(x)\abs*{ \nu\bra*{\frac{x-z}{\eps^\alpha}}-\nu\bra*{\frac{y-z}{\eps^\alpha}}}\dx x\dx y\dx\rho^\eps(z)\\
    &\le \frac{1}{\eps^{d(1+\alpha)} }\Cmom C_\mu\int_\Rd\iint_{G^{\!\!\:\eps}\cap \supp\nu_{\eps^\alpha}(\cdot-z)}\abs*{\nabla\nu\bra*{\xi}\cdot\bra*{\frac{x-y}{\eps^\alpha}}}\dx x\dx y\dx\rho^\eps(z)\\
    &\le \Csupp\eps^{1-\alpha-d(1+\alpha)}\norm{\nabla\nu}_{L^\infty}\Cmom C_\mu\int_\Rd \abs{G^{\!\!\:\eps}\cap \supp\nu_{\eps^\alpha}(\cdot-z)}\dx\rho^\eps(z)\\
    &\le 2\Csupp\eps^{1-\alpha}\norm{\nabla\nu}_{L^\infty}\abs{B_1}\Cmeas\Cmom C_\mu .
\end{align*}
As $\alpha < 1$ this concludes the proof.
\end{proof}
The next step consists in the identification of the tensor from the remaining symmetric part of the graph structure. 
\begin{proposition}[$\eps$-Tensor]\label{prop:tensor}
    For all $\varphi,\psi\in C^1_c(\Rd)$ and all $x\in\Rd$, we have
    \begin{align*}
        \frac12\int_{\Rdx}\babla \varphi(x,y)\babla\psi(x,y)\eta^\eps(x,y)\dx\mu(y) &= \nabla \varphi(x)\cdot\bbT^\eps(x)\nabla\psi(x)+o(1) ,
    \end{align*}
	where
	\begin{equation}\label{eq_def:eps-Tensor}
		\bbT^\eps(x) \coloneqq \frac12\int_{\Rdx} (x-y)\!\otimes\!(x-y)\,\eta^\eps(x,y)\dx\mu(y). 
	\end{equation}
\end{proposition}
\begin{proof}
    By means of Taylor's theorem with Peano remainder and \eqref{int}, we obtain for any $x\in\Rd$ the identity
    \begin{align*}
        &\frac12\int_{\Rdx}\babla \varphi(x,y)\babla\psi(x,y)\eta^\eps(x,y)\dx\mu(y)\\
        &=\frac12\int_{\Rdx} \nabla \varphi(x)\cdot(x-y)\!\otimes\!(x-y)\nabla \psi(x)\eta^\eps(x,y)\dx\mu(y)+\calO(\omega(\Csupp\eps))\\
        &=\nabla \varphi(x)\cdot\bra*{\frac12\int_{\Rdx} (x-y)\!\otimes\!(x-y)\,\eta^\eps(x,y)\dx\mu(y)}\nabla \psi(x)+\calO(\omega(\Csupp\eps)),
    \end{align*}
    where the modulus $\omega:[0,\infty)\to [0,\infty)$ is defined by
    \begin{align*}
        \omega(\delta) = \norm{\nabla\varphi}_{L^\infty}\omega_{\nabla\psi}(\delta)+\norm{\nabla\psi}_{L^\infty}\omega_{\nabla\varphi}(\delta)+\omega_{\nabla\varphi}(\delta)\omega_{\nabla\psi}(\delta),
    \end{align*}    
    and where we used the fact that $\abs{x-y}\le \Csupp\eps$ by Assumption~\eqref{supp}. 
\end{proof}
Next, we show that the tensors $\bbT^\eps$ converge to a unique limiting tensor as $\eps\to 0$.
\begin{proposition}[Unique limiting tensor]\label{prop:tensor_uniqueness}
    There exists a unique $\bbT\in C(\R^d; \R^{d\times d})$ such that for any compact $K\subset \R^d$ and any $\eps_0>0$ the sequence $\bra*{\bbT^\eps}_{\eps_0 \geq \eps>0}$ with $\bbT^\eps$ defined as in \eqref{eq_def:eps-Tensor} converges strongly to $\bbT$ in $C(K;\R^{d\times d})$ as $\eps\to 0$. The limiting tensor, given by
    \begin{equation}\label{eq_def:Tensor}\tag{$\bbT$}
		\bbT(x) \coloneqq \frac12\widetilde\mu(x)\int_{\Rdzero} w\!\otimes\! w\,\vartheta(x,w)\dx w,
	\end{equation}
    is bounded and uniformly continuous. 
    
    Furthermore, the tensor $\bbT$ is uniformly elliptic, i.e. there exist $c,C>0$ such that for any $x,\xi\in\Rd$ we have
    \begin{align*}
        c\abs{\xi}^2 \le \xi\cdot\bbT(x)\xi \le C\abs{\xi}^2.
    \end{align*}
    Finally, for any $x\in\R$ the matrix $\bbT(x)$ is symmetric.
\end{proposition}
\begin{proof}
	By a change of variables and arguing analogously to \eqref{eq:eta-continuity} and \eqref{eq:proof_Equicontinuity}, we obtain
	\begin{align*}
		&\abs{\bbT^\eps(x)-\bbT^\eps(y)}\\
        &= \abs*{\frac12\int_{\Rdzero} w\!\otimes\! w\,[\eta^\eps(x,x+w)\widetilde\mu(x+w)-\eta^\eps(y,y+w)\widetilde\mu(y+w)]\dx w}\\
		&\le \frac12C_{\mathsf{supp}}^{d+2}\abs{B_1}C_\mu\omega_\vartheta(\abs{x-y}) + \frac12\Cmom \Cmeas \omega_\mu(\abs{x-y}).
	\end{align*}
	Similar to this, using Assumption~\eqref{int} instead of Assumptions~\eqref{mom} and~\eqref{meas}, we obtain $\norm{\bbT^\eps}_{L^\infty}\le \frac12\Cint$. 
	Hence, the family $(\bbT^\eps)_{\eps_0\ge\eps>0}$ is equicontinuous and equibounded. In particular, for every compact $K\subset\R^d$ and every vanishing sequence $(\eps_n)_{n\in\bbN}$, there exists a (not relabeled) subsequence and a tensor $\bbT\in C(\Rd;\R^{d\times d})$ with $\norm{\bbT}_{L^\infty}\le\frac12\Cint$ and $\abs{\bbT(x)-\bbT(y)}\le \frac12\Cmom \Cmeas \omega_\mu(\abs{x-y})$ such that $\bbT^{\eps_n}\to\bbT$ strongly in $C(K;\R^{d\times d})$ as $n\to\infty$.
	
	To identify the limit, we calculate
    \begin{align*}
        2\bbT^{\eps}(x) 
        &= \int_{\Rdzero} w\!\otimes\! w \, \frac{1}{\eps^{d+2}}\vartheta\bra*{x+\frac{w}{2},\frac{w}{\eps}}\widetilde\mu(x+w)\dx w\\
        &= \int_{\Rdzero} w\!\otimes\! w\,\vartheta(x+\eps w/2,w) \widetilde\mu(x+\eps w)\dx w\\
        &= \widetilde\mu(x)\int_{\Rdzero} w\!\otimes\! w\,\vartheta(x,w)\dx w + \calO(\omega_\vartheta(\Csupp\eps)+\omega_\mu(\Csupp\eps)).
    \end{align*}
	This shows the explicit form \eqref{eq_def:Tensor} of the tensor as well as its independene of the particular subsequence $(\eps_n)_{n\in\bbN}$ and the particular compact set $K\subset \R^d$. Having shown \eqref{eq_def:Tensor}, the uniform ellipticity of $\bbT$ is an easy consequence of \eqref{mu2}, \eqref{theta3} and \eqref{theta4}, while the symmetry of the matrix $\bbT(x)$ for any $x\in\Rd$ follows from \eqref{theta1}.
\end{proof}
Finally, we combine the previous results to derive the structure of the continuity equation \eqref{eq:intro:NLIE:one} from the Finsler-type product.
\begin{theorem}[Limiting inner product]\label{thm:limiting_inner_product}
    Let $\rho^\eps \rightharpoonup\rho$ narrowly in $\calP_2(\Rd)$. The tangent-to-cotangent mapping $\widetilde l^\eps_{\rho^\eps}:\widetilde T^\eps_{\rho^\eps}\calP_2(\Rd)\to\bra[\big]{\widetilde T^\eps_{\rho^\eps}\calP_2(\Rd)}^\ast$ defined as in~\eqref{eq:def:l} satisfy
    \begin{align*}
        \lim_{\eps\to 0} \widetilde l^{\eps}_{\rho^{\eps}}(\babla\varphi)[\babla\psi] &= \int_\Rd \nabla \varphi\cdot\bbT\nabla\psi\dx\rho, 
        \qquad\forall\varphi,\psi\in C^2_c(\Rd),
    \end{align*}
	with the tensor $\bbT\in C(\Rd;\R^{d\times d})$ obtained as limit of $(\bbT^\eps)_{\eps_0 \geq \eps>0}$ defined as in~\eqref{eq_def:eps-Tensor} from Proposition~\ref{prop:tensor_uniqueness}. 
\end{theorem}
\begin{proof}%
    By Proposition~\ref{prop:tensor_uniqueness}, there exists a limiting tensor $\bbT\in C(\R^d;\R^{d\times d})$ such that for any compact $K\subset \R^d$ we have $\bbT^{\eps}\to\bbT$ strongly in $C(K;\R^{d\times d})$ as $\eps \to 0$. In particular, for every $\varphi,\psi\in C^1_c(\Rd)$, we have that $\nabla\varphi\cdot\bbT^{\eps}\nabla\psi\to\nabla\varphi\cdot\bbT\nabla\psi$ strongly in $C_0(\R)$ as $n\to\infty$. As narrow convergence of measures implies weak-$^\ast$ convergence of measures, this allows us to employ usual weak-strong convergence argument (cf.~e.g. \cite[Proposition 3.5]{Brezis2010} for the strategy) to obtain
    \begin{align*}
        \lim_{\eps\to 0}\int_\Rd \nabla \varphi(x)\cdot\bbT^{\eps}(x)\nabla\psi(x)\dx\rho^{\eps} = \int_\Rd \nabla \varphi(x)\cdot\bbT(x)\nabla\psi(x)\dx\rho.
    \end{align*}
The proof is then concluded by employing Proposition \ref{prop:symmetry}, for which the $C^2$-regularity of the test functions is necessary.
\end{proof}

\section{Graph-to-local limit for curves of maximal slope}\label{sec:graph-to-local-limit}

The graph-to-local limit for \eqref{eq:nlnl-interaction-eq} is proven by exploiting the variational formulation of the equations as curves of maximal slopes, as explained in Sections~\ref{subsec:nlnl_interaction_eq}, \ref{subsec:local_equation}. The localising graph provided by $(\mu,\eta^\eps)$, for $\eta^\eps$ as in~\eqref{eq:def:eta^eps}, identifies a sequence of weak solutions of \eqref{eq:nlnl-interaction-eq-eps}, that is, for $\rho^\eps\ll\mu$, $\mu-$a.e. $x$ and a.e. $t$,
\begin{equation}\label{eq:nlnl-interaction-eq-eps}
    \begin{split}
     \partial_t\rho_t^\eps(x)+\int_\Rd &\dgrad(K*\rho^\eps_t)(x,y)_- \eta^\eps(x,y) \rho_t^\eps(x) \dx\mu(y) \\
     &- \int_\Rd \dgrad(K*\rho^\eps_t)(x,y)_+ \eta^\eps(x,y) \dx\rho_t^\eps(y)=0.
     \end{split}\tag{$\mathsf{NL^2IE}_\eps$}
 \end{equation}
In view of~\cite[Theorem~3.9]{EPSS2021}, weak solutions of \eqref{eq:nlnl-interaction-eq-eps} are curves of maximal slope in the Finslerian quasi-metric space $(\calP_2(\Rd), \mathcal{T}_{\mu,\eta^\eps})$, that is zero-level sets of the \textit{graph De Giorgi functional} 
\begin{equation*}
	\bs\calG_{\!\!\eps}(\bs\rho^\eps)=\calE(\rho_T^\eps)-\calE(\rho_0^\eps)+\frac{1}{2}\int_0^T\bra[\big]{\calD_{\!\eps}(\rho_\tau^\eps) + |\rho_\tau'|_{\eps}^2}\dx \tau,
\end{equation*}
with the \emph{metric slope}
\begin{equation*}
    \calD_{\!\eps}(\bs\rho^\eps)=  \iint_\Rddiag \abs[\big]{\bra[\big]{\dgrad K*\rho^\eps}_-(x,y)}^2\eta^\eps(x,y)\dx\rho^\eps(x)\dx\mu(y).
\end{equation*}
We show that in the $\eps\to 0$ limit we obtain a zero-level set of the \emph{local De Giorgi functional}
\[
	\bs\calG_\bbT(\bs\rho)=\calE(\rho_T)-\calE(\rho_0)+\frac{1}{2}\int_0^T\bra[\big]{\calD_{\;\!\bbT}(\rho_\tau) + |(\rho_\tau^\eps)'|^2_{\bbT}}\dx \tau,
\]
where the \emph{metric slope} is
\[
\calD_{\;\!\bbT}(\rho)=\int_\Rd\skp[\bigg]{\nabla\frac{\delta\calE}{\delta\rho},\bbT\nabla\frac{\delta\calE}{\delta\rho}}\dx\rho,
\]
thereby proving Theorem~\ref{thm:main_result}. As byproduct, the graph-to-local limit provides a proof of Theorem~\ref{thm:weak_sol_nlie_tensor}, the existence of weak solutions of \eqref{eq:intro:NLIE:one}. Indeed, we find that weak solutions are curves of maximal slopes in local setting, according to Definitions~\ref{def:maximal_slope_local} and~\ref{def:deGiorgi_local} with respect to the Wasserstein gradient flow structure of \eqref{eq:intro:NLIE:one} in the metric space $(\calP_2(\R^d_\bbT),W_\bbT)$. 

Throughout this section we fix the tensor $\bbT$ as in~\eqref{eq_def:Tensor} for $\mu,\vartheta$ satsifying \eqref{mu1}, \eqref{mu2}, and \eqref{theta1}~--~\eqref{theta4}, respectively, and we consider $\eta^\eps$ given by \eqref{eq:def:eta^eps} as before.

\subsection{Lower limit for the metric derivative and metric slope}\label{ssec:lsc}

In this section we derive the lower limits for the metric derivatives and the metric slopes, respectively. This will then be combined with a chain rule (cf. Proposition~\ref{prop:chain_rule_ineq}) in order to obtain the convergence of the zero level sets of the graph De Giorgi functionals to the zero level set of the local De Giorgi functional as $\eps\to 0$.

\begin{proposition}[Lower limit of the metric derivative]\label{prop:lower_limit_metric_derivatives}
Consider a family $(\bs\rho^\eps)_{\eps>0}\subset\AC^2([0,T];(\mathcal{P}_2(\Rd),\mathcal{T}_{\eps}))$, with $(\rho_0^\eps)_\eps \subset \calP_{2}(\Rd)$ such that $\sup_{\eps>0} M_2(\rho_0^\eps)~<~\infty$. Then there exists $\bs \rho\in \AC^2([0,T];(\calP_2(\R^d_\bbT),W_\bbT))$ such that, for a.e. $t\in[0,T]$, $\rho_t^\eps\rightharpoonup\rho_t$ and it holds
\begin{equation*}
    \liminf_{\eps\to 0}\int_0^T  \abs[\big]{\bra{\rho^\eps_t}'}_\eps^2\dx t\ge \int_0^T|\rho'_t|_\bbT^2\dx t.
\end{equation*}
\end{proposition}
\begin{proof}
\cite[Proposition 2.25]{EPSS2021} characterises absolutely continuous curves, and, in particular, for every $\eps>0$ it allows us to infer that there exists a unique $\bs j^\eps$ such that $(\bs\rho^\eps,\bs j^\eps)\in\NCE^\eps_T$ and for a.e. $t\in[0,T]$ it holds $\abs{\bra{\rho^\eps_\tau}'}_\eps^2 = \calA(\mu,\eta^\eps,\rho^\eps_t,j^\eps_t)<\infty$. We employ Proposition \ref{prop:jhat} to obtain $(\bs\rho^\eps,\bs \jh^\eps)\in\CE_T$; Proposition~\ref{prop:convergence_CE-solutions} implies convergence to a limit $(\bs\rho,\bs \jh)\in\CE_T$ such that $\bs\rho\subset \calP_2(\Rd)$ (by Lemma~\ref{lem:2nd-mom-propagation} and lower semicontinuity).

Next, we note that for any $\varphi\in C_c^\infty(\Rd)$ the one-sided Cauchy-Schwarz inequality \cite[Lemma 3.7]{EPSS2021} and Young's inequality yield
\begin{align*}
    &\skp{\babla\varphi, j^\eps}_{\eta^\eps} = \frac12\iint_\Rddiag \babla\varphi(x,y)\eta^\eps(x,y)\dx j^\eps(x,y)\\
    &\quad= \frac12\iint_\Rddiag\babla\varphi(x,y)\eta^\eps(x,y)(v^\eps_+(x,y)\dx\rho^\eps(x)\dx\mu(y)-v^\eps_-(x,y)\dx\mu(x)\dx\rho^\eps(y))\\
    &\quad= \widetilde l^\eps_{\rho^\eps}(v^\eps)[\babla\varphi]\\
    &\quad\le \sqrt{\widetilde l^\eps_{\rho^\eps}(v^\eps)[v^\eps] \widetilde l^\eps_{\rho^\eps}(\babla\varphi)[\babla\varphi]}\\
    &\quad= \sqrt{\calA(\mu,\eta^\eps;\rho^\eps,j^\eps) \tA(\mu,\eta^\eps;\rho^\eps,\babla\varphi)}\\
    &\quad\le \frac12 \calA(\mu,\eta^\eps;\rho^\eps,j^\eps) + \frac12 \tA(\mu,\eta^\eps;\rho^\eps,\babla\varphi).
\end{align*}
For completeness we mention the second equality holds since we have finite action, whence upwind flux, cf.~\cite[Lemma 2.6]{EPSS2021}.
Therefore, for $\bs\varphi\in C_c^\infty((0,T);C_c^\infty(\Rd))$ we obtain
\begin{align*}
    \liminf_{\eps\to 0}&\frac12\int_0^T  \calA\bra{\mu,\eta^\eps;\rho^\eps_t,j^\eps_t}\dx t\\
    &\ge \liminf_{\eps\to 0} \bra*{\int_0^T\skp[\big]{\babla\varphi_t,j^\eps_t}_{\!\eta^\eps}\dx t - \frac12\int_0^T  \tA\bra{\mu,\eta^\eps;\rho^\eps_t,\babla\varphi_t}\dx t}\\
    &\ge\lim_{\eps\to 0} \int_0^T\skp[\big]{\babla\varphi_t,j^\eps_t}_{\!\eta^\eps}\dx t - \limsup_{\eps\to 0}\frac12\int_0^T  \tA\bra{\mu,\eta^\eps;\rho^\eps_t,\babla\varphi_t)}\dx t\\
    \overset{\text{Thm.}\,\ref{thm:limiting_inner_product}}&{=} \int_0^T\skp[\big]{\nabla\varphi_t,\jh_t}\dx t - \frac12\int_0^T \int_\Rd \langle\nabla\varphi_t,\bbT\nabla\varphi_t\rangle\dx\rho_t\dx t.
\end{align*}
With this, arguing analogously to the last step in the proof of \cite[Theorem~6.2~(i)]{hraivoronska2022diffusive}, we set $V\coloneq\{\bbT^{1/2}\nabla\bs\varphi:\bs\varphi\in C_c^\infty((0,T);C_c^\infty(\Rd))\}$, being $\bbT^{1/2}(x)$ the square root of the positive-definite symmetric matrix $\bbT(x)$. Fenchel--Moreau duality theorem implies
\begin{align*}
    &\frac{1}{2}\int_0^T\left\|\frac{\dx \jh_t}{\dx\rho_t}\right\|_{L^2(\rho;\R_{\bbT}^d)}^2\dx t=\frac12 \int_0^T\int_\Rd \left\langle\bbT^{-1}\frac{\dx \jh_t}{\dx\rho_t}, \frac{\dx \jh_t}{\dx\rho_t}\right\rangle\dx\rho_t\dx t\\
    &\qquad=\sup_{\bs\varphi\in V}\bra*{\int_0^T\skp*{\bbT^{1/2}\nabla\varphi_t,\bbT^{-1/2}\jh_t}\dx t - \frac12\int_0^T \int_\Rd \langle\bbT^{1/2}\nabla\varphi_t,\bbT^{1/2}\nabla\varphi_t\rangle\dx\rho\dx t}\\
    &\qquad\le \liminf_{\eps\to 0}\frac12\int_0^T  \calA\bra{\mu,\eta^\eps;\rho^\eps_t,j^\eps_t}\dx t = \liminf_{\eps\to 0}\frac{1}{2}\int_0^T  \abs[\big]{\bra{\rho^\eps_\tau}'}_\eps^2\dx t,
\end{align*}
where the last equality follows from~\cite[Proposition~2.25]{EPSS2021} as mentioned at the beginning of the proof. The above inequality and Proposition~\ref{prop:convergence_CE-solutions} imply $\bs\rho\in \AC^2([0,T];(\calP_2(\R^d_\bbT),W_\bbT))$ since $(\bs\rho,\bs\jh)\in\CE_T$ and, for $0\le s\le t\le T$, it holds due to the Cauchy-Schwarz inequality
\begin{align*}
    W_\bbT^2(\rho_s,\rho_t)\le(t-s)\int_s^t\left\|\frac{\dx \jh_\tau}{\dx\rho_\tau}\right\|_{L^2(\rho_\tau;\R^d_\bbT)}^2\dx \tau<\infty,
\end{align*}
and Lebesgue differentiation theorem gives, for a.e. $\tau$ the claimed result
\begin{equation*}
    \abs{\rho'_\tau}_\bbT\le\norm*{\frac{\dx \jh_t}{\dx\rho_t}}_{L^2(\rho_t;\R^d_\bbT)} . \qedhere
\end{equation*}
\end{proof}
Let us turn to the lower limit for the slopes.
\begin{proposition}[Lower limit of metric slopes]\label{prop:D_NL->L_lsc}
     Let $(\mu,\vartheta)$ satisfy \eqref{mu1}, \eqref{mu2} and \eqref{theta1}~--~\eqref{theta4}. Let $\eta^\eps$ be given by \eqref{eq:def:eta^eps} and let $K$ satisfy \eqref{K1}~--~\eqref{K4}. Let $(\rho^{\eps_n})_{n\in\bbN}\subset\calP_2(\Rd)$ be such that $\rho^{\eps_n}\rightharpoonup\rho\in\calP_2(\Rd)$ narrowly as $n\to\infty$. Then, up to passing to a subsequence, we have
     \begin{align*}
        \liminf_{n\to\infty}\calD_{\!\eps_n}(\rho^{\eps_n}) \ge \calD_{\;\!\bbT}(\rho).
     \end{align*}
\end{proposition}
\begin{proof}
We intend to apply Theorem \ref{thm:limiting_inner_product}, thus we shall take multiple steps to replace $K\ast\rho^{\eps_n}$ by a compactly supported, sufficiently smooth test function, which is independent of $n$. Throughout the proof, without loss of generality we assume that $\eps_n\le 1/\Csupp$ for all $n\in\bbN$, so that by \eqref{supp}
\begin{align}\label{eq:eps_le_Csupp}
    \forall n\in\bbN,\,\forall(x,y)\in G^{\eps_n}: \abs{x-y}\lor\abs{x-y}^2 = \abs{x-y}.
\end{align}
\subparagraph{\textbf{Step 1: Truncation}}

Let $\chi_R\in C^\infty_c(\Rd;[0,1])$ such that $\supp\chi_R\subset \bar B_{2R}$, $\chi_R|_{B_R} \equiv 1$, $\abs{\nabla\chi_R} \le 2/R$. First, we observe that for any $(x,y)\in G$ it holds
\begin{equation}\label{eq:bound_K-ast-rho}
\begin{aligned}
    \bra[\big]{\babla K\ast\rho^{\eps_n}}_-&\ge \bra[\big]{\babla K\ast\chi_R\rho^{\eps_n}}_- - \abs[\big]{\babla K\ast(1-\chi_R)\rho^{\eps_n}}\\
    &\ge\bra[\big]{\babla K\ast\chi_R\rho^{\eps_n}}_-
    - L_K\abs*{x-y}\sup_{n\in\bbN} \rho^{\eps_n}(B_R^c).
\end{aligned}
\end{equation}
Thus, due to the tightness of $(\rho^{\eps_n})_n$, there exists $\omega\in C([0,\infty);[0,\infty))$ with $\omega(R)\to 0$ as $R\to\infty$, such that
\begin{equation}\label{eq:D_cut-off}
\begin{aligned}
    \calD_{\!\eps_n}(\rho^{\eps_n}) &= \tA(\mu,\eta^{\eps_n};\rho^{\eps_n},-\babla K\ast\rho^{\eps_n})\\
    &\ge \tA(\mu,\eta^{\eps_n};\rho^{\eps_n},-\babla K\ast\chi_R\rho^{\eps_n})-\omega(R)\\
    &\ge \tA(\chi_R\mu,\eta^{\eps_n};\chi_R\rho^{\eps_n},-\babla K\ast\chi_R\rho^{\eps_n})-\omega(R)\\
    &\eqqcolon \tA(\chi_R\mu,\eta^{\eps_n};\chi_R\rho^{\eps_n},\babla\varphi_R^{\eps_n})-\omega(R),
\end{aligned}
\end{equation}
where the last inequality holds by the monotonicity of the integral. 

\subparagraph{\textbf{Step 2: Mollification}}

In order to apply Theorem \ref{thm:limiting_inner_product} later, we need further regularity. Given a mollifier $(\nu_\delta)_{\delta>0}$, for every $n\in\bbN$ we define $\varphi^{\eps_n}_{R,\delta}\coloneqq\nu_\delta\ast\varphi_R^{\eps_n}\in C_c^\infty(\Rd)$. Then, using \eqref{K3} and \eqref{eq:eps_le_Csupp} we calculate
\begin{align*}
    &\abs[\big]{\tA(\chi_R\mu,\eta^{\eps_n};\chi_R\rho^{\eps_n},\babla \varphi_R^{\eps_n}) - \tA(\chi_R\mu,\eta^{\eps_n};\chi_R\rho^{\eps_n},\babla\varphi^{\eps_n}_{R,\delta})}\\
    &= \abs[\bigg]{\iint_\Rddiag\pra[\Big]{\abs[\big]{\bra[\big]{\babla\varphi_R^{\eps_n}}_+}^2-\abs[\big]{\bra[\big]{\babla\varphi^{\eps_n}_{R,\delta}}_+}^2}(\chi_R\otimes\chi_R)\eta^{\eps_n}\dx\bra*{\rho^{\eps_n}\otimes\mu}}\\
    &\le 2 L_K\iint_\Rddiag\abs[\big]{\babla\bra[\big]{\varphi_R^{\eps_n}-\varphi^{\eps_n}_{R,\delta}}(x,y)}\abs{x-y}\chi_R(x)\chi_R(y)\eta^{\eps_n}(x,y)\dx\rho^{\eps_n}(x)\dx\mu(y)\\
    &\le 2 \Cint L_K\norm{\nabla\varphi_R^{\eps_n}-\nabla\varphi^{\eps_n}_{R,\delta}}_{L^\infty(B_{2R})}.
\end{align*} 
Next, we show that $\norm{\nabla\varphi_R^{\eps_n}-\nabla\varphi^{\eps_n}_{R,\delta}}_{L^\infty(B_{2R})}\to 0$ as $\delta\to 0$, uniformly in $n$. Indeed, for any $x\in B_{2R}$ it holds
\begin{align*}
    \abs[\big]{\nabla\varphi_R^{\eps_n}(x)-\nabla\varphi^{\eps_n}_{R,\delta}(x)} &=\abs*{\int_\Rd\bra*{\nabla\varphi_R^{\eps_n}(x)-\nabla\varphi_R^{\eps_n}(y)}\nu_\delta(x-y)\dx y}\\
    &\le\sup_{y\in B_\delta(x)}\abs*{\nabla\varphi_R^{\eps_n}(x)-\nabla\varphi_R^{\eps_n}(y)}\\
    &= \sup_{y\in B_\delta(x)}\abs*{\int_\Rd\bra*{ \nabla_1 K(y,z)-\nabla_1 K(x,z)}\chi_R(z)\dx\rho^{\eps_n}(z)}\\
    &\le\sup_{z\in B_{2R}}\sup_{y\in B_\delta(x)}\abs*{\nabla_1 K(y,z)-\nabla_1 K(x,z)}\\
    &\eqqcolon\omega_R(\delta)\overset{\delta\to 0}{\longrightarrow}0,
\end{align*}
where we used that the continuous function $\nabla_1 K$ is uniformly continuous on the compact set $B_{2R}$.

\subparagraph{\textbf{Step 3: $n$-independent test function}}

Next, we show that up to a negligible error, vanishing as $n\to \infty$, we can replace $\varphi_{R,\delta}^{\eps_n}$ by $\varphi_{R,\delta}\coloneqq -\nu_\delta\ast K\ast\chi_R\rho$. 
Indeed, we observe
\begin{align*}
    &\abs*{\bra*{\babla \varphi_{R,\delta}^{\eps_n}(x,y)}_+ - \bra*{\babla \varphi_{R,\delta}(x,y)}_+}\chi_R(x)\chi_R(y)\\
    &\le \abs*{\babla\bra*{\varphi_{R,\delta}^{\eps_n} - \varphi_{R,\delta}}(x,y)}\chi_R(x)\chi_R(y)\le  \norm{\nabla\varphi_{R,\delta}^{\eps_n} - \nabla\varphi_{R,\delta}}_{L^\infty(B_{2R})}\abs{x-y}.
\end{align*}
Hence, due to \eqref{K3}, keeping in mind \eqref{eq:eps_le_Csupp}, we obtain
\begin{align*}
    &\abs*{\tA(\chi_R\mu,\eta^{\eps_n};\chi_R\rho^{\eps_n},\babla\varphi_{R,\delta}^{\eps_n}) - \tA(\chi_R\mu,\eta^{\eps_n};\chi_R\rho^{\eps_n},\babla\varphi_{R,\delta})}\\
    &\le 2L_K\iint_\Rddiag\abs*{\bra*{(\babla \varphi_{R,\delta}^{\eps_n}(x,y))_+}-\bra*{(\babla \varphi_{R,\delta}(x,y))_+}}\\
    &\hphantom{\le 2L_K\iint_\Rddiag[(}\chi_R(x)\chi_R(y)\abs{x-y}\eta^{\eps_n}(x,y)\dx\rho^{\eps_n}(x)\dx\mu(y)\\
    &\le2L_K\Cint \norm{\nabla\varphi_{R,\delta}^{\eps_n}-\nabla\varphi_{R,\delta}}_{L^\infty(B_{2R})}.
\end{align*}
We conclude this step by showing that for any $\delta,R>0$, it holds
\begin{align}\label{eq:norm_phi^n_R->0}
    \limsup_{n\to\infty}\norm{\nabla\varphi_{R,\delta}^{\eps_n}-\nabla\varphi_{R,\delta}}_{L^\infty(B_{2R})} = 0.
\end{align}
To see this, we first observe that by the triangle inequality for any $N\in\bbN$ and $a,a_k\in\R$, $k=1,\dots,N$, we have
\begin{align*}
    \abs{a}\le \sum_{k=1}^N \abs{a_k} +\min_{k} \abs{a-a_k}.
\end{align*}
On the other hand, for any $\sigma>0$ there exists $N_\sigma\in\bbN$ and points $x^{(k)}$, $k=1,\dots,N_\sigma$, such that $\min_k\abs{x-x^{(k)}}\le \sigma$ for any $x\in B_{2R}$. We define $K_\delta\coloneqq\nu_\delta\ast K$ and denote by $\omega_{K_\delta,R}$ the modulus of continuity of $\nabla K_\delta$ on $B_{2R}\times B_{2R}$. Recalling $\supp\chi_R=B_{2R}$, we obtain
\begin{align*}
    &\norm{\nabla\varphi_{R,\delta}^{\eps_n}-\nabla\varphi_{R,\delta}}_{L^\infty(B_{2R})}\\
    &= \max_{1\le i\le d}\sup_{x\in B_{2R}}\abs*{\int_\Rd \partial_{x_i} K_\delta(x,z)\chi_R(z)\dx\bra*{\rho^{\eps_n}-\rho}(z)}\\
    &\le \max_{1\le i\le d}\sum_{k=1}^{N_\sigma} \abs*{\int_\Rd \partial_{x_i} K_\delta(x^{(k)},z)\chi_R(z)\dx\bra*{\rho^{\eps_n}-\rho}(z)}\\
    &\quad+ \max_{1\le i\le d}\sup_{x\in B_{2R}}\min_{k}\abs*{\int_\Rd \bra*{\partial_{x_i} K_\delta(x,z)-\partial_{x_i} K_\delta(x^{(k)},z)}\chi_R(z)\dx\bra*{\rho^{\eps_n}-\rho}(z)}\\
    &\le \max_{1\le i\le d}\sum_{k=1}^{N_\sigma} \abs*{\int_\Rd \partial_{x_i} K_\delta(x^{(k)},z)\chi_R(z)\dx\bra*{\rho^{\eps_n}-\rho}(z)} + 2\omega_{K_\delta,R}(\sigma).
\end{align*}
Since $N_\sigma$ is finite, the first sum vanishes as $n\to\infty$ due to the narrow convergence of $\rho^{\eps_n}$ towards $\rho$. Hence, since $\sigma>0$ is arbitrary, the claim is proved.

\subparagraph{\textbf{Step 4: Conclusion}}

Applying Theorem \ref{thm:limiting_inner_product} with $\chi_R\mu$ and $\chi_R\rho^{\eps_n}$, keeping in mind Proposition \ref{prop:tensor_uniqueness} on the uniqueness of the limit, we have 
\begin{align*}
    \lim_{n\to\infty}\tA(\chi_R\mu,\eta^{\eps_n};\chi_R\rho^{\eps_n},\babla\varphi_{R,\delta}) &= \int_\Rd \nabla \varphi_{R,\delta}\cdot\bbT_R\nabla\varphi_{R,\delta}\chi_R\dx\rho,
\end{align*}
where, recalling $\eta^{\eps_n}(x,y) = \frac{1}{\eps_n^{d+2}}\vartheta\bra*{\frac{x+y}{2},\frac{y-x}{\eps_n}}$, it holds
\begin{align*}
    \bbT_R(x) &= \!\lim_{n\to\infty}\frac12\int_{\Rdx} (x-y)\!\otimes\!(x-y)\,\eta^{\eps_n}(x,y)\chi_R(y)\dx\mu(y)\\
    &= \!\lim_{n\to\infty}\frac{1}{2\eps_n^d}\int_{\Rdx} \!\bra*{\frac{y-x}{\eps_n}}\!\otimes\!\bra*{\frac{y-x}{\eps_n}}\,\vartheta\bra*{x+\frac{\eps_n}{2}\frac{y-x}{\eps_n},\frac{y-x}{\eps_n}}\chi_R(y)\widetilde\mu(y)\dx y\\
    &= \frac12\chi_R(x)\widetilde\mu(x)\int_{\Rdzero} w\!\otimes\! w\,\vartheta(x,w)\dx w,
\end{align*}
which, as $R\to \infty$, converges pointwise to
\begin{align*}
    \bbT(x) &= \frac12\widetilde\mu(x)\int_{\Rdzero} w\!\otimes\! w\,\vartheta(x,w)\dx w.
\end{align*}
Regarding $\varphi_{R,\delta}$, we define $\varphi_R\coloneqq - K\ast\chi_R\rho$ and observe that for every $\delta_0>0$ the family $(\varphi_{R,\delta})_{0<\delta\le\delta_0}$ is supported on a compact set, so that we find
\begin{align*}
    \norm{\nabla\varphi_{R,\delta}-\nabla\varphi_R}_{L^\infty(B_{2R})}\overset{\delta\to 0}{\longrightarrow} 0.
\end{align*}
Since $\chi_R\to 1$ as $R\to\infty$ pointwise on $\Rd$, \eqref{K4} and the dominated convergence theorem give us for for every $x\in\Rd$
\begin{align*}
    \nabla\varphi_R(x) &= -(\nabla_1 K)\ast \rho\chi_R(x)\overset{R\to\infty}{\longrightarrow}-\nabla K\ast \rho(x).
\end{align*}
Thus, another application of the dominated convergence theorem yields
\begin{align*}
    \liminf_{n\to\infty}\calD_{\!\eps_n}(\rho^{\eps_n}) &\ge \liminf_{n\to\infty}\tA(\chi_R\mu,\eta^{\eps_n};\chi_R\rho^{\eps_n},\babla\varphi_R^{\eps_n})-\omega(R)\\
    &\ge \liminf_{n\to\infty} \tA(\chi_R\mu,\eta^{\eps_n};\chi_R\rho^{\eps_n},\babla\varphi_{R,\delta}^{\eps_n})-\omega_R(\delta) -\omega(R)\\
    &\ge \liminf_{n\to\infty} \tA(\chi_R\mu,\eta^{\eps_n};\chi_R\rho^{\eps_n},\babla\varphi_{R,\delta})-\omega_R(\delta) -\omega(R)\\
    &= \lim_{n\to\infty} \tA(\chi_R\mu,\eta^{\eps_n};\chi_R\rho^{\eps_n},\babla\varphi_{R,\delta}) -\omega_R(\delta) -\omega(R)\\
    &= \int_\Rd \nabla \varphi_{R,\delta}\cdot\bbT_R\nabla\varphi_{R,\delta}\chi_R\dx\rho -\omega_R(\delta)-\omega(R)\\
    &\overset{\delta\to0}{\longrightarrow}\int_\Rd \nabla \varphi_R\cdot\bbT_R\nabla\varphi_R\chi_R\dx\rho -\omega(R) \overset{R\to\infty}{\longrightarrow}\calD_{\;\!\bbT}(\rho). \qedhere
    \end{align*}
\end{proof}

\subsection{Wasserstein gradient flow structure for \texorpdfstring{\eqref{eq:intro:NLIE:one}}{(NLIE)}}\label{sec:gf_structure}

In order to prove Theorem~\ref{thm:main_result}, it remains to show that the local De Giorgi functional $\bs\calG_\bbT$ is non-negative. This is a consequence of a chain rule inequality proven in Proposition~\ref{prop:chain_rule_ineq}. Afterwards we prove Theorem~\ref{thm:weak_sol_nlie_tensor} by establishing convergence of gradient flows of $\calE$ in the Finslerian spaces $(\calP_2(\R^d),\calT_\eps)$ towards gradient flows of the same energy, $\calE$, in the Riemannian space $(\calP_2(\R^d_\bbT),W_{\bbT})$ as $\eps\to 0$, using the concept of curves of maximal slope in the corresponding spaces, following~\cite{AmbrosioGigliSavare2008}.
\begin{proposition}(Chain-rule inequality)\label{prop:chain_rule_ineq}
Assume $K$ satisfies \eqref{K1}~--~\eqref{K4}. For any $\bs\rho\in \AC^2([0,T];(\calP_2(\R^d_\bbT),W_\bbT))$ the following chain rule inequality holds
\[
\bs\calG_\bbT(\bs\rho)=\calE(\rho_T)-\calE(\rho_0)+\frac{1}{2}\int_0^T\left(\calD_{\;\!\bbT}(\rho_t)+|\rho'_t|^2_\bbT\right)\dx t\ge0.
\]
\end{proposition}
\begin{proof}
Let us remind the reader the tensor $\bbT$ is continuous, symmetric, and uniformly elliptic, cf.~Proposition~\ref{prop:tensor_uniqueness}. According to \cite[Theorem 2.4]{Lisini_ESAIM2009}, if $\bs\rho\in\AC^2([0,T];(\calP_2(\R^d_\bbT),W_\bbT)$, there exists a unique vector field $\bs u:[0,T]\times\Rd\to\Rd$ such that $(\bs\rho,\bs\rho\bs u)\in\CE_T$ and 
\[
\abs{\rho_t'}_\bbT=\|u_t\|_{L^2(\rho_t;\R^d_\bbT)} \qquad \mbox{a.e. } t\in[0,T].
\]
The uniform ellipticity of $\bbT$ (Proposition~\ref{prop:tensor_uniqueness}) implies 
\[
\frac{1}{C}\int_0^T\int_\Rd\abs*{u_t}^2\dx\rho_t\dx t\le \int_0^T\int_\Rd \skp{\bbT^{-1}u_t,u_t}\dx\rho_t\dx t=\int_0^T\abs{\rho_t'}_\bbT^2\dx t<\infty,
\]
by assumption. Since $(\bs\rho,\bs \rho\bs u)\in\CE_T$ such that $\abs{\bs u}_{L^2(\bs\rho;\Rd)}\in L^1(0,T)$, due to the previous inequality, \cite[Theorem 8.3.1]{AmbrosioGigliSavare2008} implies $W_2$-absolute continuity of the curve $\bs\rho = (\rho_t)_{t\in[0,T]}\subset \calP_2(\Rd)$. 

Due to Proposition \ref{prop:chain_rule_gen} and the fact that $\bbT$ is symmetric and uniformly elliptic, for any $0\le s\le t\le T$, we obtain the result from
\begin{align*}
  \MoveEqLeft\calE(\rho_t)-\calE(\rho_s)=\int_s^t\int_\Rd\nabla K*\rho_\tau u_\tau\dx \rho_\tau\dx\tau\\
  &=\int_s^t\int_\Rd \skp{\bbT^{1/2}\nabla K*\rho_\tau,\bbT^{-1/2}u_\tau}\dx\rho_\tau\dx\tau\\
  &\ge -\frac{1}{2}\int_s^t \pra*{ \int_\Rd \skp{\nabla K*\rho_\tau,\bbT\nabla K*\rho_\tau}\dx\rho_\tau \int_\Rd \skp{\bbT^{-1} u_\tau,u_\tau}\dx\rho_\tau }\dx\tau\\
  &=-\frac{1}{2}\int_s^t\left(\calD_{\;\!\bbT}(\rho_\tau)+|\rho'_\tau|^2_\bbT\right)\dx\tau .\qedhere
\end{align*}
\end{proof}
We now combine the lower limits of Section~\ref{ssec:lsc} with the chain-rule inequality from Proposition~\ref{prop:chain_rule_ineq} to prove Theorem~\ref{thm:main_result}.
\begin{proof}[Proof of Theorem~\ref{thm:main_result}]
Continuity of the energy with respect to narrow convergence,~\cite[Proposition 3.3]{EPSS2021}, and the lower limits for the metric derivatives and the slopes, Propositions~\ref{prop:lower_limit_metric_derivatives} and~\ref{prop:D_NL->L_lsc}, imply
\begin{align*}
    0=\liminf_{\eps\to0}\bs\calG_{\!\!\eps}(\bs\rho^\eps)\ge \bs\calG_\bbT(\bs\rho).  
\end{align*}
With this, the chain-rule inequality proven in Proposition~\ref{prop:chain_rule_ineq} ensures $\bs\calG_\bbT(\bs\rho)=0$.
\end{proof}

Next, we want to establish the connection between the zero level sets of $\bs\calG_\bbT$ and weak solutions of \eqref{eq:intro:NLIE:one}. To this end, we first show that $\sqrt{\calD_{\;\!\bbT}}$ is a strong upper-gradient for $\calE$.

\begin{corollary}[Strong upper gradient]\label{cor:strong_up_gr_diss}
Assume $K$ satisfies \eqref{K1}~--~\eqref{K4}. For any $\bs\rho\in \AC^2([0,T];(\calP_2(\R^d_\bbT),W_\bbT))$ it holds $\sqrt{\calD_{\;\!\bbT}}$ is a strong upper gradient for $\calE$ in the sense of Definition~\ref{def:strong_ug}, i.e. 
\begin{equation*}%
|\calE(\rho_t)-\calE(\rho_s))|\leq \int_s^t \sqrt{\calD_{\;\!\bbT}(\rho_\tau)}\abs{\rho_\tau'} \dx \tau, \quad \forall\, 0<s\leq t<T. 
\end{equation*}
\end{corollary}
\begin{proof}
Arguing as in~Propositon~\ref{prop:chain_rule_ineq}, we infer there exists a unique vector field $\bs u:[0,T]\times\Rd\to\Rd$ such that $(\bs \rho,\bs\rho \bs u)\in\CE_T$ with $\left\|\bs u\right\|_{L^2(\bs\rho;\Rd)}\in L^1(0,T)$, hence the curve $(\rho_t)_{t\in[0,T]}\subset \calP_2(\Rd)$ is $W_2$-absolutely continuous. Therefore, by applying Proposition~\ref{prop:chain_rule_gen}, we infer what is claimed, i.e.
\begin{equation*}
\begin{split}
  \MoveEqLeft\left|\calE(\rho_t)-\calE(\rho_s)\right|=\biggl|\int_s^t\int_\Rd\nabla K*\rho_\tau u_\tau\dx \rho_\tau\dx\tau\biggr|\\
  &\le\int_s^t\int_\Rd \abs[\big]{\skp[\big]{\bbT^{1/2}\nabla K*\rho_\tau,\bbT^{-1/2}u_\tau}}\dx\rho_\tau\dx\tau\\
  &\le \int_s^t\int_\Rd \abs[\big]{\skp[\big]{\bbT\nabla K*\rho_\tau,\nabla K*\rho_\tau}}^{\frac{1}{2}}\abs[\big]{\skp[\big]{ u_\tau,\bbT^{-1}u_\tau } }^{\frac{1}{2}} \dx\rho_\tau\dx\tau\\
  &\le\int_s^t\sqrt{\calD_{\;\!\bbT}(\rho_\tau)}\; \abs{\rho_\tau'}\dx\tau. \qedhere
\end{split}
\end{equation*}
\end{proof}
We are now able to identify weak solutions of \eqref{eq:intro:NLIE:one} as curves of maximal slope of $\calE$ with respect to the strong upper gradient $\sqrt{\calD_\bbT}$, thus gradient flows in $(\calP_2(\R^d_{\bbT}),W_{\bbT})$.
\begin{theorem}\label{thm:weak_curves_max_slope}
Assume $K$ satisfies \eqref{K1}~--~\eqref{K4}. A curve $\bs \rho \subset \calP_2(\Rd)$ is a weak solution of \eqref{eq:intro:NLIE:one} if and only if it is a curve of maximal slope for $\calE$ with respect to its strong upper gradient $\sqrt{\calD_\bbT}$, i.e. $\bs\rho\in \AC^2([0,T];(\calP_2(\R^d_\bbT),W_\bbT))$ and $\bs\calG_\bbT(\bs\rho)=0$.
\end{theorem}
\begin{proof}
According to Definition~\ref{def:weak-nlie_tensor}, for a weak solution of \eqref{eq:intro:NLIE:one} we have
\[
\dx j_t(x)=-\bbT\nabla\frac{\delta \calE}{\delta\rho}(x)\dx\rho_t(x).
\]
By again setting $u_t\coloneqq-\bbT\nabla\frac{\delta\calE}{\delta\rho_t}$ the density of $j_t$ with respect to $\rho_t$, we notice
\[
\|u_t\|_{L^2(\rho_t;\R^d_\bbT)}=\calD_{\;\!\bbT}(\rho_t)=\int_\Rd\skp[\bigg]{\nabla\frac{\delta\calE}{\delta\rho_t},\bbT\nabla\frac{\delta\calE}{\delta\rho_t}}\dx\rho_t<\infty.
\]
The metric slope is uniformly bounded as follows as consequence of the uniform ellipticity of $\bbT$ and Assumption~\eqref{K4}:
\begin{align*}
    \calD_{\;\!\bbT}(\rho_t)&=\int_\Rd\skp[\bigg]{\nabla\frac{\delta\calE}{\delta\rho_t},\bbT\nabla\frac{\delta\calE}{\delta\rho_t}}\dx\rho_t\\
    &\le C\iint_{\R^{2d}}\left|\nabla K(x,y)\right|^2\dx\rho_t(y)\dx\rho_t(x)\\
    &\le C\iint_{\R^{2d}}\left(1+|x|+|y|\right)^2\dx\rho_t(y)\dx\rho_t(x)\\
    &\le \tilde C\left(1+M_2(\rho_t)\right)\le \bar C(\rho_0,T).
\end{align*}
The uniform bound on the second order moments of $\rho_t$ can be proven by a standard procedure, which we include for completeness. Upon considering a smooth cut-off function, by Remark~\ref{rem:ce_Lip_test-function} we have
\begin{align*}
    \frac{\dx}{\dx t}\int_\Rd\bra[\big]{1+ |x|^2}\dx\rho_t(x)&=-2\int_\Rd\skp[\bigg]{x,\bbT\nabla\frac{\delta\calE}{\delta\rho_t}}\dx\rho_t\\
    &\le2\left(\int_\Rd\langle\bbT^{-1} x,x\rangle\dx\rho_t\right)^{\frac{1}{2}}\left(\int_\Rd\skp[\bigg]{\nabla\frac{\delta\calE}{\delta\rho_t},\bbT\nabla\frac{\delta\calE}{\delta\rho_t}}\dx\rho_t\right)^{\frac{1}{2}}\\
    &\le C \int_\Rd|x|^2\dx\rho_t(x)+\calD_{\;\!\bbT}(\rho_t)\\
    &\le \tilde C \int_\Rd\bra[\big]{1 + |x|^2}\dx\rho_t(x).
\end{align*}
Gronwall's inequality provides propagation of second order moments for weak solutions of~\eqref{eq:intro:NLIE:one}.

We have $\bs\rho\in \AC^2([0,T];(\calP_2(\R^d_\bbT),W_\bbT))$ by the uniform bound on the metric slope~$\calD_{\;\!\bbT}$. 
Since $(0,T)\ni t \mapsto \|u_t\|_{L^2(\rho_t;\R^d_\bbT)}\in L^2(0,T)$ and by arguing as in the proof of Proposition~\ref{prop:lower_limit_metric_derivatives}, we obtain $|\rho'_t|_\bbT\le\left\|u_t\right\|_{L^2(\rho_t;\R^d_\bbT)}$. In turn, uniform ellipticity of $\bbT$ gives $\|u_t\|_{L^2(\rho_t;\Rd)}\in L^2(0,T)$, thus $W_2$-absolute continuity of $\bs\rho$, due to \cite[Theorem 8.3.1]{AmbrosioGigliSavare2008}. Exploiting the chain rule from Proposition~\ref{prop:chain_rule_gen}, we obtain, for any $0\le s\le t\le T$, 
\begin{align*}
    \calE(\rho_t)-\calE(\rho_s)=-\int_s^t \calD_\bbT(\rho_\tau)\dx \tau\le-\int_s^t \sqrt{\calD_{\;\!\bbT}(\rho_\tau)}\abs{\rho_\tau'} \dx \tau.
\end{align*}
Being $\sqrt{\calD_{\;\!\bbT}}$ a strong upper gradient by Corollary \ref{cor:strong_up_gr_diss}, we infer $\bs\calG_\bbT(\bs\rho)=0$.

Let $\bs\rho$ be a curve of maximal slope for $\calE$ with respect to its strong upper gradient, $\sqrt{\calD_{\;\!\bbT}}$, cf.~Definition~\ref{def:maximal_slope_local}. More precisely, $\bs\rho\in\AC^2([0,T];(\calP_2(\R^d_\bbT),W_\bbT))$ such that $\bs\calG_\bbT(\bs \rho)=0$. Arguing as in previous proofs, there exists a unique vector field ${\bs u}:[0,T]\times\Rd\to\Rd$ such that $(\bs\rho,\bs\rho\bs u)\in\CE_T$ and $
\abs{\rho_t'}_\bbT=\|u_t\|_{L^2(\rho_t;\R^d_\bbT)}$ for a.e. $t\in[0,T]$. Furthermore, uniform ellipticity of $\bbT$ implies $W_2$-absolute continuity. Proposition~\ref{prop:chain_rule_gen} gives, for any $0\le s \le t\le T$,
\begin{align*}
    \MoveEqLeft \calE(\rho_t)-\calE(\rho_s)=\int_s^t\int_\Rd\nabla K*\rho_\tau u_\tau \dx \rho_\tau \dx\tau\\
    &=\int_s^t\int_\Rd \bigl\langle\bbT^{1/2}\nabla K*\rho_\tau,\bbT^{-1/2}u_\tau\bigr\rangle\dx\rho_\tau\dx\tau\\
    &\ge-\int_s^t\int_\Rd \bigl|\bigl\langle\bbT\nabla K*\rho_\tau,\nabla K*\rho_\tau\bigr\rangle\bigr|^{\frac12}\bigl|\bigr\langle u_\tau,\bbT^{-1}u_\tau\bigr\rangle\bigl|^{\frac12} \dx\rho_\tau\dx\tau\\
    &\ge -\frac{1}{2}\int_s^t\!\!\int_\Rd \bigl\langle\bbT\nabla K*\rho_\tau,\nabla K*\rho_\tau\bigr\rangle\dx\rho_\tau\dx\tau -\frac{1}{2}\int_s^t\!\!\int_\Rd \bigl\langle u_\tau,\bbT^{-1}u_\tau\bigr\rangle\dx\rho_\tau\dx\tau\\
  &=-\frac{1}{2}\int_s^t\left(\calD_{\;\!\bbT}(\rho_\tau)+|\rho'_\tau|^2_\bbT\right)\dx\tau.
\end{align*}
Since $\bs\calG_\bbT(\bs\rho)=0$, the inequalities above are equalities, which is true if and only if $u_t(x)=-\bbT\nabla K*\rho_t(x)$ for $\rho_t$-a.e. $x$ and a.e. $t\in[0,T]$. In particular, $(\bs\rho, \bs j)\in\CE_T$ for $\bs j=-\bs\rho \bbT \nabla*\bs \rho$.
\end{proof}

Finally, we conclude with the proofs of~Theorem~\ref{thm:weak_sol_nlie_tensor} and Theorem~\ref{thm:finite_graph}.

\begin{proof}[Proof of Theorem~\ref{thm:weak_sol_nlie_tensor}] 
The proof is a graph-approximation of~\eqref{eq:intro:NLIE:one}. Consider $(\mu,\vartheta)$ satisfying \eqref{mu1}, \eqref{mu2} and \eqref{theta1}~--~\eqref{theta4}, and let $\eta^\eps$ be given by \eqref{eq:def:eta^eps}. Theorem~\ref{thm:extension_sigma_finite} provides a sequence $(\bs\rho^\eps)_\eps$ of weak solutions to~\eqref{eq:nlnl-interaction-eq}, which are curves of maximal slope for $\calE$ in $(\calP_2(\Rd),\calT_{\mu,\eta^\eps})$, hence such that $\bs\calG_{\!\!\eps}(\bs\rho^\eps)=0$, for any $\eps>0$, cf.~\cite[Theorem 3.9]{EPSS2021}. The graph-to-local limit proven previous to this subsection, Theorem~\ref{thm:main_result}, implies existence of a limiting curve $\bs\rho \in \AC^2([0,T];(\calP_2(\R^d_\bbT),W_{\bbT}))$ which is $\bs \rho$ is a gradient flow of $\calE$ in $(\calP_2(\R^d_\bbT),W_{\bbT})$, i.e. $\bs\calG_\bbT(\bs\rho) = 0$. In particular, Theorem~\ref{thm:weak_curves_max_slope} asserts this is a weak solution of~\eqref{eq:intro:NLIE:one}.
\end{proof}

\begin{proof}[Proof of Theorem~\ref{thm:finite_graph}]
    Let $\bs\rho$ be as in Theorem~\ref{thm:weak_sol_nlie_tensor} with initial datum $\varrho_0\in\calP_2(\Rd)$, and let $(\bs\rho^\eps)_{\eps>0}$ be the approximating sequence of graph gradient flows from Theorem~\ref{thm:main_result} with initial data $\varrho^\eps_0\in\calP_2(\Rd)$. In particular, we have (along a subsequence) for all $t\in[0,T]$ 
    \begin{align*}
        \rho_t^{\eps} \rightharpoonup \rho_t \text{ narrowly as }\eps\to 0,
    \end{align*}
    and
    \begin{align*}%
    	\sup_{\eps>0} M_2(\varrho^{\eps}_0) <\infty
    \end{align*}
    On the other hand, by the proof of Theorem~\ref{thm:extension_sigma_finite}, for any $\eps>0$ there exists a sequence $(\mu^{\eps,n})_n$ of finite base measures such that $\mu^{\eps,n}\rightharpoonup^\ast\mu^\eps$ weakly-$^\ast$ as $n\to\infty$, corresponding initial data $\varrho^{\eps,n}_0\in\calP_2(\Rd)$, and gradient flows $\bs\rho^{\eps,n}$. From the same proof, we can further conclude that for all $\eps>0$ and all $t\in[0,T]$ (up to an unrenamed subsequence) it holds
    \begin{align*}
        \rho_t^{\eps,n} \rightharpoonup \rho_t^{\eps} \text{ narrowly as }n\to \infty,
    \end{align*}
    and
    \begin{align*}
    	\sup_{\eps>0}\sup_{n\in\bbN} M_2(\varrho^{\eps,n}_0) <\infty.
    \end{align*}
Observe that by \cite[Lemma~2.16]{EPSS2021}, this implies a uniform moment bound also in time, i.e. 
    \begin{align}\label{eq:uni_mom_bound_M2}
    	\sup_{t\in[0,T]}\sup_{\eps>0}\sup_{n\in\bbN}  M_2(\rho^{\eps,n}_t)< \infty.
    \end{align}
    Indeed, keeping in mind that \eqref{int} does in fact not depend on $\eps$, this is a consequence of the metric slopes being uniformly bounded with respect to $\eps$ and $n$ for any finite time horizon $T$.
	
   From the uniform bound \eqref{eq:uni_mom_bound_M2}, we deduce by \cite[Proposition~7.1.5]{AmbrosioGigliSavare2008} that narrow convergence is equivalent to convergence with respect to the $W_1$ metric. 
    
    Next, we observe that by Proposition~\cite[Proposition~2.21]{EPSS2021} (where the constant is again independent of $\eps$ by \eqref{int}) for every $\eps>0$, $n\in\bbN$, and $0\le s\le t\le T$ it holds
    \begin{align*}
        W_1(\rho^{\eps,n}_s,\rho^{\eps,n}_t) &\le \calT_{\mu^n,\eta^\eps}(\rho^{\eps,n}_s,\rho^{\eps,n}_t) \le \int_s^t \abs{(\rho^{\eps,n}_\tau)'}_{\eps,n}\dx\tau\\
        &\le \sqrt{t-s}\sup_{n\in\bbN}\sqrt{\int_0^T \abs{(\rho^{\eps,n}_\tau)'}^2_{\eps,n}\dx\tau},
    \end{align*}
    where $\sup_{\eps>0}\sup_{n\in\bbN}\int_0^T \abs{(\rho^{\eps,n}_\tau)'}^2_{\eps,n}\dx\tau <\infty$ by \cite[Lemma~3.13]{EPSS2021}. Indeed, the De Giorgi functionals are uniformly bounded by assumption and a uniform bound on the energies, also with respect to time, is a consequence of \eqref{eq:uni_mom_bound_M2}. The same argument yields for every $\eps>0$ and $0\le s\le t\le T$
    \begin{align*}
        W_1(\rho^\eps_s,\rho^\eps_t) &\le \sqrt{t-s}\sup_{n\in\bbN}\sqrt{\int_0^T \abs{(\rho^\eps_\tau)'}^2_{\eps,n}}\dx\tau,
    \end{align*}
    where $\sup_{\eps>0}\int_0^T \abs{(\rho^\eps_\tau)'}^2_{\eps,n}\dx\tau <\infty$.
    In particular, the set $\set{\bs\rho^{\eps,n}:n\in\bbN,\eps>0}\cup\set{\bs\rho^\eps:\eps>0}$ of maps $\bs\rho:[0,T]\to((\calP(\Rd),W_1))^N$ is uniformly equicontinuous, so that the above pointwise in time convergence results are in fact uniform in time. 
    This allows us to conclude with the following diagonal argument: For every $k\in\bbN$ there exists $\eps_k>0$ and $n_k\in\bbN$ such that for every $t\in[0,T]$ we have
    \begin{align*}
        W_1(\bs\rho_t^{\eps_k},\bs\rho_t)+ W_1(\bs\rho_t^{\eps_k,n_k},\bs\rho_t^{\eps_k})\le \frac{1}{k}.
    \end{align*}
    Consequently, the sequence $(\bar{\bs\rho}^k)_k$ defined by $\bar{\bs\rho}^k \coloneqq \bs\rho_t^{\eps_k,n_k}$ satisfies the claimed conditions.
\end{proof}

\appendix

\section{Chain-rule for the nonlocal interaction energy}
Below we provide a general chain-rule result for the energy \eqref{eq:interaction_energy} using $\CE$, as alternative to \cite[Section 10]{AmbrosioGigliSavare2008}, where convexity is required.

\begin{proposition}[Chain-rule continuity equation]\label{prop:chain_rule_gen}
Let $K$ satisfy assumption \eqref{K1}~--~\eqref{K4} and $\bs\rho\in\AC^2([0,T];(\calP_2(\Rd), W_2))$. Then, for any $0\le s\le t\le T$, the following chain-rule holds
\begin{equation}\label{eq:chain_rule_gen}
    \calE(\rho_t)-\calE(\rho_s)=\int_s^t\int_\Rd\nabla K*\rho_\tau(x)\dx j_\tau(x),
\end{equation}
being $(j_t)_{t\in[0,T]}\subset \calM(\Rd;\Rd)$ such that $(\bs\rho,\bs j)\in\CE_T$.
\end{proposition}
\begin{proof}
Assumption \eqref{K3} ensures $\calE(\rho_t)<\infty$ for all $t\in[0,T]$, since $\bs\rho\subset \calP_2(\Rd)$. Having $\bs\rho\in\AC^2([0,T];(\calP_2(\Rd), W_2))$ implies there exists a unique Borel vector field $\bs u\in L^2([0,T];L^2(\rho_t;\Rd))$ such that $(\bs\rho,\bs\rho\bs u)\in\CE_T$ as well as $\abs{\rho_t'} = \norm{u_t}_{L^2(\rho_t;\Rd)}$ for $\scrL^1$-a.e. $t\in[0,T]$, cf.~\cite[Theorem 8.3.1, Proposition 8.4.5]{AmbrosioGigliSavare2008}. For consistency with our notation we denote $\bs j\coloneqq\bs\rho\bs u$ and $u_t\coloneqq\frac{\dx j_t}{\dx \rho_t}$ for a.e. $t\in[0,T]$. We follow the procedure in, e.g., \cite[Proposition 3.10]{EPSS2021}, applying two regularization arguments. First, for all $(x,y)\in \Rd\times\Rd$ we define $K^\eps(x,y)=K*m_\eps(x,y)=\iint_{\Rd\times\Rd} K(z,z')m_{\eps}(x-z,y-z')\dx z\dx z'$, where $m_\eps(z)=\frac{1}{\eps^{2d}}m(\frac{z}{\eps})$ for all $z\in\R^{2d}$ and $\eps>0$, being $m\in C_c^\infty(\R^{2d})$ a standard mollifier. We also introduce a smooth cut-off function $\varphi_R\in C_c^\infty(\R^{2d})$, which is such that $\varphi_R(z)=1$ on $B_R$, $\varphi_R(z)=0$ on $\R^{2d}\setminus B_{2R}$ and $\abs{\nabla\varphi_R}\le\frac{2}{R}$. We set $K_R^\eps\coloneqq\varphi_R K^\eps$ and note that it is a~$C_\mathrm{c}^\infty(\R^{2d})$ function. We now introduce the approximate energies, indexed by $\eps$ and $R$,
\[
\calE_R^\eps(\nu)=\frac{1}{2}\int_\Rd\int_{\Rd} K_R^\eps(x,y)\dx\nu(y)\dx\nu(x) \quad \mbox{for all $\nu\in\calP_2(\Rd)$}.
\]
Next, we extend $\bs\rho$ and $\bs j$ to $[-T,2 T]$ periodically in time by setting $\rho_{-s}\coloneqq\rho_{T-s}$ and $\rho_{T+s}\coloneqq\rho_{s}$ for all $s\in (0,T]$ and likewise for $\bs j$. We regularize  $\bs\rho$ and $\bs j$ in time by using a standard mollifier $n\in C_c^\infty(\R)$, supported on $[-1,1]$, by setting $n_\sigma(t)=\frac{1}{\sigma}n(\frac{t}{\sigma})$ and 
\begin{align*}
    &\rho_t^\sigma(A)=n_\sigma*\rho_t(A)=\int_{-\sigma}^\sigma n_\sigma(t-s)\rho_s(A)\dx s, \qquad \forall A\subseteq\Rd,\\
    &j_t^\sigma(A)=n_\sigma*j_t(A)=\int_{-\sigma}^\sigma n_\sigma(t-s)j_s(A)\dx s, \qquad \forall A\subseteq \Rd,
\end{align*}
for any $\sigma\in(0,T)$ and any $t\in[0,T]$; whence $\rho_t^\sigma\in\calP_2(\Rd)$ for all $t\in[0,T]$ and it is straightforward to check that $(\bs\rho^\sigma,\bs j^\sigma)\in\CE_T$. Furthermore, we observe
\begin{align*}
\int_0^T\int_\Rd\left|\frac{\dx j_t}{\dx\rho_t}\right|^2\dx\rho_t\dx t=\int_0^T\int_\Rd\alpha\left(\frac{\dx\rho_t}{\dx|\lambda|},\frac{\dx j_t}{\dx|\lambda|}\right)\dx |\lambda|\dx t,
\end{align*}
where $|\lambda|\in\calM^+(\Rd)$ is such that $\rho_t,j_t\ll|\lambda|$ for a.e. $t\in[0,T]$, where
\begin{align*}
    \alpha(a,b)\coloneqq\begin{cases}
    \frac{|b|^2}{a} \qquad &\mbox{if } a>0,\\
    0 &\mbox{if } a=0 \mbox{ and } b\ne 0,\\
    +\infty &\mbox{if } a=0 \mbox{ and } b\ne 0,
    \end{cases}
\end{align*}
is lower semicontinuous, jointly convex, and one-homogeneous, see \cite{DNS09_CVPDE}. Arguing as in \cite[Proposition 3.10]{EPSS2021}, Jensen's inequality and Fubini's Theorem ensure
\begin{align*}
\int_0^T\int_\Rd\left|\frac{\dx j_t^\sigma}{\dx\rho_t^\sigma}\right|^2\dx\rho_t^\sigma\dx t \le c \int_0^T\int_\Rd\left|\frac{\dx j_t}{\dx\rho_t}\right|^2\dx\rho_t\dx t,
\end{align*}
for a constant $c>0$. In view of the regularity we have for $\eps>0$ and $\sigma>0$, we compute
\begin{align*}
\frac{\dx}{\dx t}\calE_R^\eps(\rho_t^\sigma)=\int_\Rd\nabla K_R^\eps*\rho_t^\sigma(x)\dx j_t^\sigma(x),
\end{align*}
whence, by integrating in time for $0\le s\le t\le T$, it holds
\begin{equation}\label{eq:chain_rule_reg}
    \calE_R^\eps(\rho_t^\sigma)-\calE_R^\eps(\rho_s^\sigma)=\int_s^t\int_\Rd\nabla K_R^\eps*\rho_\tau^\sigma(x)\dx j_\tau^\sigma(x).
\end{equation}
The proof will be completed once we let $\eps$ and $\sigma$ to $0$ and $R$ to $\infty$ in the equality above. In this regard, we first note that \cite[Corollary 4.10]{DNS09_CVPDE} gives $\rho_t^\sigma\rightharpoonup\widetilde{\rho}_t$ weakly-$^\ast$ in $\calM_{loc}^+(\Rd)$, for any $t\in[0,T]$, and $\bs j^\sigma\rightharpoonup \widetilde{\bs\jmath}$ in $\calM_{loc}^+(\Rd\times[0,T];\Rd)$, being $(\widetilde{\bs\rho},\widetilde{\bs\jmath})\in\CE_T$. On the other hand, since $n_\sigma\rightharpoonup\delta_0$ weakly-$^\ast$, as $\sigma\to0$, we have $\rho_t^\sigma\rightharpoonup\rho_t$ weakly-$^\ast$ in $\calM_{loc}^+(\Rd)$ and $j_t^\sigma\rightharpoonup j_t$ weakly-$^\ast$ in $\calM_{loc}^+(\Rd;\Rd)$, so that $\tilde{\bs\rho}=\bs\rho$ and $\widetilde{\bs\jmath}=\bs j$ by uniqueness of the limit. Since $K\in C^1(\Rd\times\Rd)$, it is well known that $K^\eps\to K$ uniformly on compact sets as $\eps\to0$. In particular, on both sides of \eqref{eq:chain_rule_reg} we can send $\sigma\to0$ and $\eps\to0$. By further using Lebesgue dominated convergence theorem on the $R\to\infty$ limit --- exploiting the moment control established in Remark \ref{rem:second_moment} in conjunction with Assumption \eqref{K4} ---  we obtain the desired result, i.e. \eqref{eq:chain_rule_gen}.
\end{proof}

\section{Extension \texorpdfstring{of \cite{EPSS2021}}{} to regular \texorpdfstring{$\sigma$}{σ}-finite base measures}\label{app:extension}
Here we extend the results from \cite{EPSS2021}, shown for finite base measures $\mu$, to $\sigma$-finite base measures $\mu$ satisfying \eqref{mu1}, \eqref{mu2}. For the readers convenience, we remind the relevant assumptions in~\cite{EPSS2021} for the pair~$(\mu,\eta)$:
\begin{equation}
    \label{eq:eta_cont_sym}
    \begin{cases}
        \eta \text{ is continuous on the set } \set{\eta>0},\\
        \eta(x,y)=\eta(y,x) \text{ for all }x,y\in\Rd;\\
\end{cases}
\end{equation}
the moment bounds
\begin{align}
	\label{eq:2lor4-mom_bound}
	\sup_{(x,y)\in \Rddiag}|x-y|^2\lor|x-y|^4\eta(x,y)&\le C_\eta,\\
    \label{eq:2lor4-int_bound}
    \sup_{x\in \Rd}\int_{\Rdx}|x-y|^2\lor|x-y|^4\eta(x,y)\dx\mu(y)&\le C_\eta^\mu,
\end{align}
and the local blow-up control
\begin{align}
	\label{eq:blow-up_control}
	\lim_{\delta\rightarrow 0}\sup_{x\in\Rd} \int_{B_\delta(x)\setminus\set*{x}}\abs{x-y}^2\eta(x,y)\dx{\mu}(y) = 0.
\end{align}
Assumption~\eqref{eq:2lor4-mom_bound} is required in~\cite[Theorem 3.15]{EPSS2021} to show existence of weak solutions to~\eqref{eq:nlnl-interaction-eq}. Instead of $\mu(\Rd)<\infty$, we assume \eqref{mu1}, \eqref{mu2} and that $G_{\!\!\:x}= \set{y\in\Rd\!\setminus\!\set*{x}:\eta(x,y)>0}\subset \Rd$ is a bounded set.
Note that by Lemma~\ref{lem:properties_eta-mu} and Remark \ref{rem:link_to_EPSS} all the above conditions are satisfied under the assumptions \eqref{mu1}, \eqref{mu2} and \eqref{theta1}~--~\eqref{theta4}.

In order to reprove the results from \cite{EPSS2021} under the altered assumption on $\mu$, we need to ensure that the lower semicontinuity of the De Giorgi functional \cite[Lemma~3.13]{EPSS2021} holds true when the sequence $(\mu^n)_n\subset\calM^+(G)$ converges weakly-$^\ast$ to $\mu\in\calM^+(\Rd)$, instead of narrowly, and it satisfies the above assumptions uniformly in $n$. We notice this is already pointed out in~\cite[Proposition 2.22]{HPS21}, namely the action $\A(\mu,\eta;\rho,j)$ is in fact weakly-$^\ast$ lower semicontinuous with respect to the base measure $\mu$ (see e.g. \cite[Theorem 2.34]{Ambrosio_Fusco_Pallara} or \cite[Theorem 3.4.3]{buttazzo1989semicontinuity}). Due to the identity $\abs{\rho'}_{\mu,\eta}^2 = \A(\mu,\eta;\rho,j)$, for a specific $j\in\calM(\Rddiag)$ (see \cite[Proposition~2.25]{EPSS2021} or \cite[Proposition~2.31]{HPS21}), weak-$^\ast$ lower semicontinuity is inherited by metric derivatives. 
Moreover, for the sake of completeness we also mention compactness for weak solutions of the continuity equation still holds true as weak-$^\ast$ convergence is only needed when showing lower semicontinuity of the total action, see~\cite[Proposition 2.17]{EPSS2021}.

It remains to show lower semicontinuity of the metric slope with respect to weak-$^\ast$ convergence for the base measure $\mu$.
\begin{proposition}[Weak-$^\ast$ lower semicontinuity of the metric slope]
Let $K$ satisfy \eqref{K2}, \eqref{K3} and $\eta:\Rddiag\to[0,\infty)$ satisfy \eqref{eq:eta_cont_sym}. Assume that $(\mu_n)_{n\in\mathbb{N}}\subset \calM^+(\Rd)$ is such that~\eqref{eq:2lor4-int_bound} and~\eqref{eq:blow-up_control} hold uniformly in $n$ and $\mu^n\oset{\ast}{\rightharpoonup}\mu$, as $n\to\infty$. Let $(\rho_n)_{n\in\bbN}\subset\calP_2(\Rd)$ such that $\rho^n\rightharpoonup\rho$, as $n\to\infty$ for some $\rho\in\calP_2(\Rd)$. Then, the metric slope is weak-$^\ast$ lower semicontinuous
\begin{align*}
 	\liminf_{n\to\infty}\calD(\mu^n,\eta;\rho^n)\ge \calD(\mu,\eta;\rho).
\end{align*}
\end{proposition}
\begin{proof}
We argue similarly to the proof of Proposition \ref{prop:D_NL->L_lsc}.
Let $R>0$. We set $G_R\coloneqq\set*{(x,y)\in \Rddiag: \abs{x-y}\ge1/R}$ and choose $\bar\chi_R\in C_c^\infty(\Rddiag;[0,1])$ such that $\bar\chi_R|_{G_R\cap(B_R\times B_R)}\equiv 1$ and $\supp\bar\chi_R\subset\overline{G_{2R}\cap(B_{2R}\times B_{2R})}$. Further, we choose $\chi_R\in C_c^\infty(\Rd;[0,1])$ such that $\supp\chi_R\subset \overline B_{2R}$, $\chi_R|_{B_R} \equiv 1$ and $\abs{\nabla\chi_R} \le 2/R$. Arguing analogously to the derivation of \eqref{eq:D_cut-off}, we obtain
\begin{align*}
    \calD(\mu^n,\eta;\rho^n) &\ge \iint_\Rddiag\bra*{\bra*{\babla\varphi_R^n}_+}^2\bar\chi_R\eta\dx\bra*{\rho^n\otimes\mu^n}-\omega(R)\\
    &\eqqcolon \widetilde{\mathcal{A}}(\bar\chi_R(\mu^n\otimes\rho^n),\babla\varphi_R^n)-\omega(R),
\end{align*}
where $\varphi_R^n\coloneqq -K\ast\chi_R\rho^n$. Setting $\varphi_R\coloneqq -K\ast\chi_R\rho$, for any $R>0$, we have
\begin{align*}
    &\abs[\big]{\widetilde{\mathcal{A}}(\bar\chi_R(\mu^n\!\otimes\!\rho^n),\babla\varphi_R^n)-\widetilde{\mathcal{A}}(\bar\chi_R(\mu^n\!\otimes\!\rho^n),\babla\varphi_R)}\le 8R L_K C_\eta^\mu \norm{\varphi^n_R-\varphi_R}_{L^\infty(B_{2R})},
\end{align*}
where we used that $\abs{x-y}>1/2R$ for any $(x,y)\in \supp\bar\chi_R$ and~\eqref{eq:2lor4-int_bound}. Arguing as in the proof of \eqref{eq:norm_phi^n_R->0} we find $\norm{\varphi^n_R-\varphi_R}_{L^\infty(B_{2R})}$ vanishes as $n\to\infty$. Due to the cut-off $\bar\chi_R$, the weak-$^\ast$ convergence of $\rho^n\otimes\mu^n$ towards $\rho\otimes\mu$ yields
\begin{align*}
    \lim_{n\to\infty}\widetilde{\mathcal{A}}(\bar\chi_R(\mu^n\otimes\rho^n),\babla\varphi_R) = \widetilde{\mathcal{A}}(\bar\chi_R(\mu\otimes\rho),\babla\varphi_R).
\end{align*}
Furthermore, denoting $\varphi\coloneqq -K\ast\rho$, employing again \eqref{K3} as well as \eqref{eq:2lor4-int_bound} and arguing similar to the estimate \eqref{eq:bound_K-ast-rho}, for every $R>0$ we find
\begin{align*}
    \abs*{\widetilde{\mathcal{A}}(\bar\chi_R(\mu\otimes\rho),\babla\varphi_R)-\widetilde{\mathcal{A}}(\bar\chi_R(\mu\otimes\rho),\babla\varphi)} &\le 2 L_K^2 C_\eta^\mu \rho(B_{R}^c).
\end{align*}
For the outer cut-off, we employ again \eqref{K3} and \eqref{eq:2lor4-int_bound}
\begin{align*}    
    &\abs[\big]{\widetilde{\mathcal{A}}(\bar\chi_R(\mu\otimes\rho),\babla\varphi)-\widetilde{\mathcal{A}}(\mu;\rho,\babla\varphi)}\\
    &\le \abs[\bigg]{\iint_{G_R^c}\bra*{\bra*{\babla\varphi}_+}^2\eta\dx\bra*{\rho\otimes\mu}} + \abs[\bigg]{\iint_{B_R^c\times B_R^c}\bra*{\bra*{\babla\varphi}_+}^2\eta\dx\bra*{\rho\otimes\mu}}\\
    &\le L_K^2 \sup_{x\in\Rd}\int_{B_{R^{-1}}(x)\setminus\set{x}}\bra*{\abs{x-y}^2\lor\abs{x-y}^4}\eta(x,y)\dx\mu(y) + L_K^2 C_\eta^\mu \rho(B_R^c),
\end{align*}
which all vanish as $R\to\infty$ due to the tightness of $\rho$ and \eqref{eq:blow-up_control}. Combining all the previous estimates, we obtain
\begin{align*}
    \liminf_{n\to\infty}\calD(\mu^n,\eta;\rho^n) 
    &\ge \liminf_{n\to\infty}\widetilde{\mathcal{A}}(\bar\chi_R(\mu^n\otimes\rho^n),\babla\varphi_R^n)-\omega(R)\\
    &\ge \liminf_{n\to\infty}\widetilde{\mathcal{A}}(\bar\chi_R(\mu^n\otimes\rho^n),\babla\varphi_R)-\omega(R)\\
    &= \widetilde{\mathcal{A}}(\bar\chi_R(\mu\otimes\rho),\babla\varphi_R)-\omega(R)\\
    &\overset{R\to\infty}{\longrightarrow} \widetilde{\mathcal{A}}(\mu;\rho,\babla\varphi) = \calD(\mu,\eta;\rho). \qedhere
\end{align*}
\end{proof}

The final step is adapting \cite[Theorem 3.15]{EPSS2021} to the altered framework. 
In order to obtain an $n$-uniform bound similar to \cite[Eq.~(3.31)]{EPSS2021}, we will use a finite approximation.
\begin{lemma}[Finite approximation]\label{lem:counting-measures-appendix}
    Let $\mu\in\calM^+(\Rd)$ such that \eqref{mu1}, \eqref{mu2} hold. For every $m\in\bbN$ define $\bar\mu^m\coloneqq \mu|_{\overline{B_m}}$. There exists a sequence $(\bar\mu^{m,n})_{n\in\bbN}$ given by
\begin{align*}
    \bar\mu^{m,n} = \sum_{k=1}^{N^m_n} \bar\mu^{m,n}_k \delta_{x^{m,n}_k},
\end{align*}
for suitable $N^m_n\in\bbN$, $\bar\mu^{m,n}_k\in [0,\infty)$, $x^{m,n}_k\in\Rd$, such that $\supp\bar\mu^{m,n}\subset \supp\bar\mu^m=\overline{B_m}$ for all $n\in\bbN$ and $\bar\mu^{m,n}\to\bar\mu^m$ with respect to $W_2$ as $n\to\infty$. These measures have equal total mass, i.e. $\bar\mu^m(\Rd)=\bar\mu^{m,n}(\Rd)$ for all $n\in\bbN$. Furthermore, there is a constant $\bar C_\mu>0$ such that
\begin{equation}\label{eq:n-uniform_mass-bound_mu_app}
\begin{aligned}
    \sup_{m,n\in\bbN}\sup_{x\in \Rd}\bar\mu^{m,n}(G_{\!\!\:x})
    \le \bar C_\mu \quad \mbox{and}\quad \sup_{m\in\bbN}\sup_{x\in \Rd}\bar\mu^m(G_{\!\!\:x})
    \le \bar C_\mu,
\end{aligned}
\end{equation}
where we recall that $G_{\!\!\:x}=\set{y\in\Rd\!\setminus\!\set*{x}:\eta(x,y)>0}$. In particular, $(\bar\mu^{m,n},\eta)$ and $(\bar\mu^{m},\eta)$ satisfy \eqref{eq:2lor4-int_bound} and \eqref{eq:blow-up_control} uniformly in $m$ and $n$.
\end{lemma}
\begin{proof}
    A possible approximation is the $d$-dimensional midpoint Riemann sum, which is obtained by choosing the evaluation points $\set{x_k^{m,n}:1\le k \le N^m_n} = \bbZ^d/2^n\cap \overline{B_m}$ and the weights $\bar\mu^{m,n}_k = \bar\mu(x_k^{m,n})/2^{dn}$, where we recall that $\tilde\mu$ is the density of $\mu$ with respect to $\scrL^d$. Since the ball $\overline{B_m}$ is compact, $W_2$ convergence holds as consequence of weak-$^\ast$ convergence and its equivalence with the narrow convergence, plus that of second order moments. The equal total mass is ensured by a standard rescaling argument. Regarding the inequality \eqref{eq:n-uniform_mass-bound_mu_app}, we notice that $G_0 = \set{y\in\Rdzero:\eta(0,y)>0}\subset\Rd$ is a bounded set, hence so is  $G_0^R\coloneqq G_0+B_R$ for any $R>0$. Therefore $\bar\mu^{m}(G_0^R)$ and $\bar\mu^{m,n}(G_0^R)$ are bounded uniformly with respect to $m$ and $n$, by construction of the measures $\tilde\mu^{m,n}$. Moreover, we notice that
     \begin{align*}
	\sup_{x\in\Rd}\diam(G_{\!\!\:x})= \sup_{x\in\Rd}\diam\bra{\set{y\in\Rd\!\setminus\!\set*{x}:\eta(x,y)>0}} <\infty,
    \end{align*}
     so that, by the construction of the measures $\tilde\mu^{m,n}$, there is an $R>0$ such that for any $x\in\Rd$ and any $m,n\in\bbN$ we have $\tilde\mu^{m,n}(G_{\!\!\:x}) \le \tilde\mu^{m,n}(G_0^R)$. Here, $R>0$ compensates for the fact that more points $x_k^{m,n}$ might lie inside $G_x$ than inside $G_0$ and for fluctuations in $\mu$, keeping in mind, the uniform bounds $c_\mu \le \tilde\mu \le C_\mu$. 
\end{proof}

Now we are in the position to show existence of solutions to \eqref{eq:nlnl-interaction-eq} in the adapted setting.
\begin{theorem}[Existence of weak solutions to \eqref{eq:nlnl-interaction-eq}]\label{thm:extension_sigma_finite}
    Let $K$ satisfy \eqref{K1}~--~\eqref{K3}, $\mu\in\calM^+(\Rd)$ be $\sigma$-finite, and assume that $(\mu,\eta)$ satisfy \eqref{eq:eta_cont_sym}~--~\eqref{eq:blow-up_control}. Consider $\varrho_0\in\calP_2(\Rd)$ with $\varrho_0\ll\mu$. Then there exists a narrowly continuous curve $\bs\rho:[0,T]\to\calP_2(\Rd)$ such that $\rho_t\ll\mu$ for all $t\in[0,T]$, which is a weak solution to \eqref{eq:nlnl-interaction-eq} with initial datum $\rho_0=\varrho_0$.
\end{theorem}
\begin{proof}
By Lemma~\ref{lem:counting-measures-appendix} we obtain for every $m\in\bbN$ the finite measure $\bar\mu^{m}$ with bounded second moments, as well as finite combinations of Dirac measures $\bar\mu^{m,n}$ with $\bar\mu^{m,n}(\Rd) =\bar\mu^{m}(\Rd)=\mu({B_m})$ that converge to $\bar\mu^{m}$ in $W_2$ as $n\to\infty$. Moreover, they satisfy \eqref{eq:2lor4-int_bound} and \eqref{eq:blow-up_control} uniformly in $m$ and $n$. Let $m_0\in\bbN$ such that $\varrho_0(B_m) > 0$ for all $m\ge m_0$. For $m\ge m_0$ the truncated and renormalized initial data $\bar\varrho_0^m \coloneqq \varrho_0|_{B_m}/\varrho_0(B_m)\in\calP(\Rd)$ satisfying $\bar\varrho^m_0 \ll \bar\mu^m$ for every $m\in\bbN$. They converge to $\varrho_0$ in $(\calP_2(\Rd),W_2)$, since $\varrho_0\in\calP_2(\Rd)$ by assumption.
The above objects allow us for each $m,n\in\bbN$, $m\ge m_0$ to employ the construction in the proof of \cite[Theorem 3.15]{EPSS2021} to obtain initial data $\bar\varrho^{m,n}_0 \ll \bar\mu^{m,n}$ which converge to $\bar\varrho^m_0$ as $n\to\infty$ in $(\calP(\Rd),W_2)$. Therefore, for any $k\in\bbN$ there exist $m,n\in\bbN$ such that
\begin{align*}
W_2(\bar\varrho^{m,n}_0, \bar\varrho^m_0) + W_2(\bar\varrho^m_0, \bar\varrho_0)\le \frac1k.
\end{align*}
Denote $\varrho^k_0 \coloneqq \bar\varrho^{m,n}_0$ for this $m,n$. These initial data satisfy all the conditions required to follow the remainder of the proof of \cite[Theorem 3.15]{EPSS2021} and obtain the result.
\end{proof}

\section{Relation to EDP convergence for gradient structures}

In this part of the appendix, we translate the gradient flow formulation in terms of curves of maximal slope for the quasi-metric provided in Section~\ref{subsec:nlnl_interaction_eq} into the recent notion of gradient flows in continuity equation format~\cite[Definition 1.1]{PeletierSchlichting2022}.
The starting point is an abstract continuity equation including the conservation laws of the systems under consideration; both the nonlocal~\eqref{eq:intro:NL-CE} and local continuity equation~\eqref{eq:intro:CE} represent indeed examples of that formalism.

The flux formulation has the advantage that the kinetic relations~\eqref{eq:intro:NL-flux} and~\eqref{eq:intro:fluxtensor} can be encoded as the subdifferential of a suitable convex functional. 
For~\eqref{eq:intro:NL-flux} those functionals can be formally defined by (see Definition~\ref{def:action} for how undefined cases are handled)
\begin{align}
	\calR(\rho,j) &= \iint_G \frac{1}{4}\abs*{\bra*{\pderiv{j}{\rho\otimes \mu}}_+}^2 \eta  \dx(\rho\otimes \mu) + \iint_G \frac{1}{4}\abs*{\bra*{\pderiv{j}{\mu\otimes \rho}}_-}^2 \eta \dx(\mu\otimes \rho), \label{eq:intro:R}\\
	\calR^*(\rho,v) &= \iint_G\frac{\abs*{ v_+ }^2}{4} \eta  \dx(\rho\otimes \mu) + \iint_G \frac{\abs*{ v_- }^2}{4} \eta  \dx(\mu\otimes \rho) . \label{eq:intro:R*}
\end{align}
Hereby, the formal duality is understood in the $\eta$-weighted dual product $\skp*{v,j}_{\!\eta} = \frac{1}{2} \iint_G v \, \eta \dx j$ and it is easy to verify (see Remark~\ref{rem:duality}) that
\begin{equation*}
	\calR^*(\rho,v) = \sup_{j\in \calM(G)} \bra*{ \skp{v,j}_{\!\!\:\eta} - \calR(\rho,j)} . 
\end{equation*}
Hence, we have always the bounds 
\begin{equation}\label{eq:LegFen:bound}
\skp{v,j}_{\!\!\:\eta} \leq \calR(\rho,j) + \calR^*(\rho,v). 
\end{equation}

 \begin{remark}[Duality]\label{rem:duality}
	     Given $r,s\in [0,\infty]$, define the function $f:\R\to[0,\infty]$ by
	     \begin{align*}
		         f(j) &\coloneqq \frac12\pra*{\alpha(j,r) + \alpha(-j,s)}, 
		 \end{align*}
		  with $\alpha:\R\times[0,\infty]\to[0,\infty]$ defined in~\eqref{eq:def:alpha}.
	     For any $r,s\in[0,\infty)$ the function $f$ is proper, lsc. and convex, so that by the Fenchel-Moreau theorem $f$ is its own convex biconjugate. The convex conjugate of~$f$ is given by
	     \begin{align*}
		         f^\ast(v) = \frac12\pra*{r(v_+)^2+s(v_-)^2}.
		     \end{align*}
	   To see this, at first we assume that $r,s > 0$ and calculate
	     \begin{align*}
		         f^\ast(v) &= \sup_{j\in\R} \pra*{v\cdot j - f(j)}= \sup_{j\in\R} \pra*{v\cdot j - \frac{(j_+)^2}{2r}-\frac{(j_-)^2}{2s}}
		         = \frac{1}{2}\pra*{r(v_+)^2+s(v_-)^2}.
		     \end{align*}
	     By the definition of $\alpha$, the cases $r=0$ and $s=0$ satisfy the same equality, since the supremum is then achieved at $j_+=0$ or $j_-=0$ respectively.
	 \end{remark}
This provides a robust formulation thanks to the characterization of the subdifferential through \emph{Legendre-Fenchel duality}: any pair $(v,j)\in C_0(G)\times \calM(G)$ satisfies the identity~\eqref{eq:intro:NL-flux} if and only if
\begin{equation}\label{eq:LegFen:sub}
	j \in \partial_2 \calR^*(\rho,v) 
	\quad\Leftrightarrow\quad
	v \in \partial_2 \calR(\rho,j)
	\quad\Leftrightarrow
	\skp*{v,j}_{\!\eta} = \calR(\rho,j) + \calR^*(\rho,v),
\end{equation}
where $\partial_2 \calR$ and $\partial_2 \calR$ denote the subdifferential for the second argument, which for~\eqref{eq:intro:R} and~\eqref{eq:intro:R*} are single-valued. 
We note that $\calR(\rho,j) = \frac{1}{2} \calA(\mu,\eta;\rho,j)$ from Definition~\ref{def:action} and we have thanks to Corollary~\ref{cor:A-bound} the improved integrability of the flux whenever $\calR(\rho,j)<\infty$ of the type
\begin{equation*}
	\frac12 \iint \abs{x-y}\eta(x,y) \dx j(x,y) \leq \sqrt{2 \Cint \calR(\rho,j)} <\infty.
\end{equation*}
In particular, this allows to extend the test-function class in~\eqref{eq:LegFen:bound} to nonlocal gradients of the type $\babla \varphi$ with $\varphi \in C^1_b(\Rd)$, since we can estimate
\begin{align*}
\abs*{\skp{ \babla \varphi, j}_\eta} &\leq \frac{1}{2}\iint_G \frac{\abs{\varphi(y)-\varphi(x)}}{\abs{x-y}} \abs{x-y} \eta(x,y)\dx j(x,y) \\
&\leq \norm{\varphi}_{C^1(\Rd)} \frac12 \iint \abs{x-y}\eta(x,y) \dx j(x,y) < \infty.
\end{align*}
With this observation, we find in equation \eqref{eq:intro:NL-CE} another duality structure between the nonlocal divergence $\bdiv$ and the nonlocal gradient $\babla$. Both satisfy, for any test function $\varphi\in C_b^1(\R^d)$, the identity
$\skp[\big]{\varphi, \bdiv j} = -\skp[\big]{\babla \varphi, j}_{\!\eta}$. 

Let us suppose that the variational derivative of the energy $\calE'(\rho) \in C^1_b(\Rd)$, which is usually not the case. Then, we can connect the continuity equation~\eqref{eq:intro:NL-CE} with the specific choice for the velocity in~\eqref{eq:intro:NL-velocity} as gradient of the derivative $v_t = -\babla \calE'(\rho_t)$ via the identity
\begin{align}
	\pderiv{\calE(\rho_t)}{t} &= \skp{\calE'(\rho_t), \partial_t \rho_t} = -\skp{\calE'(\rho_t), \bdiv j_t} = - \skp{-\babla \calE'(\rho_t),j_t}_\eta \notag \\
	&\geq  - \calR(\rho_t,j_t) - \calR^*(\rho_t,-\babla \calE'(\rho_t)), \label{eq:LegFen:Energy}
\end{align}
where equality holds if and only if $j_t$ and $v_t$ are related through~\eqref{eq:LegFen:sub} and hence equivalently by~\eqref{eq:intro:NL-flux}. 
By integrating the estimate~\eqref{eq:LegFen:Energy}, we obtain an energy-dissipation functional defined for solutions $(\bs\rho,\bs j)\in \NCE_T$ to~\eqref{eq:intro:NL-CE} by
\begin{equation*}%
	\bs\calL(\bs\rho,\bs j) = \calE(\rho_T) - \calE(\rho_0) + \int_0^T \bra*{ \calR(\rho_t,j_t) + \calR^*(\rho_t,-\babla \calE'(\rho_t)) } \dx{t} . 
\end{equation*}
This is connected to the De Giorgi functional $\bs\calG$ by the relation
\begin{align*}
    \bs\calG(\bs\rho) = \inf \set[\big]{ \bs\calL(\bs\rho,\bs j) : \bs j\text{ such that }(\bs\rho,\bs j)\in\NCE_T }.
\end{align*}
The previous considerations lead to the following definition of a nonlocal gradient system in continuity equation format in our setting.
\begin{definition}[Nonlocal gradient structure]%
	A \emph{nonlocal gradient structure} in continuity equation format has the building blocks:
	\begin{enumerate}
		\item A graph structure induced by $(\mu,\eta)$ defining the nonlocal divergence $\bdiv$ (implicity depending on $\eta$) and providing a notion of solutions for $(\bs\rho,\bs j)\in\NCE_T$ given in~\eqref{eq:intro:NL-CE} through the duality product $C_b^1(\R^d)\times \calM(\R^d) \ni (\varphi,j)\mapsto \skp{\babla \varphi,j}_\eta$. 
		\item An \emph{energy functional} $\calE:\calP(\R^d)\to [0,\infty)$. 
		\item A \emph{dissipation functional} $\calR: \calP(\R^d)\times \calM(G)\to \R$: For any $\rho\in \calP(\R^d)$, the map $\calR(\rho,\cdot)$ is convex and lower semicontinuous with $\min \calR(\rho,\cdot) = \calR(\rho,0)=0$.
	\end{enumerate} 
	The gradient structure $((\mu,\eta),\babla,\calE,\calR)$ is called \emph{good}, provided that for any $(\bs\rho,\bs j)\in \NCE_T$, one has $\bs\calL(\bs\rho,\bs j)\geq 0$.
	
	A curve $(\bs\rho,\bs j)\in \NCE_T$ is an \emph{EDP solution} of the good gradient structure $((\mu,\eta),\babla,\calE,\calR)$ provided that $\bs\calL(\bs\rho,\bs j)= 0$.
\end{definition}
Since, in general the variational derivative of the driving energy is not in the class of admissible test functions, in our case $\calE'(\rho)\notin C_b^1(\Rd)$, the goodness of the energy-dissipation functional $\bs\calL$ has to be proven. This is typically done with the help of establishing the \emph{chain rule inequality}
\begin{equation*}%
		\calE(\rho_t)- \calE(\rho_0) \ge \int_0^t \skp[\big]{ \babla \calE'(\rho_s),j_s}_{\!\eta} \dx{s}  \qquad \text{a.e. } t\in(0,T).
\end{equation*}
One of the main results of~\cite{EPSS2021} can be restated as follows:
\begin{theorem}[{Variational characterization of NL$^2$IE~\cite{EPSS2021}}]
	Any measure-valued solution of~\eqref{eq:intro:NLNL} is an EDP~solution for the gradient structure $((\mu,\eta),\babla,\calE,\calR)$, with elements given by~\eqref{eq:intro:NL-CE}, \eqref{eq:interaction_energy} and~\eqref{eq:intro:R}.
\end{theorem}
Having established the variational formulation for~\eqref{eq:intro:NLNL}, we can ask about the variational convergence for the $\eps$-rescaled version, where now $\eta$ is replaced by $\eta^\eps$ from~\eqref{eq:intro:etaeps} and hence, we study the gradient structure $((\mu,\eta^\eps),\bdiv_{\!\eps},\calE,\calR_\eps)$ and ask about its limit as $\eps \to 0$. From the heuristics in the introduction, we expect a limit of the form $(\R^d,\div, \calE,\calR_{\bbT})$, where the graph structure is replaced by a $\bbT$-weighted Euclidean structure with respect to the limiting tensor in~\eqref{eq_def:Tensor} and the nonlocal divergence by the standard one $\div j = \sum_{i=1}^d \partial_i j_i$. The limiting dissipation potential defines a dynamic dissipation after Otto and co-authors~\cite{JKO1998,Otto2001GeometryPME} giving rise to the Wasserstein distance~\cite{BenamouBrenier2000}, which in its $\bbT$-weighted form was studied in~\cite{Lisini_ESAIM2009} and is given by
\begin{equation*}%
	\calR_{\bbT}(\rho,j) =  \int_\Rd \frac{1}{2} \abs[\Big]{ \pderiv{j}{\rho}(x)}^2_{\bbT(x)} \dx{\rho(x)}, 
\end{equation*}
where the weighted Euclidean norm is given by $\abs*{\xi}_{\bbT}^2 = \xi\cdot \bbT^{-1}\xi$ for $\xi\in \R^d$. The limiting gradient structure $(\R^d,\div, \calE,\calR_{\bbT})$ is good thanks to the chain rule proven in Proposition~\ref{prop:chain_rule_ineq}, implying that for $(\bs\rho,\bs j)\in \CE_T$ the ED functional $	\bs\calL_\bbT(\bs \rho,\bs j)$ given by
\begin{equation*}
	\bs\calL_\bbT(\bs \rho,\bs j) = \calE(\rho_T) - \calE(\rho_0) + \int_0^T \bra*{ \calR_\bbT(\rho_t,j_t) + \calR_\bbT^*(\rho_t, - \nabla \calE'(\rho_t)) }\dx{t} 
\end{equation*}
is non-negative, i.e. $\bs\calL_\bbT(\bs \rho,\bs j)\geq 0$, with the dual dissipation functional $\calR_\bbT^*(\rho,\xi) = \int \frac{1}{2} \skp{\xi(x), \bbT(x)\xi} \dx{\rho(x)}$.
Hence, any curve $(\bs\rho,\bs j)\in \CE_T$ with $\bs\calL_\bbT(\bs \rho,\bs j)=0$ is a measure-valued solution to~\eqref{eq:intro:NLIE:one}. 

For the EDP convergence statement, we need a common notion of \emph{curves} for $\eps>0$ and the limit. This is possible thanks to the reconstruction of the flux in Proposition~\ref{prop:jhat}, where we showed that any $(\bs\rho^\eps,\bs j^\eps)\in \NCE^\eps_T$ can be associated with a curve $(\bs\rho^\eps,\bs\jh^\eps)\in \CE_T$. In this way, we arrive at a notion of convergence, which we call $\tau$-convergence
\begin{equation*}%
	\NCE^\eps_T \ni (\bs\rho^\eps,\bs j^\eps)  \xrightharpoonup{\tau} (\bs\rho,\bs j) \in \CE_T 
\end{equation*}
provided that $\rho^\eps_t \xrightharpoonup{} \rho_t$ narrowly in $\calP(\R^d)$ for a.e. $t\in[0, T]$ and $\int_\cdot\jh^\eps_t\dx{t} \xrightharpoonup{\ast} \int_\cdot\jh\dx{t}$ weakly-$^\ast$ in $\calM((0,T)\times\Rd;\Rd)$.

With this preliminary considerations, the main result of this work contained in Theorem~\ref{thm:main_result} can be recast in terms of \emph{EDP-convergence}, or also called \emph{evolutionary $\Gamma$-convergence}~\cite{SandierSerfaty2004,Serfaty2011,Mielke2016,LieroMielkePeletierRenger2017,MielkeMontefuscoPeletier21,PeletierSchlichting2022}.
The EDP convergence statement is formulated in terms of the total dissipation functional of a curve $(\bs\rho^\eps,\bs j^\eps)\in \NCE^\eps_T$ defined by
\begin{equation*}
	\scrD_\eps(\bs\rho^\eps,\bs j^\eps) = \int_0^T \bra*{ \calR_\eps(\rho^\eps_t,j^\eps_t) + \calR_\eps^*(\rho^\eps_t,-\babla \calE'(\rho^\eps_t)) } \dx{t} .
\end{equation*}
Now, we can restate the Theorem~\ref{thm:main_result} in the language of EDP convergence. Typically, the energy $\calE$ is dependent of $\eps$, and the EDP convergence statement contains a $\Gamma$-limit statement of the type $\calE_\eps \xrightarrow{\Gamma} \calE_0$. In our case this is not necessary, since the driving energy $\calE$ does not depend on $\eps$ and is continuous.

\begin{theorem}[EDP convergence of \eqref{eq:nlnl-interaction-eq-eps} to~\eqref{eq:intro:NLIE:one}]%
    Let $(\mu,\vartheta)$ satisfy \eqref{mu1}, \eqref{mu2} and \eqref{theta1}~--~\eqref{theta4}. Let $\eta^\eps$ be given by \eqref{eq:def:eta^eps} and assume $K$ satisfies \eqref{K1}~--~\eqref{K4}. 
	Then, the sequence of gradient structures $((\mu,\eta^\eps),\bdiv_{\!\eps},\calE,\calR_\eps)$ for~\eqref{eq:nlnl-interaction-eq-eps} \emph{EDP converges} to the limiting gradient structure $(\R^d,\div, \calE,\calR_{\bbT})$ for~\eqref{eq:intro:NLIE:one} as $\eps\to 0$. 
	
	Let $(\bs\rho^\eps,\bs j^\eps)\in \NCE^\eps_T$ with $\sup_{0<\eps\le\eps_0} \scrD_\eps(\bs\rho^\eps,\bs j^\eps) < \infty$, then there exists a subsequence such that $\NCE^\eps_T \ni (\bs\rho^\eps, \bs j^\eps) \xrightharpoonup{\tau} (\bs\rho ,\bs \jh) \in \CE_T$ and
	\begin{equation*}
		\liminf_{\eps\to 0} \scrD_\eps(\bs\rho^\eps,\bs j^\eps) \geq \scrD_\bbT(\bs\rho,\bs \jh),
	\end{equation*}
	where the limiting total dissipation function is defined by
	\begin{equation*}%
		\scrD_\bbT(\bs \rho,\bs\jh) =
			\int_0^T \Bigl[ \calR_\bbT(\rho_t,\jh_t)+\calR^*_\bbT(\rho_t,-\nabla \calE'(\rho_t))\Bigr]\dx t ,	
	\end{equation*}
	provided $(\bs\rho ,\bs \jh) \in \CE_T$.
\end{theorem}
\begin{proof}
	The compactness and convergence of solutions to the continuity equation is contained in Proposition~\ref{prop:convergence_CE-solutions}. The lower semicontinuity of $\scrD_\eps$ is a consequence of the individual lower semicontinuity statements for $\calR_\eps$ and $\calR_\eps^*$, which follow from Proposition~\ref{prop:lower_limit_metric_derivatives} and Proposition~\ref{prop:D_NL->L_lsc}, respectively.
\end{proof}

\noindent
\textbf{Acknowledgements. }
The authors are grateful to Dejan Slep\v{c}ev and Francesco Patacchini for enlightening discussions on the original question, posed in a previous manuscript. Furthermore, the authors would like to thank José Antonio Carrillo for his valuable suggestions on the content of the current manuscript and for hosting GH and AS in Oxford. AE was supported by the Advanced Grant Nonlocal-CPD (Nonlocal PDEs for Complex Particle Dynamics: Phase Transitions, Patterns and Synchronization) of the European Research Council Executive Agency (ERC) under the European Union’s Horizon 2020 research and innovation programme (grant agreement No. 883363). GH acknowledges support of the German National Academic Foundation (Studienstiftung des deutschen Volkes) and the Free State of Saxony in the form of PhD scholarships. Part of this work was carried out while GH was a PhD student at TU Chemnitz and the University of Augsburg under supervision of Jan-Frederik Pietschmann. AS is supported by the Deutsche Forschungsgemeinschaft (DFG, German Research Foundation) under Germany's Excellence Strategy EXC 2044 -- 390685587, \emph{Mathematics M\"unster: Dynamics--Geometry--Structure}.

\bibliography{bibliography}
\bibliographystyle{abbrv}

\end{document}

%% file: header.tex
\usepackage[dvipsnames]{xcolor}
\usepackage{latexsym,amssymb}
\usepackage{amsmath,amsthm,amsxtra,amsbsy,accents,mathrsfs}
\usepackage{mathtools}
\usepackage{lmodern,microtype}
\usepackage{dsfont}

\usepackage{newtxtext} %
\usepackage{newtxmath} %

\usepackage{enumerate}
\usepackage{enumitem}

\usepackage{comment}

\usepackage[bookmarks=true, colorlinks=true]{hyperref}
\hypersetup{urlcolor=blue,citecolor=black,linkcolor=black}

\usepackage{aligned-overset}

\DeclareMathAlphabet{\mathbbmsl}{U}{bbm}{m}{sl}

\DeclareMathAlphabet{\mathup}{OT1}{\familydefault}{m}{n}
\newcommand{\dx}[1]{\mathop{}\!\mathup{d} #1}
\newcommand{\pderiv}[3][]{\frac{\mathop{}\!\mathup{d}^{#1} #2}{\mathop{}\!\mathup{d} #3^{#1}}}
\DeclarePairedDelimiter{\abs}{\lvert}{\rvert}
\DeclarePairedDelimiter{\norm}{\lVert}{\rVert}
\DeclarePairedDelimiter{\bra}{(}{)}
\DeclarePairedDelimiter{\pra}{[}{]}
\DeclarePairedDelimiter{\set}{\{}{\}}
\DeclarePairedDelimiter{\skp}{\langle}{\rangle}
\makeatletter
\newcommand{\customlabel}[2]{%
   \protected@write \@auxout {}{\string \newlabel {#1}{{#2}{\thepage}{#2}{#1}{}} }%
   \hypertarget{#1}{}
}
\makeatother

\makeatletter
\newcommand{\oset}[3][0ex]{%
  \mathrel{\mathop{#3}\limits^{
    \vbox to#1{\kern-2\ex@
    \hbox{$\scriptstyle#2$}\vss}}}}
\makeatother
\numberwithin{figure}{section}

\numberwithin{equation}{section}

\def\bbone{{\mathds{1}}}

\newtheorem{theorem}{Theorem}[section]
\newtheorem{lemma}[theorem]{Lemma}
\newtheorem{corollary}[theorem]{Corollary}
\newtheorem{definition}[theorem]{Definition}

\newtheorem{proposition}[theorem]{Proposition}
\newtheorem{remark}[theorem]{Remark}

\newcommand{\bs}[1]{{\boldsymbol{#1}}}

\newcommand{\R}{{\mathds{R}}}
\newcommand{\Rd}{{\R^d}}
\newcommand{\Rdzero}{{\Rd\setminus\set*{0}}}
\newcommand{\Rdx}{{\Rd\setminus\set*{x}}}
\newcommand{\Rddiag}{{\R^{2d}_{\!\scriptscriptstyle\diagup}}}

\newcommand{\jh}{{\hat{\jmath}}}
\newcommand{\babla}{\overline{\nabla}}

\newcommand{\dgrad}{\overline{\nabla}}

\renewcommand{\div}{\operatorname{div}}
\newcommand{\bdiv}[1]{\mathop{\overline{\mathrm{div}}}\nolimits#1}
\newcommand{\bdiveps}[1]{\mathop{\overline{\mathrm{div}}_{\!\eps}}\nolimits#1}

\newcommand{\eps}{\varepsilon}

\DeclareMathOperator{\supp}{supp}

\DeclareMathOperator{\CE}{CE}
\DeclareMathOperator{\NCE}{NCE}
\DeclareMathOperator{\AC}{AC}

\DeclareMathOperator{\diam}{diam}

\DeclareMathOperator{\Id}{Id}

\DeclareMathOperator{\I}{I}
\DeclareMathOperator{\II}{II}

\newcommand{\A}{{\mathcal{A}}}

\newcommand{\tA}{{\widetilde{\mathcal{A}}}}

\newcommand{\Ctheta}{{C_{\mathsf{mom}}}}
\newcommand{\ctheta}{{c_{\mathsf{nd}}}}

\newcommand{\Csupp}{{C_{\mathsf{supp}}}}
\newcommand{\Cmeas}{{C_{\mathsf{meas}}}}

\newcommand{\Cmom}{{C_{\mathsf{mom}}}}
\newcommand{\Cint}{{C_{\mathsf{int}}}}

\newcommand{\Cmu}{{C_{\mu}}}

\newcommand{\Mplus}{\calM^+}
\newcommand{\Mdiv}{\calM_{\mathsf{div}}}

\makeatletter
\newcommand{\anabla}{{\mathpalette\a@nabla\relax}}
\newcommand\a@nabla[2]{%
	\setbox\z@=\hbox{$\m@th#1\bigtriangledown$}%
	\ht\z@.7\ht\z@
	\raise\dp\z@\box\z@
}
\makeatother

\makeatletter
\newcommand{\subalign}[1]{%
  \vcenter{%
    \Let@ \restore@math@cr \default@tag
    \baselineskip\fontdimen10 \scriptfont\tw@
    \advance\baselineskip\fontdimen12 \scriptfont\tw@
    \lineskip\thr@@\fontdimen8 \scriptfont\thr@@
    \lineskiplimit\lineskip
    \ialign{\hfil$\m@th\scriptstyle##$&$\m@th\scriptstyle{}##$\hfil\crcr
      #1\crcr
    }%
  }%
}
\makeatother

\def\calA{{\mathcal A}} \def\calB{{\mathcal B}} 
\def\calD{{\mathcal D}} \def\calE{{\mathcal E}} 
\def\calG{{\mathcal G}} \def\calH{{\mathcal H}} 
  \def\calL{{\mathcal L}}
\def\calM{{\mathcal M}}  \def\calO{{\mathcal O}}
\def\calP{{\mathcal P}}  \def\calR{{\mathcal R}}
 \def\calT{{\mathcal T}}

\def\scrD{{\mathscr  D}}  
  
  \def\scrL{{\mathscr  L}}

 \def\bbN{{\mathbb N}} 
  
 \def\bbT{{\mathbb T}} 
  
 \def\bbZ{{\mathbb Z}}

\usepackage[normalem]{ulem}  %